\documentclass[a4paper]{amsart}
\usepackage{amsmath,amsthm,amssymb,mathtools,mathrsfs}
\usepackage{bm}
\usepackage{tikz-cd}
\usepackage[all]{xy}
\usepackage{hyperref}
\usepackage{soul}
  \definecolor{urlcolor}{rgb}{0,0,0}
  \definecolor{linkcolor}{rgb}{.7,0.10,0.2}
  \definecolor{citecolor}{rgb}{.12,.54,.11}
\hypersetup{
   breaklinks=true, 
   colorlinks=true,
   urlcolor=urlcolor,
   linkcolor=linkcolor,
   citecolor=citecolor,
}
\usepackage{tikz-cd}

\usepackage{setspace}
\usepackage{xcolor}
\usepackage[shortlabels]{enumitem}
\usepackage[utf8]{inputenc}
\numberwithin{equation}{section}
\usepackage[top=2.3cm,
bottom=2.3cm,
left=2.3cm,
right=2.3cm]{geometry}

\usepackage[
style=alphabetic,
doi=false,
isbn=false,
url=false,
eprint=false,
block = space,
sorting=anyt,
maxnames=50]{biblatex}

\newtheorem{theorem}{Theorem}[section]
\newtheorem{corollary}[theorem]{Corollary}
\newtheorem{proposition}[theorem]{Proposition}
\newtheorem{proposition-definition}[theorem]{Proposition-Definition}
\newtheorem{lemma}[theorem]{Lemma}

\newtheorem{question}[theorem]{Question}
\newtheorem{conjecture}[theorem]{Conjecture}

\newtheorem{innercustomthm}{Assumption}
\newenvironment{assumption}[1]
 {\renewcommand\theinnercustomthm{#1}\innercustomthm}
 {\endinnercustomthm}
\theoremstyle{definition} 
\newtheorem{definition}[theorem]{Definition}

\newtheorem{theorem-definition}[theorem]{Theorem-Definition}

\theoremstyle{remark} 
\newtheorem{remark}[theorem]{Remark}

\newcommand{\longto}{\longrightarrow}

\newcommand{\into}{\hookrightarrow}

\newcommand{\abs}[1]{\lvert #1 \rvert}
\renewcommand{\bar}{\overline}

\renewcommand{\geq}{\geqslant}
\renewcommand{\leq}{\leqslant}
\renewcommand{\subset}{\subseteq}
\renewcommand{\setminus}{\smallsetminus}
\renewcommand{\tilde}{\widetilde}
\renewcommand{\vec}{\mathbf}

\newcommand{\BA}{\mathbb{A}}

\newcommand{\BC}{\mathbb{C}}
\newcommand{\BD}{\mathbb{D}}

\newcommand{\BG}{\mathbb{G}}

\newcommand{\BoA}{\mathbf{A}}
\newcommand{\BoC}{\mathbf{C}}
\newcommand{\BoF}{\mathbf{F}}

\newcommand{\BoL}{\mathbf{L}}
\newcommand{\BoN}{\mathbf{N}}
\newcommand{\BoQ}{\mathbf{Q}}

\newcommand{\BoZ}{\mathbf{Z}}

\newcommand{\ulBoQ}{\ul{\BoQ}}

\newcommand{\CA}{\mathcal{A}}
\newcommand{\CB}{\mathcal{B}}
\newcommand{\CC}{\mathcal{C}}
\newcommand{\CD}{\mathcal{D}}

\newcommand{\CF}{\mathcal{F}}

\newcommand{\CH}{\mathcal{H}}
\newcommand{\CI}{\mathcal{I}}

\newcommand{\CK}{\mathcal{K}}

\newcommand{\CM}{\mathcal{M}}
\newcommand{\CN}{\mathcal{N}}

\newcommand{\CT}{\mathcal{T}}
\newcommand{\CU}{\mathcal{U}}

\newcommand{\CX}{\mathcal{X}}

\newcommand{\Fg}{\mathfrak{g}}
\newcommand{\Fh}{\mathfrak{h}}
\newcommand{\Fn}{\mathfrak{n}}

\newcommand{\FM}{\mathfrak{M}}
\newcommand{\FN}{\mathfrak{N}}

\newcommand{\FX}{\mathfrak{X}}

\newcommand{\SA}{\mathscr{A}}

\newcommand{\SC}{\mathscr{C}}

\newcommand{\SF}{\mathscr{F}}
\newcommand{\SG}{\mathscr{G}}
\newcommand{\SH}{\mathscr{H}}

\newcommand{\SQ}{\mathscr{Q}}
\newcommand{\SR}{\mathscr{R}}

\newcommand{\Higgs}{\mathbf{Higgs}}

\DeclareMathOperator{\Tr}{Tr}
\DeclareMathOperator{\rat}{\mathbf{rat}}

\DeclareMathOperator{\BoMo}{BM}
\DeclareMathOperator{\ICA}{IH^*}

\DeclareMathOperator{\sstab}{sst}
\DeclareMathOperator{\B}{B\!}
\DeclareMathOperator{\stab}{s}
\DeclareMathOperator{\K}{K}
\let\ul\underline

\newcommand{\lazy}{\bm{e}}
\newcommand{\phip}[1]{{}^{\mathfrak{p}}\!\phi_{#1}}
\newcommand\mydots{\hbox to 0.8em{.\hss.\hss.}}

\newcommand{\ttBPS}{\mathtt{\aBPS}}

\newcommand{\rmabs}{\mathrm{abs}}
\newcommand{\ad}{\operatorname{ad}}
\newcommand{\Alg}{\mathrm{Alg}}
\newcommand{\arr}{\mathrm{ar}}
\DeclareMathOperator{\sAlg}{-Alg}
\DeclareMathOperator{\Aut}{Aut}
\DeclareMathOperator{\sLie}{-Lie}
\DeclareMathOperator{\snil}{-nil}

\newcommand{\rmBPS}{\mathrm{BPS}}
\newcommand{\D}{\mathcal{D}}

\newcommand{\ulBPS}{\underline{\BPS}}

\newcommand{\ch}{\mathrm{ch}}
\newcommand{\cl}{\mathrm{cl}}
\newcommand{\coker}{\operatorname{coker}}

\newcommand{\cone}{\mathrm{cone}}
\newcommand{\cusp}{\mathrm{cusp}}
\newcommand{\dd}{\mathbf{d}}
\newcommand{\ee}{\mathbf{e}}
\newcommand{\ff}{\mathbf{f}}
\newcommand{\mm}{\mathbf{m}}

\newcommand{\End}{\operatorname{End}}

\newcommand{\Ext}{\operatorname{Ext}}
\newcommand{\ext}{\operatorname{ext}}
\newcommand{\Free}{\operatorname{Free}}
\newcommand{\GL}{\mathrm{GL}}

\newcommand{\hyp}{\mathrm{hyp}}

\newcommand{\Hom}{\operatorname{Hom}}
\newcommand{\IC}{\mathcal{IC}}
\newcommand{\IP}{\mathrm{IP}}

\newcommand{\iso}{\mathrm{iso}}
\newcommand{\id}{\operatorname{id}}

\DeclareMathOperator{\JH}{\mathtt{JH}}

\newcommand{\Lie}{\mathrm{Lie}}

\newcommand{\MHM}{\mathrm{MHM}}
\newcommand{\MMHM}{\mathrm{MMHM}}

\newcommand{\Msp}{\CM} 
\newcommand{\Mst}{\FM} 
\newcommand{\mult}{\mathfrak{m}} 

\newcommand{\nil}{\mathrm{nil}}

\newcommand{\Perv}{\operatorname{Perv}}

\newcommand{\ptau}{{^\mathfrak{p}\tau}}
\newcommand{\QuivDim}{\mathrm{QuivDim}}

\newcommand{\relCoHA}{\mathscr{A}}
\newcommand{\ulrelCoHA}{\underline{\mathscr{A}}}

\newcommand{\uleqrelBPSalg}[2]{\underline{\BPS}^{#2}_{#1,\Alg}}

\newcommand{\Rep}{\operatorname{Rep}}

\newcommand{\real}{\mathrm{real}}
\newcommand{\im}{\mathrm{im}}

\newcommand{\PBW}{\mathrm{PBW}}
\DeclareMathOperator{\prim}{prim}
\DeclareMathOperator{\T}{T}
\newcommand{\red}{\mathrm{red}}

\newcommand{\cdotsh}{\!\cdot\!}

\newcommand{\RHom}{\operatorname{RHom}}
\newcommand{\rmb}{\mathrm{b}}
\newcommand{\td}{\mathtt{3d}}
\newcommand{\rmc}{\mathrm{c}}
\newcommand{\Spec}{\operatorname{Spec}}
\newcommand{\SSN}{\mathcal{SSN}}
\newcommand{\SemiN}{\mathcal{SN}}
\newcommand{\simp}{\mathrm{simp}}

\newcommand{\sm}{\mathrm{sm}}

\newcommand{\supp}{\operatorname{supp}}

\newcommand{\Tan}{\mathrm{T}}
\newcommand{\Tot}{\operatorname{Tot}}

\DeclareMathOperator{\UEA}{\mathbf{U}}

\newcommand{\Vect}{\mathrm{Vect}}
\newcommand{\vir}{\mathrm{vir}}

\DeclareMathOperator{\Exp}{Exp}
\DeclareMathOperator{\Id}{Id}
\DeclareMathOperator{\Hilb}{Hilb}

\newcommand{\MO}{\mathtt{MO}}
\newcommand{\BPS}{\mathcal{BPS}} 
\DeclareMathOperator{\aBPS}{BPS} 
\DeclareMathOperator{\Sym}{Sym}
\DeclareMathOperator{\DT}{DT}
\newcommand{\pt}{\mathrm{pt}}
\DeclareMathOperator{\HO}{H}
\newcommand{\cms}{/\!\!/}

\newcommand{\bs}[1]{\boldsymbol{#1}}

\newcommand{\bsE}{\bs{\mathfrak{E}}}

\title{BPS algebras and generalised Kac--Moody algebras from 2-Calabi--Yau categories}
\date{\today}

\author{Ben Davison}
\address{School of Mathematics, University of Edinburgh, Edinburgh, UK}
\email{Ben.Davison@ed.ac.uk}

\author{Lucien Hennecart}
\address{Laboratoire Ami\'enois de Math\'ematique Fondamentale et Appliqu\'ee, CNRS UMR 7352, Universit\'e de Picardie Jules Verne, 33 rue Saint Leu, 80000 Amiens, France}
\email{Lucien.Hennecart@u-picardie.fr}

\author{Sebastian Schlegel Mejia}
\address{Ecole Polytechnique Fédérale de Lausanne (EPFL), CH-1015 Lausanne, Switzerland}
\email{sebastian.schlegelmejia@epfl.ch}

\begin{document}

\begin{abstract}
We determine the structure of the BPS algebra of $2$-Calabi--Yau Abelian categories for which the stack of objects admits a good moduli space. We prove that this algebra is isomorphic to the positive part of the enveloping algebra of a generalised Kac--Moody Lie algebra generated by the intersection cohomology of certain connected components (corresponding to roots) of the good moduli space. Some major examples include the BPS algebras of (1) the category of semistable coherent sheaves of given slope on a K3 surface or, more generally, quasiprojective symplectic surface, (2) semistable Higgs bundles on a smooth projective curve, (3) preprojective algebras of quivers, (4) multiplicative preprojective algebras and (5) fundamental groups of (quiver) Riemann surfaces. We define the BPS Lie algebras of $2$-Calabi--Yau categories and prove that they coincide with the ones obtained by dimensional reduction from the critical cohomological Hall algebra in the case in which the 2-Calabi--Yau category is the category of representations of a preprojective algebra. Consequences include (1) A proof in full generality of the Bozec--Schiffmann positivity conjecture for absolutely cuspidal polynomials, a strengthening of the Kac positivity conjecture (2) A proof of the cohomological integrality conjecture for the category of semistable coherent sheaves on local K3 surfaces (3) A description of the cohomology (in all degrees) of Nakajima quiver varieties as direct sums of irreducible lowest weight representations over the BPS Lie algebra.
\end{abstract}

\maketitle

\setcounter{tocdepth}{1}

\tableofcontents

\section{Introduction}
\label{section:introduction}
In this paper, we study connections between generalised Kac--Moody (also known as Borcherds) Lie algebras \cite{kac1990infinite,borcherds1988generalized}, and the Borel--Moore homology of stacks of objects of $2$-Calabi--Yau (2CY) Abelian categories. 2CY categories feature prominently in geometric representation theory, mathematical physics, algebraic geometry and nonabelian Hodge theory \cite{hitchin1987self,donaldson1987twisted,corlette1988flat,simpson1992higgs}. Concrete examples of such categories include representations of various flavours of preprojective algebras \cite{schiffmann2020cohomological,crawley2022deformed,kaplan2019multiplicative,bozec2023calabi}, categories of coherent sheaves on K3 and Abelian surfaces \cite{huybrechts2010geometry}, and categories of Higgs bundles over smooth projective curves \cite{simpson1992higgs}. We prove our results via the BPS algebras defined in \cite{davison2022BPS}, following the sequence of papers \cite{harvey1996algebras,kontsevich2010cohomological,davison2020cohomological}.

In the original paper \cite{harvey1996algebras} on BPS algebras of Harvey and Moore, a physical argument for the appearance of generalised Kac--Moody (GKM) algebras in algebras of BPS states is given.  In the study of operators via correspondences on Hilbert schemes of points on $\BoA^2$, it was shown by Nakajima and Grojnowski \cite{nakajima1997heisenberg,grojnowski1996instantons} that an infinite-dimensional Heisenberg algebra, which is a very particular generalised Kac--Moody Lie algebra, plays a central role.  These papers in turn were inspired by the pioneering work of Nakajima \cite{nakajima1994instantons}, in which the deep connection between (usual) Kac--Moody Lie algebras and algebras of correspondences on quiver varieties was first uncovered.  In this paper we explain how \emph{every} suitably geometric 2CY category gives rise to a (generalised) Kac--Moody Lie algebra.  More precisely, we show that the BPS algebra of a 2CY category $\CA$ with good moduli space of objects is the universal enveloping algebra of one half of a generalised Kac--Moody Lie algebra, and that the Cartan datum of this Lie algebra is explicitly given in terms of intersection cohomology of certain coarse moduli spaces of objects in $\CA$.  Special cases of (the doubled versions of) these BPS Lie algebras are identified with algebras of correspondences appearing in \cite{nakajima1994instantons,nakajima1997heisenberg,nakajima1998quiver,grojnowski1996instantons}; indeed part of our motivation was to find a unified explanation for all of these results.  On the other hand, our results apply to very general 2CY categories.  Reflecting the range of 2CY categories in mathematics, we can then use our structural results to prove conjectures in combinatorics (positivity of cuspidal polynomials), representation theory (decomposition of the cohomology of Nakajima quiver varieties), and algebraic geometry (cohomological integrality for K3 surfaces).

For the combinatorial applications, we consider certain functions counting isomorphism classes of representations of quivers over finite fields. A consequence of Theorem \ref{corollary:absBor} applied to the category $\Rep\Pi_Q$ of representations of the preprojective algebra of a quiver is a proof in full generality of the Bozec--Schiffmann positivity conjecture for (absolutely) cuspidal polynomials of quivers \cite{bozec2019counting}. This is a strengthening of the positivity of Kac polynomials conjectured in \cite{kac1983root}. As in the proofs \cite{hausel2013positivity,davison2018purity,dobrovolska2016moduli} of the Kac positivity conjecture, we prove the positivity of absolutely cuspidal polynomials by identifying their coefficients with dimensions of certain vector spaces, more specifically by identifying these polynomials with intersection Poincar\'e polynomials of singular Nakajima varieties. We also define and prove the positivity of nilpotent and $1$-nilpotent absolutely cuspidal polynomials, which are associated to nilpotent and $1$-nilpotent Kac polynomials counting representations of quivers over finite fields satisfying various nilpotency conditions \cite{bozec2020number}.

In geometric representation theory, we apply our main result to the cohomology of Nakajima quiver varieties, considered as representations of the BPS Lie algebra. Applying the main theorem to the category of $\Pi_{Q_{\ff}}$-modules, where $Q_{\ff}$ is a framed quiver, we describe a decomposition of the \emph{entire} cohomology of Nakajima quiver varieties (i.e. in all cohomological degrees) into irreducible lowest weight modules for the (doubled) BPS Lie algebra. Since this Lie algebra contains the usual Kac--Moody Lie algebra associated to the unframed quiver $Q$, this answers a longstanding question of Nakajima regarding the possibility of extending his famous decomposition of (part of) the cohomology of quiver varieties into irreducible lowest weight modules. This provides a unified construction of the action of Kac--Moody Lie algebras and Heisenberg algebras on cohomology of various quiver varieties \cite{nakajima1997heisenberg,nakajima1998quiver}, as well as cohomology of more general moduli spaces of objects in 2CY categories.

Within enumerative algebraic geometry, and more specifically cohomological Donaldson--Thomas theory \cite{davison2020cohomological,szendroi2014cohomological}, we use our results to prove the cohomological integrality theorem for general 2CY categories and their undeformed 3CY completions. The 2CY version of this conjecture states that there is an isomorphism 
\[
\HO^{\BoMo}(\FM_{\CA},\underline{\BoQ}^{\vir})\cong \Sym\left(\bigoplus_{a\in \K(\CA)}\underline{\CF}_a\otimes \HO(\B\BoC^*,\BoQ)\right)
\]
where $\FM_{\CA}$ is the stack of objects in a fixed 2CY category $\CA$ and $\K(\CA)$ is a version of a numerical Grothendieck group for $\CA$. Up to a cohomological shift, the left-hand side is the Borel--Moore homology of $\FM_{\CA}$, and each $\underline{\CF}_a$ is a bounded complex of mixed Hodge structures. We prove this conjecture by proving a stronger relative version (Theorem~\ref{theorem:relPBW}), in terms of the tensor category of mixed Hodge modules on the coarse moduli space $\CM_{\CA}$. The above isomorphism is an isomorphism of Poincar\'e--Birkhoff--Witt (PBW) type, and indeed we prove it by producing a (relative) PBW isomorphism for the cohomological Hall algebra (CoHA) associated to $\CA$. Via dimensional reduction \cite{kinjo2022dimensional}, the theorem implies the cohomological integrality conjecture regarding the 3CY completions of 2CY categories; a concrete consequence is the (refined) integrality theorem for the DT partition function of local K3 surfaces. In the rest of the introduction we explain, connect and motivate these results in more depth.

\subsection{Generalised Kac--Moody Lie algebras}
\label{subsection:mainresults}
Our main structural theorems in \S \ref{subsubsection:BPSalg2CY} are expressed in terms of generalised Kac--Moody Lie algebras, and their positive halves, which we briefly introduce in this section following \cite{kac1990infinite,borcherds1988generalized}.
\subsubsection{Generalised Kac--Moody Lie algebras associated to a monoid with bilinear form and weight function}

\label{subsubsection:pospartGKM}
Let $(M,+)$ be a monoid. Often, we will take $M$ such that the natural map $M\rightarrow \widehat{M}$ to the groupification is an inclusion into a lattice (which becomes convenient when considering Weyl group actions). Let $(-,-)\colon M\times M\rightarrow \BoZ$ be a symmetric bilinear form. In \S\ref{subsection:roots}, we define the sets of \emph{primitive positive roots} $\Sigma_{M,(-,-)}$, \emph{simple positive roots} $\Phi^+_{M,(-,-)}$ and the Cartan matrix $A_{M,(-,-)}=(a_{m,n}=(m,n))_{m,n\in\Phi^+_{M,(-,-)}}$. We assume that $a_{m,m}\in 2\BoZ_{\leq 1}$ and that $a_{m,n}\leq 0$ if $a_{m,m}=2$. For a function
\[
 P\colon \Phi_{M,(-,-)}^+\rightarrow \BoN[t^{\pm 1/2}];\quad\quad
 m  \mapsto P_m(t)
\]
such that $P_m(t)=1$ if $a_{m,m}=2$, we define a $\widehat{M}\times\BoZ$-graded generalised Kac--Moody algebra $\mathfrak{g}_{M,(-,-),P}$ (in the sense of Borcherds \cite{borcherds1988generalized}, see \cite{bozec2019counting} for the graded version) associated to this datum, by generators and relations. We refer to such a function $P$ as a \emph{weight function}. The Lie algebra $\mathfrak{g}_{M,(-,-),P}$ has a triangular decomposition $\mathfrak{g}_{M,(-,-),P}=\mathfrak{n}^-_{M,(-,-),P}\oplus\mathfrak{h}_{M,(-,-)}\oplus \mathfrak{n}^+_{M,(-,-),P}$. We also have a description by generators and relations (see \S \ref{subsection:GKMmonoids}) and a triangular decomposition of its enveloping algebra
\[
\UEA(\mathfrak{g}_{M,(-,-),P})\cong \UEA(\mathfrak{n}^-_{M,(-,-),P})\otimes \UEA(\mathfrak{h}_{M,(-,-)})\otimes \UEA(\mathfrak{n}^+_{M,(-,-),P}).
\]
Both $\Fn^+_{M,(-,-),P}$ and $\UEA(\Fn_{M,(-,-),P}^+)$ are $M\times\BoZ$-graded.

In concrete situations, the polynomial $P_m(t)$ is the graded dimension of some given $\BoZ$-graded vector space $V_m$ (the intersection cohomology of some algebraic variety). The above algebra is the unique (up to isomorphism, and some choice regarding the Cartan subalgebra) cohomologically graded generalised Kac--Moody Lie algebra with generators in the $m$th root weight space identified with $V_m$.

\subsubsection{Positive halves of generalised Kac--Moody (Lie) algebras in symmetric tensor categories}
\label{subsubsection:pospartGKMtensorcategories}
Let $(\CM,\oplus)$ be a commutative monoid in the category of complex schemes, for which we have a decomposition $\CM=\coprod_{m\in M}\CM_m$ into finite type unions of finite type connected components $\CM_m$ of $\CM$ (i.e. $\CM_m$ might be disconnected but it is a finite type scheme) for some monoid $(M,+)$ as in \S\ref{subsubsection:pospartGKM}. We allow $\CM$ to have infinitely many connected components. We assume that $\oplus$ respects this decomposition. In this case, we do not have a groupification $\widehat{\CM}$ with which to define the full generalised Kac--Moody Lie algebra with generators for positive roots given by some selection of sheaves on $\CM$, but we nonetheless can define the \emph{positive half} of such a Lie algebra.

We let $(-,-)\colon M\times M\rightarrow \BoZ$ be a bilinear form, $\Sigma_{M,(-,-)}$ be the set of primitive positive roots, $\Phi_{M,(-,-)}^+$ be the set of simple positive roots and $A_{M,(-,-)}$ be the Cartan matrix as in \S\ref{subsubsection:pospartGKM}. We assume that $\CM_{m}=\pt$ for any $m\in\Phi_{M,(-,-)}^{+,\real}$ and that the sum map $+\colon M\times M\rightarrow M$ is finite.

Let $\underline{\SG}_m\in\MHM(\CM_m)$ be a set of mixed Hodge modules indexed by $m\in\Phi_{M,(-,-)}^+$. We assume that $\underline{\SG}_m=\underline{\BoQ}_{\CM_m}=\underline{\BoQ}_{\pt}$ is the constant mixed Hodge module if $m\in\Phi_{M,(-,-)}^{+,\real}$. We let $P$ be the weight function given by the Poincar\'e polynomials of these $\underline{\SG}_m$:
 \[
 \begin{matrix}
  P\colon &\Phi_{M,(-,-)}^+&\rightarrow&\BoN[t^{\pm 1/2}]&\\
      &m   &\mapsto  &P(\underline{\SG}_m)&=\sum_{j\in\BoZ}\dim\HO^j(\underline{\SG}_m)t^{j/2}.
 \end{matrix}
 \]
Given the above data, we define $\mathfrak{n}^+_{\CM,\underline{\SG}}$ to be the free Lie algebra object generated by $\underline{\SG}=\bigoplus_{m\in M} \underline{\SG}_m$ in the symmetric tensor category $\MHM(\CM)$, quotiented by the Lie ideal generated by generalised Serre relations (see \S \ref{subsection:GKMmonoids} for the precise construction). The Lie algebra obtained by taking derived global sections of $\mathfrak{n}^+_{\CM,\underline{\SG}}$ and forgetting the mixed Hodge structures is isomorphic to $\mathfrak{n}^+_{M,(-,-),P}$, (half) the generalised Kac--Moody Lie algebra defined as in \S\ref{subsubsection:pospartGKM}.

Similarly we define $\UEA(\mathfrak{n}^+_{\CM,\underline{\SG}})$ to be the free unital associative algebra generated by $\underline{\SG}$, quotiented by the two-sided ideal generated by the same relations. The algebra obtained by taking derived global sections and forgetting the MHS of $\UEA(\mathfrak{n}^+_{\CM,\underline{\SG}})$ is isomorphic to $\UEA(\mathfrak{n}^+_{M,(-,-),P})$.

\subsection{The BPS algebra of a 2-Calabi--Yau category}
\label{subsubsection:BPSalg2CY}
We let $\CA$ be an Abelian 2-Calabi--Yau category in the sense of \cite{davison2022BPS}. It is an Abelian subcategory of the homotopy category of some ambient dg category $\SC$. We briefly recall in \S\ref{section:modulistack2dcats} the assumptions on $\CA$. In particular, we assume that it has a moduli stack of objects, $\FM_{\CA}$ which is an Artin stack with a good moduli space $\JH_{\CA}\colon\FM_{\CA}\rightarrow\CM_{\CA}$ (as defined in \cite{alper2013good}). The good moduli space $\CM_{\CA}$ is assumed to be a scheme and it has the monoid structure $\oplus\colon\CM_{\CA}\times\CM_{\CA} \to \CM_{\CA}$ induced by the direct sum $\oplus\colon\FM_{\CA}\times\FM_{\CA}\rightarrow\FM_{\CA}$. The complex of mixed Hodge modules $\underline{\SA}_{\CA}\coloneqq(\JH_{\CA})_*\BD\underline{\BoQ}_{\FM_{\CA}}^{\vir}\in\CD^+(\MHM(\CM_{\CA}))$ is equipped with a multiplication morphism
\[
\mult\colon \underline{\SA}_{\CA}\boxdot\underline{\SA}_{\CA}\rightarrow\underline{\SA}_{\CA}
\]
and called the \emph{the relative cohomological Hall algebra} \cite{davison2022BPS}. 

The superscript $\vir$ indicates a twist by a certain power (half the virtual dimension of the stack) of the Tate twist which we make precise in \S\ref{subsection:theCoHA}. The complex of mixed Hodge modules $\ulrelCoHA_{\CA}$ is concentrated in nonnegative cohomological degrees (Lemma~\ref{lemma:nonnegativeperverse}) and the degree $0$ cohomology mixed Hodge module, $\underline{\BPS}_{\CA,\Alg}\coloneqq\CH^0(\underline{\SA}_{\CA})$, carries the induced multiplication $\CH^{0}(\mult)$. It is called the \emph{BPS algebra sheaf}.

All the CoHAs we consider in this paper are graded by $M=\pi_0(\CM_{\CA})$, the monoid of connected components of the moduli space of objects under consideration.  Given a $M$-graded algebra object $\mathscr{A}$, with multiplication denoted $\mult$, and a morphism of monoids $\psi\colon M^{\times 2}\rightarrow \BoZ/2\BoZ$ we denote by $\mathscr{A}^{\psi}$ the sign-twisted multiplication obtained by setting $\mult^{\psi}_{a,b}=(-1)^{\psi(a,b)}\mult_{a,b}$ for any $a,b\in M$.  Our main results are most cleanly stated in terms of a certain sign-twisted multiplication, which is defined in \S \ref{subsubsection:twistmultiplication}.

We let $\Sigma_{\K(\CA)}\coloneqq \Sigma_{\K(\CA),(-,-)_{\SC}}$ and $\Phi^+_{\K(\CA)}\coloneqq \Phi^+_{\K(\CA),(-,-)_{\SC}}$ be the sets of primitive positive roots and simple positive roots respectively (for the monoid with bilinear form $(\K(\CA),(-,-)_{\SC})$), as in \S\ref{subsection:roots}. For $m\in\Phi_{\K(\CA)}^{+,\iso}$, we write $m=lm'$ for $m'\in\Sigma_{\K(\CA)}$ and $l\in\BoN$. This decomposition of $m$ is unique and we have a natural closed immersion
\[
\begin{matrix}
 u_{m}&\colon& \CM_{\CA,m'}&\rightarrow& \CM_{\CA,m}\\
   &   & x   &\mapsto  &x^{\oplus l}.
\end{matrix}
\]
The (connected components of the) complex scheme $\CM_{\CA,m'}$ are of dimension $2$ (Proposition~\ref{proposition:geometrygoodmodspace}).  We let $\underline{\SG'}_m=(u_m)_*(\underline{\IC}(\CM_{\CA,m'}))$. Then for $m\in \Phi^+_{\K(\CA)}$ we define
\[
\underline{\SG}_m\coloneqq\begin{cases}\ul{\IC}(\CM_{\CA,m})&\textrm{if }m\in \Phi^+_{\K(\CA)}\setminus\Phi_{\K(\CA)}^{+,\iso}\\
(u_m)_*(\underline{\IC}(\CM_{\CA,m'}))&\textrm{if }m=lm'\in\Phi_{\K(\CA)}^{+,\iso},\;m'\in \Sigma_{\K(\CA)}.\end{cases}
\]
The definitions of the intersection complexes $\ul{\IC}(-)$ are recalled in \S \ref{notations_subsec}. For every $m$, $\underline{\SG}_m$ is a semisimple mixed Hodge module on $\CM_{\CA,m}$, and in the isotropic case it is supported on the \emph{small diagonal} of $\CM_{\CA,m}$ which we define to be the image of $u_m$. Its restriction to each of the connected components of $\CM_{\CA,m}$ is simple.

\subsubsection{BPS algebras as GKM algebras}
We let $\underline{\SG}=\bigoplus_{m\in \Phi_{\K(\CA)}^+}\underline{\SG}_m$ be as defined above. Note that $\underline{\SG}_m\cong\underline{\BoQ}_{\pt}$ if $m\in\Phi_{\K(\CA)}^{+,\real}$. We can finally state our first main theorem, proved in \S \ref{subsection:mainproof}.
\begin{theorem}
\label{theorem:BPSalgisenvBorcherds}
There exists a canonical isomorphism of algebra objects
 \[
 \UEA(\mathfrak{n}^+_{\CM_{\CA},(-,-)_{\SC},\underline{\SG}})\xrightarrow{\cong}\ul{\BPS}_{\CA,\Alg}^{\psi}
 \]
 in $\MHM(\CM_{\CA})$.
\end{theorem}

Recall that $\underline{\rmBPS}_{\CA,\Alg}^{\psi}\coloneqq\HO^*(\ul{\BPS}_{\CA,\Alg}^{\psi})$ is the algebra formed by taking the derived global sections of the BPS algebra sheaf. It is an ordinary $\K(\CA)\times \BoZ$-graded algebra (where the copy of $\BoZ$ encodes the cohomological grading) and it is called \emph{the BPS algebra}. By construction, it has a lift to an algebra object in the category of mixed Hodge structures.

If $\imath_{\triangle}\colon \CM_{\CA}^{\triangle}\hookrightarrow \CM_{\CA}$ is a locally closed inclusion of a saturated submonoid (for example, $\triangle$ may denote various versions of nilpotency in the case $\CA=\mathrm{rep}(\Pi_Q)$), the complex $\imath^!_{\triangle}\ul{\BPS}_{\CA,\Alg}^{\psi}$ inherits the structure of an algebra object in the derived category of mixed Hodge modules on $\CM_{\CA}^{\triangle}$ via restriction, and we similarly define the \emph{$\triangle$ BPS algebra} by taking derived global sections:
\[
\underline{\rmBPS}_{\CA,\Alg}^{\triangle,\psi}\coloneqq\HO^*(\CM_{\CA}^{\triangle},\imath_{\triangle}^!\ul{\BPS}_{\CA,\Alg}^{\psi}).
\]

Applying the functor $\imath_{\triangle}^!$ and taking derived global sections in Theorem~\ref{theorem:BPSalgisenvBorcherds} implies our main result on BPS algebras:

\begin{theorem}
\label{corollary:absBor}
There is a canonical isomorphism of algebras
\[
\UEA(\mathfrak{n}^+_{\K(\CA),(-,-)_{\SC},\HO^*(\imath_{\triangle}^!\SG)})\xrightarrow{\cong}\rmBPS_{\CA,\Alg}^{\triangle,\psi}.
\]
and thus an isomorphism of algebras $\UEA(\mathfrak{n}^+_{\K(\CA),(-,-)_{\SC},\IP^{\triangle}_{\CA}})\xrightarrow{\cong}\rmBPS_{\CA,\Alg}^{\triangle,\psi}$, where $\IP^{\triangle}_{\CA}$ is the weight function
\[
\begin{matrix}
 \IP^{\triangle}_{\CA}\colon &\Phi^+_{\K(\CA)}&\rightarrow&\BoN[t^{\pm 1/2}]&\\
      &m      &\mapsto  &\ch^{\BoZ}(\HO^*(\imath_{\triangle}^!\SG_m))&\coloneqq \sum_{j\in\BoZ}\dim \HO^j(\imath_{\triangle}^!\SG_m)t^{j/2}.
\end{matrix}
\]
In particular, there is an embedding $\rmBPS_{\CA,\Alg}^{\triangle,\psi}\subset \UEA(\Fg_{\K(\CA),(-,-)_{\SC},\IP^{\triangle}_{\CA}})$
 identifying the BPS algebra as the positive half of a cohomologically graded generalised Kac--Moody algebra.
\end{theorem}
Explicitly, for $m\in\Sigma_{\K(\CA)}\subset\Phi_{\K(\CA)}^+$, $\IP^{\triangle}_{\CA}(m)$ is the Poincar\'e polynomial of $\HO^*\!(\CM^{\triangle}_{\CA,m},\imath_{\triangle}^!\IC(\CM_{\CA,m}))$, and for $m=lm'$ with $m'\in\Sigma_{\CA}$ isotropic, $\IP^{\triangle}_{\CA}(m)$ is the Poincar\'e polynomial of $\HO^*\!(\CM^{\triangle}_{\CA,m'},\imath_{\triangle}^!\IC(\CM_{\CA,m'}))$, so that spaces of simple roots of $\Fg_{\CA}^{\ttBPS, \triangle}$ defined below are identified with $\imath_{\triangle}$-restrictions of intersection complexes on moduli spaces of objects of $\CA$.

\label{subsubsection:cohintegrality2CY}
We let $\CA$, $\FM_{\CA}$, $\CM_{\CA}$ be as in \S\ref{subsubsection:BPSalg2CY}. Motivated by Theorem~\ref{theorem:BPSalgisenvBorcherds}, we define the \emph{relative BPS Lie algebra} of $\CA$ as 
\begin{align}
\label{equation:BPS_LA_def}
\ul{\BPS}_{\CA,\Lie}\coloneqq\mathfrak{n}^+_{\CM_{\CA},(-,-)_{\SC},\underline{\SG}},
\end{align}
we define
\begin{equation}
\label{equation:BPS_ALA_defMHM}
\ul{\Fn}^{\ttBPS,\triangle,+}_{\CA}\coloneqq \HO^*(\CM^{\triangle}_{\CA},\imath_{\triangle}^!\ulBPS_{\CA,\Lie})
\end{equation}
so that (forgetting about Hodge structures)
\begin{align}
\label{equation:BPS_ALA_def}
\Fn^{\ttBPS,\triangle,+}_{\CA}\coloneqq & \HO^*(\CM^{\triangle}_{\CA},\imath_{\triangle}^!\BPS_{\CA,\Lie})\\
\nonumber
\cong &\Fn^+_{\K(\CA),(-,-)_{\SC},\HO^*(\CM^{\triangle}_{\CA},\imath_{\triangle}^!\SG)}.
\end{align}
We define the \emph{full} BPS Lie algebra:
\begin{equation}
\label{definition:fullBPSLA}
\Fg^{\ttBPS,\triangle}_{\CA}\coloneqq \Fg_{\K(\CA),(-,-)_{\SC},\HO^*(\CM^{\triangle}_{\CA},\imath_{\triangle}^!\SG)}.
\end{equation}
By construction, the full BPS Lie algebra is a generalised Kac--Moody Lie algebra containing the BPS Lie algebra $\HO^*(\CM^{\triangle}_{\CA},\imath_{\triangle}^!\BPS_{\CA,\Lie})$ as its strictly positive half.

In the rest of the introduction we describe some of the consequences of Theorems \ref{theorem:BPSalgisenvBorcherds} and \ref{corollary:absBor}.
\subsection{Positivity for cuspidal polynomials}
\label{subsection:positivitycusppolsintro}
Cuspidal polynomials $C_{Q,\dd}(q)$ of a quiver $Q$ were introduced in \cite{bozec2019counting}, and defined as the polynomials counting the dimensions of the graded components of the space of primitive elements of the constructible Hall algebra $H_{Q,\BoF_q}$ over a finite field. The primitive elements are defined using Green's coproduct on $H_{Q,\BoF_q}$ \cite{green1995hall} and they form a minimal subspace generating the Hall algebra \cite{sevenhant2001relation}. See \cite{schiffmann2018kac} for motivation and background regarding Hall algebras and cuspidal polynomials.

Bozec and Schiffmann also defined a closely related family of polynomials, the absolutely cuspidal polynomials $C_{Q,\dd}^{\rmabs}(q)$. They are defined in a combinatorial way from cuspidal polynomials and coincide with the latter for \emph{hyperbolic} dimension vectors. More precisely, $C_{Q,\dd}(q)=C_{Q,\dd}^{\rmabs}(q)$ if $(\dd,\dd)_Q<0$, $C_{Q,\dd}(q)=C_{Q,\dd}^{\rmabs}(q)=q^{g_i}$ if $\dd=1_i\in\BoN^{Q_0}$ is a basis vector and the vertex $i$ carries $g_i$ loops, and, if $(\dd,\dd)_Q=0$ with $\dd$ indivisible, then the families $(C_{Q,l\dd}(q))_{l\geq 1}$ and $(C_{Q,l\dd}^{\rmabs}(q))_{l\geq 1}$ are related by the formula utilising plethystic exponentials (see for example \cite[\S1.5]{bozec2019counting} for a definition of this combinatorial operation)
\[
 \Exp_{z}\left(\sum_{l\geq 1}C_{Q,l\dd}(q)z^{l\dd}\right)=\Exp_{z,q}\left(\sum_{l\geq 1}C_{Q,l\dd}^{\rmabs}(q)z^{l\dd}\right).
\]

Bozec and Schiffmann formulated the following positivity conjecture for the family $(C_{Q,\dd}^{\rmabs}(q))_{\dd\in\BoN^{Q_0}}$:
\begin{conjecture}[{\cite{bozec2019counting}}]
\label{conjecture:positivity}
 For $\dd\in\BoN^{Q_0}$, the polynomial $C^{\rmabs}_{Q,\dd}(q)$ has nonnegative coefficients.
\end{conjecture}

By taking the Serre subcategories of the category $\Rep(Q)$ of finite dimensional representations of $Q$, of nilpotent representations, $\Rep^{\nil}(Q)$, and of $1$-nilpotent representations, $\Rep^{1\snil}(Q)$ (\S\ref{subsubsection:seminilpotentrepquivers}), one can define the associated constructible Hall algebras $H_{Q,\BoF_q}^{\nil}$ and $H_{Q,\BoF_q}^{1\snil}$. We have inclusions of bialgebras $H_{Q,\BoF_{Q}}^{\nil}\subset H_{Q,\BoF_q}^{1\snil}\subset H_{Q,\BoF_q}$. As in \cite{bozec2019counting}, we define the nilpotent cuspidal polynomials $C_{Q,\dd}^{\nil}(q)$ and $1$-nilpotent cuspidal polynomials $C_{Q,\dd}^{1\snil}(q)$ as the dimension counts of primitive elements of the twisted bialgebras $H_{Q,\BoF_{q}}^{\nil}$ and $H_{Q,\BoF}^{1\snil}$ respectively, as well as their absolute versions $C_{Q,\dd}^{\nil,\rmabs}(q)$ and $C_{Q,\dd}^{1\snil,\rmabs}(q)$. We may therefore formulate in the same way the positivity conjecture for these new families of polynomials.

We prove all of these conjectures in this paper.

\begin{theorem}[Theorems \ref{theorem:ICNQV} and \ref{theorem:positivitycusppols}]
Absolutely cuspidal polynomials have nonnegative coefficients (Conjecture~\ref{conjecture:positivity} is true). Additionally, absolutely cuspidal nilpotent and $1$-nilpotent polynomials have nonnegative coefficients.
\end{theorem}

A key fact in the proof of the positivity conjectures is the interpretation of cuspidal polynomials as graded multiplicities of simple roots. Namely, for the monoid $M=\BoN^{Q_0}$ with the symmetrised Euler form and a weight function $P$ as in \S\ref{subsubsection:pospartGKM}, there is a combinatorial formula giving the character of the positive part of the GKM Lie algebra $\Fg_{M,(-,-),P}$ (a generalization of the Weyl character formula). In the graded context, it is described in detail in \cite{bozec2019counting}. By letting $P$ vary, we obtain a function
\[
 W_Q:\{P\colon\Phi_{M,(-,-)}^+\rightarrow\BoN[q^{\pm1/2}]\}\rightarrow\BoN[q^{\pm1/2}][\![M]\!].
\]
sending a weight function to the character of the corresponding positive part of the generalised Kac--Moody Lie algebra.

\begin{proposition}[{\cite{bozec2019counting}}]
\label{proposition:Weylcharacter}
 The map $W_Q$ is injective and if $\Exp_{z,t}\left(\sum_{\dd\in\BoN^{Q_0}}A_{Q,\dd}(t)z^{\dd}\right)$ is in its image, then it is the image under $W_Q$ of the weight function $\dd\mapsto C_{Q,\dd}^{\rmabs}(t)$ (we denote by $A_{Q,\dd}(q)$ the Kac polynomials of $Q$ counting absolutely indecomposable $\dd$-dimensional representations of $Q$ over $\BoF_q$).
\end{proposition}
This proposition is proven in [loc. cit.] by an iterative inversion of the Weyl character formula for graded GKM algebras. There is an analogue of Proposition~\ref{proposition:Weylcharacter} for the nilpotent and $1$-nilpotent Kac and cuspidal polynomials. Proposition~\ref{proposition:Weylcharacter} implies that the positivity of cuspidal polynomials is a strengthening of the positivity of Kac polynomials of quivers proven in \cite{hausel2013positivity}. Simultaneously, our results prove a strengthening of the Kac constant term conjecture \cite{kac1983root}, proved originally in \cite{hausel2010kac}. This conjecture states that the constant terms of the Kac polynomials are equal to the root multiplicities of a certain Kac--Moody Lie algebra $\Fg_{\Gamma}$; we show that the \emph{entire} Kac polynomial is the characteristic function of a larger cohomologically graded (generalised) Kac--Moody Lie algebra containing $\Fg_{\Gamma}$ as its (cohomological) degree zero piece.

Our proof of the positivity of cuspidal polynomials also gives a cohomological interpretation of their coefficients.

\begin{theorem}[Theorem~\ref{theorem:ICNQV}]
\label{theorem:IPPNQV}
 Let $\dd\in\Sigma_{\Pi_Q}$. Then, $C_{Q,\dd}(q)=\IP(\CM_{\Pi_Q,\dd})(q^{-1/2})$.
\end{theorem}

This interpretation gives a way to calculate the intersection Poincar\'e polynomials of quiver varieties. The same way that Hua's formula \cite{hua2000counting} gives a way to calculate Kac polynomials by giving an effective formula for the plethystic exponential of their generating series, the above results enable us to recover the intersection Poincar\'e polynomials from the same series.

The intersection Poincar\'e polynomial of any quiver variety can be recovered using Theorem~\ref{theorem:IPPNQV}, the canonical decomposition of dimension vectors of quivers \cite[Theorem 1.1]{crawley2002decomposition} which induces a decomposition of Nakajima quiver varieties \cite{crawley2002decomposition}, and the formula for the intersection Poincar\'e polynomial of symmetric products (see Corollary \ref{corollary:IPallNQV} for the general result). Since the canonical decomposition of a dimension vector can be determined algorithmically, Theorem~\ref{theorem:IPPNQV} implies that intersection cohomology of any (singular) quiver variety can be too.

\subsection{Cohomological integrality and PBW theorems in Donaldson--Thomas theory}

The BPS algebra $\aBPS^{\psi}_{\CA,\Alg}$ is defined to be the zeroeth part of the perverse filtration on the cohomological Hall algebra $\HO^*\!\relCoHA_{\CA}$ associated to the morphism $\JH_{\CA}$ \cite{davison2022BPS}. In particular, it is a subalgebra of $\HO^*\!\relCoHA_{\CA}$. Theorem~\ref{corollary:absBor} completely determines the structure of this algebra, up to determining the intersection Poincar\'e polynomials of some singular varieties associated to simple roots (and in the case $\CA=\Rep\Pi_Q$ this can be done algorithmically using absolutely cuspidal polynomials, as explained in \S\ref{subsection:positivitycusppolsintro}).

Although Theorem~\ref{corollary:absBor} only concerns a subalgebra of the CoHA, we can use the same techniques to write down an explicit PBW isomorphism for the \emph{entire} CoHA. To start with, in \S \ref{subsubsection:fullPBW} we prove a stronger, ``relative'' PBW isomorphism in the category of complexes of mixed Hodge modules on the coarse moduli space $\CM_{\CA}$:
\begin{theorem}[Relative PBW isomorphism]
\label{theorem:relPBW}
 Let $\CA$ be a $2$-Calabi--Yau Abelian category satisfying the assumptions of \S\ref{section:modulistack2dcats}.  After choosing a positive determinant line bundle $\det\colon \FM_{\CA}\rightarrow \B\BoC^*$, we have a canonical isomorphism
 \[
 \Psi^{\psi}_{\CA}\colon \Sym_{\boxdot}\left(\ul{\BPS}_{\CA,\Lie}^{}\otimes \HO^*(\B\BoC^*,\ul{\BoQ})\right)\xrightarrow{\cong}\ulrelCoHA_{\CA}^{\psi}
 \]
 of complexes of mixed Hodge modules in $\CD^+(\MHM(\CM_{\CA}))$, with $\ul{\BPS}_{\CA,\Lie}$ defined as in \eqref{equation:BPS_LA_def}. The isomorphism is obtained by evaluating symmetric tensors in $\ul{\BPS}_{\CA,\Lie}^{}\otimes \HO^*(\B\BoC^*,\ul{\BoQ})$ using the ($\psi$-twisted) Hall algebra structure on the target.
\end{theorem}

See \S \ref{subsection:determinant} for the definition of a positive determinant bundle. Applying the restriction functor $\imath_{\triangle}^!$ and taking derived global sections, we obtain the following corollary.
\begin{corollary}[PBW isomorphism]
\label{corollary:PBW}
 Let $\CA$ be a $2$-Calabi--Yau Abelian category satisfying the assumptions of \S\ref{section:modulistack2dcats}. Let $\CA^{\triangle}\subset \CA$ be a Serre subcategory. After choosing a positive determinant line bundle, we have a canonical PBW isomorphism
\[
 \HO^*(\imath_{\triangle}^!\Psi^{\psi}_{\CA})\colon \Sym\!\left(\ul{\Fn}^{\ttBPS,\triangle,+}_{\CA}\otimes\HO^*(\B\BoC^*,\ul{\BoQ})\right)\xrightarrow{\cong}\HO^*\!\ulrelCoHA_{\CA}^{\triangle,\psi}=\HO^{\BoMo}(\FM_{\CA}^{\triangle},\ul{\BoQ}^{\vir})
\]
of $\K(\CA)\times \BoZ$-graded mixed Hodge structures, with $\ul{\Fn}^{\ttBPS,\triangle,+}_{\CA}$ defined as in \eqref{equation:BPS_ALA_defMHM}. This isomorphism is obtained via the inclusion of cohomologically graded mixed Hodge structures $\ul{\Fn}^{\ttBPS,\triangle,+}_{\CA}\otimes\HO^*(\B\BoC^*,\BoQ)\hookrightarrow \HO^*\!\ulrelCoHA_{\CA}^{\triangle,\psi}$ determined by our choice of determinant line bundle, and the evaluation of symmetric tensors using the ($\psi$-twisted) cohomological Hall algebra structure on the target.
\end{corollary}
As usual, we emphasize that the isomorphisms of Theorem~\ref{theorem:relPBW} and Corollary~\ref{corollary:PBW} are \emph{not} algebra morphisms, although both sides have their own algebra structures. The left-hand sides have the algebra structures given by the $\Sym$ construction (i.e. they are supersymmetric algebras) while the right-hand sides have the cohomological Hall algebra structure. However, the isomorphism $\Psi^{\psi}_{\CA}$ \emph{is} constructed using the algebra structure on $\ulrelCoHA_{\CA}^{\psi}$. This is similar to the usual PBW theorem for enveloping algebras, which gives an isomorphism, after passing to underlying vector spaces, between a polynomial algebra (i.e. the symmetric algebra over a vector space) and the enveloping algebra of a Lie algebra.

If we choose $\CA$ to be the 2-Calabi--Yau category of $\Pi_Q$-representations, the existence of \textit{some} cohomological integrality isomorphism as in Theorem~\ref{theorem:relPBW} and Corollary~\ref{corollary:PBW} is (via dimensional reduction \cite{davison2017critical}) a particular case of the PBW isomorphism of Davison--Meinhardt (see \cite{davison2020cohomological}, then \cite{davison2016integrality}), and purity of the BPS sheaf is one of the main results of \cite{davison2020bps}. However, even in this case, the precise determination of the IC complexes making up the BPS sheaf is a new result.

\subsubsection{Donaldson--Thomas theory}
\label{section:DTcrashcourse}
The main application of Corollary~\ref{corollary:PBW} in enumerative geometry concerns the cohomological Donaldson--Thomas (DT) theory of local K3 surfaces, which we explain here. To start with some background, let $X$ be a projective 3CY variety, meaning that $X$ is smooth, three-dimensional, and has trivial canonical bundle, $\mathscr{O}_X\cong K_X$. Donaldson--Thomas theory, as introduced by Richard Thomas in \cite{thomas2000holomorphic}, assigns numbers to certain (fine) moduli schemes $\CM$ of coherent sheaves on $X$, using perfect obstruction theories to produce a degree zero virtual fundamental class of $\CM$. Thomas then defines $\#_{\DT}(\CM)$ to be the degree of this class. This theory was subsequently generalised in multiple directions. Firstly it was proved by Behrend \cite{behrend2009donaldson} that the count $\#_{\DT}(\CM)$ could be calculated as a weighted Euler characteristic of a certain constructible function $\nu_{\CM}$ (the Behrend function), so that it makes sense to define this number even if $\CM$ is not projective. In particular, we can replace $X$ with a local K3 surface $S\times\BoA^1$. Secondly, via the papers of Joyce--Song \cite{joyce2012theory} and Kontsevich--Soibelman \cite{kontsevich2008stability}, a theory was developed that removed the need to study only fine moduli spaces $\CM$, i.e. stacks $\FM$, possibly containing strictly semistable objects, are handled by their theory. Finally, an effective theory of DT mixed Hodge modules on stacks of objects in 3CY categories was developed in the sequence of papers \cite{brav2015symmetries,brav2019darboux,ben2015darboux}; associated to a moduli stack $\FM$ of objects in a 3CY category $\CA$, equipped with an orientation in the sense of \cite{joyce2015classical}, they define the mixed Hodge module $\phi_{\DT}$, the pointwise Euler characteristic of which is equal to the Behrend function. We recover the weighted Euler characteristics (where they make sense) by taking actual Euler characteristics of derived global sections of $\phi_{\DT}$. 

Let $\CA$ be a 2CY category with a good moduli space, satisfying the assumptions of \S \ref{section:modulistack2dcats}. A motivating example for us comes from setting $\CA$ to be the category of Gieseker semistable coherent sheaves, of fixed normalised Hilbert polynomial, on a polarised K3 surface $S$. In order to obtain a 3CY category, we take the 3CY completion $\tilde{\CA}=\Pi_3(\CA)$ in the sense of \cite{keller2011deformed}. It is shown in \cite{bozec2020relative} that the shifted cotangent stack $\bs{\tilde{\FM}}_m=\T^*[-1]\bs{\FM}_m$ is equivalent to a substack of the (derived) stack of objects in $\tilde{\CA}$. Being the classical truncation of a $(-1)$-shifted cotangent stack, $\tilde{\FM}_{m}=t_0(\bs{\tilde{\FM}}_m)$ carries a canonical orientation, and indeed a canonical d-critical structure, determining the mixed Hodge module $\phi_{\DT}$ on it; see \cite{kinjo2022dimensional} for the details. In [loc. cit] it is moreover proved that there is a canonical \emph{dimensional reduction} isomorphism, for each $m\in\K(\CA)$:
\[
\HO(\tilde{\FM}_m,\phi_{\DT})\cong \HO^{\BoMo}(\FM_{m},\BoQ^{\vir})
\]
where the $\vir$ superscript denotes a system of Tate twists that we do not make explicit in this introduction (see \S\ref{subsection:theCoHA} for details).

A foundational/definitional conjecture in Donaldson--Thomas theory is the \emph{integrality conjecture}. In the enumerative version of the theory this is the claim that certain partition functions with coefficients given by DT invariants are obtained by taking plethystic exponentials of partition functions involving simpler invariants, called BPS invariants: see \cite{kontsevich2008stability} for an exact formulation in the enumerative setup. This enumerative version is implied by a stronger \emph{cohomological} version of the integrality theorem:
\begin{conjecture}
\label{conjecture:3CYIntegrality}
Let $\CC$ be a 3CY Abelian category on which the Euler form vanishes and $\cl\colon \pi_0(\FM_{\CC})\rightarrow \K(\CC)$ is a surjective morphism to a monoid. Then there is some isomorphism of $\K(\CC)$-graded, cohomologically graded mixed Hodge structures 
\[
\bigoplus_{m\in \K(\CC)}\HO(\FM_{\CC,m},\underline{\phi}_{\DT})\cong \Sym\left(\bigoplus_{0\neq m\in\K(\CC)} \ul{\aBPS}_{\CC,m}\otimes\HO(\B \BoC^*,\ul{\BoQ})\right)
\]
where each $\ul{\aBPS}_{\CC,m}$ has bounded cohomological amplitude, and finite total dimension.
\end{conjecture}

Via dimensional reduction, we obtain from Theorem~\ref{theorem:relPBW} our main result in enumerative geometry:
\begin{theorem}
Conjecture~\ref{conjecture:3CYIntegrality} is true for 3CY categories of the form $\tilde{\CA}=\Pi_3(\CA)$ for $\CA$ a 2-Calabi--Yau category satisfying the assumptions of \S\ref{section:modulistack2dcats}. In particular, if $X=\Tot(K_S)\cong S\times \BoA^1$ for $S$ a smooth quasiprojective surface satisfying $\K_S\cong \mathscr{O}_S$, the cohomological integrality conjecture holds for the category of Gieseker semistable\footnote{With respect to a polarisation of a compactification of $S$.} compactly supported coherent sheaves of fixed normalised Hilbert polynomial on $X$.
\end{theorem}
In the case $\CA=\Higgs_{\theta}(C)$, the category of semistable Higgs sheaves of fixed slope $\theta$ on a smooth projective curve $C$, $\tilde{\CA}=\Pi_3(\CA)$ can be identified with the category of semistable coherent sheaves of pure dimension one and slope $\theta$ on $\Tot_C(\mathscr{O}_C\oplus K_C)$, and Conjecture~\ref{conjecture:3CYIntegrality} is proved in \cite{kinjo2021cohomological}, building on \cite{kinjo2021global,maulik2020cohomological}. In the case in which $g(C)\geq 2$, the precise determination of $\ul{\aBPS}_{\tilde{\CA}}$ in terms of intersection cohomology was obtained in our earlier work \cite{davison2022BPS}, while for $g(C)\leq 1$ it can be found in \cite{davison2021nonabelian}.
\begin{remark}
We may alternatively \emph{define} $\ul{\aBPS}_{\CC,m}\coloneqq \HO(\CM_{\CC,m},\CH^{1}((\FM_{\CC}\rightarrow \CM_{\CC})_*\ul{\phi}_{\DT}))$ and make the stronger conjecture that for this choice of mixed Hodge structure, there is an isomorphism as in Conjecture~\ref{conjecture:3CYIntegrality}. The local version of this conjecture (for Jacobi algebras) is the main result of \cite{davison2020cohomological}; see \S \ref{subsec_3dcomp} for this special case.
\end{remark}

\subsection{Lowest weight modules from Nakajima quiver varieties}
Part of the inspiration for constructing generalised Kac--Moody Lie algebras out of 2CY categories was the famous work of Nakajima on irreducible highest/lowest weight representations constructed out of cohomology of quiver varieties, and work of Nakajima and Grojnowski on irreducible highest/lowest weight representations of Heisenberg algebras on cohomology of Hilbert schemes. In this section we explain how we recover and extend this theory.

Given a quiver $Q$ without loops, and a framing dimension vector $\ff\in\BoN^{Q_0}$, Nakajima constructs in \cite{nakajima1998quiver} an action of the Kac--Moody Lie algebra $\Fg_{\Gamma}$ on $\bigoplus_{\dd\in\BoN^{Q_0}}\HO^{0}(N(Q,\dd,\ff),\BoQ^{\vir})$ where $\Gamma$ is the underlying graph of $Q$, $\Fg_{\Gamma}$ is the associated Kac--Moody Lie algebra, and $N(Q,\dd,\ff)$ are Nakajima quiver varieties (we recall the definition in \S\ref{NQVs_def}, along with the cohomological shifts bundled into the superscript $\vir$). Furthermore, this representation is shown to be the irreducible lowest weight representation of lowest weight $-\sum_{i\in Q_0} f_i\delta_i$, where the $\delta_i$s denote the fundamental weights. Note that the same correspondences considered in \cite{nakajima1998quiver} act on the entire cohomology $\mathbb{M}_{\ff}(Q)\coloneqq\bigoplus_{\dd\in\BoN^{Q_0}}\HO^*(N(Q,\dd,\ff),\BoQ^{\vir})$, suggesting the following question:
\begin{question}
\label{question1}
For fixed $Q$ and $\ff\in \BoN^{Q_0}$, what is the decomposition of the \emph{entire} cohomology $\mathbb{M}_{\ff}(Q)$ into irreducible lowest weight $\Fg_{\Gamma}$-modules?
\end{question}
Next, let $Q$ be instead the Jordan quiver; this is the quiver with one vertex and one loop. We fix the framing vector $f=1$; then there is an isomorphism $N(Q,d,1)\cong \Hilb_d(\BoA^2)$. Again, $\mathbb{M}_{1}(Q)$ is acted on by an algebra of correspondences (Nakajima's raising and lowering operators), and it is proved in \cite{nakajima1997heisenberg,nakajima1999lectures} that the resulting Lie algebra of correspondences $\Fg$ is an infinite-dimensional Heisenberg algebra. Moreover, it is proven that $\mathbb{M}_{1}(Q)$ is an irreducible lowest weight module for this Lie algebra. Note that this is a long way from being a special case of the result in the paragraph above Question \ref{question1}, since $Q$ \textit{only} contains a loop. On the other hand, the algebra $\Fg$ is a (very degenerate example of a) \emph{generalised} Kac--Moody Lie algebra, containing \emph{only} imaginary simple roots.

These two central results in the geometric representation theory of quiver varieties motivate the following more ambitious question:
\begin{question}
\label{question2}
For fixed $Q$, is there a natural generalised Kac--Moody Lie algebra $\tilde{\Fg}_Q$ of correspondences, containing the Kac--Moody Lie algebra associated to the real subquiver (see \S \ref{subsection:QandR}), for which, for fixed $\ff\in\BoN^{Q_0}$, the module $\mathbb{M}_{\ff}(Q)$ admits a concrete decomposition into irreducible lowest weight modules?
\end{question}
Although Question \ref{question2} is more subjective, we still venture to answer it with a ``yes'' in this paper.  Firstly, we propose the (full) BPS Lie algebra $\Fg^{\ttBPS}_{\Pi_Q}$ introduced at \eqref{definition:fullBPSLA} to be the Lie algebra $\tilde{\Fg}_Q$. Then the decomposition into irreducible lowest weight representations, with lowest weight spaces identified with intersection cohomology of certain singular moduli spaces of framed representations, is provided by the following theorem, a more general version of which is proved as Theorem~\ref{theorem:mainRTthm}:
\begin{theorem}
\label{rep_thy_thm}
The $\Fn_{\Pi_Q}^{\ttBPS,+}$-action on $\mathbb{M}_{\ff}(Q)$ extends to a $\Fg_{\Pi_Q}^{\ttBPS}$-action, for which $\mathbb{M}_{\ff}(Q)$ decomposes as a direct sum of cohomologically shifted simple lowest weight representations, with lowest weight spaces identified with intersection cohomology of singular quiver varieties:
\begin{equation*}
\mathbb{M}_{\ff}(Q)\cong \bigoplus_{(\dd,1)\in\Phi_{Q_{\ff}}^+}\ICA\!(\CM_{\Pi_{Q_{\ff}},(\dd,1)})\otimes L_{((\dd,1),(-,0))_{Q_{\ff}}}.
\end{equation*}
Here, $\ICA$ is the derived global section functor applied to the perverse IC complex, and $L_{\lambda}$, with $\lambda$ a linear form on $\BoZ^{Q_0}$, denotes the irreducible lowest weight representation of $\Fg_{\Pi_Q}^{\ttBPS}$ of lowest weight $\lambda$ (see \S\ref{subsection:lowestweightmodule}).
\end{theorem}
We explain how to recover the above results of \cite{nakajima1997heisenberg,nakajima1998quiver,grojnowski1996instantons} from Theorem~\ref{theorem:mainRTthm} in \S \ref{subsec:GN}.
\begin{remark}
A consequence of our results is that $\Fg_{\Pi_Q}^{\ttBPS}$ acts faithfully on the direct sum of modules $\mathbb{M}_{\ff}(Q)$ provided by taking the cohomology of Nakajima quiver varieties for various framing vectors $\ff$. In \cite{maulik2019quantum} Maulik and Okounkov associate to each quiver $Q$ a Lie algebra $\Fg_{\MO,Q}$ (after extending scalars from $\BoQ$ to an equivariant cohomology ring). By construction, their Lie algebra acts faithfully on cohomology of Nakajima quiver varieties, and admits a $\BoZ^{Q_0}$-graded weight space decomposition. Furthermore, $\Fg_{\MO,Q}$ is conjectured to be a generalised Kac--Moody Lie algebra, is conjectured to be already defined over $\BoQ$, and is known to contain the usual Kac--Moody Lie algebra associated to $Q$. The main results of this paper, then, should be taken as strong evidence for the conjectures (in \cite{maulik2019quantum} and \cite{davison2020bps} respectively) that $\Fg_{\MO,Q}$ can be defined over $\BoQ$, and (after extending scalars) there is an isomorphism $\Fn^{\ttBPS,+}_{\Pi_Q}\cong\Fn^+_{\MO,Q}$.
\end{remark}
\begin{remark}
These conjectures have now been proved in \cite{botta2023okounkov}. The proof uses in an essential way the description of the BPS Lie algebra as a generalised Kac--Moody Lie algebra which forms Theorem \ref{corollary:absBor}, as well as the description of its representation theory contained in Theorem \ref{rep_thy_thm}.  These conjectures in turn imply Okounkov's conjecture equating the graded dimensions of $\Fg_{\MO,Q}$ with coefficients of Kac polynomials, see \cite{botta2023okounkov, SVMOproof}. The conjecture in \cite{maulik2019quantum} that the Maulik--Okounkov Lie algebra is defined over $\BoQ$ is reduced to our Corollary \ref{triv_ext_BPS} by \cite{botta2023okounkov} and the present paper.
\end{remark}

\subsection{Comparison with the 3d BPS Lie algebra}
\label{subsec_3dcomp}
If $\CA=\Rep\Pi_Q$, a definition of the BPS Lie algebra sheaf associated to the vanishing cycle cohomology of the stack of $\Pi_Q[x]$-modules is given in \cite{davison2020bps} using the fact that $\Pi_Q[x]$ is the $3$-Calabi--Yau Jacobi algebra associated to a quiver with potential. CoHAs for vanishing cycle cohomology of Jacobi algebras were defined and studied in \cite{kontsevich2010cohomological}. We denote this BPS Lie algebra sheaf by $\BPS_{\Pi_Q,\Lie}^{\td,\psi}\in\Perv(\CM_{\Pi_Q})$. The superscript $\td$ indicates that the definition involves dimensional reduction from the stack of representations of a Jacobi algebra, which we think of as a kind of noncommutative threefold. This sheaf is the direct image of the BPS sheaf defined in the context of quivers with potential in \cite{davison2020cohomological}, and this approach to defining BPS sheaves is very much inspired by the theory recalled in \S \ref{section:DTcrashcourse}. 

Although the BPS Lie algebra sheaf (for an arbitrary 2CY category) in this paper is defined very differently, for the category of $\Pi_Q$-modules we nonetheless prove the following in \S \ref{subsection:2d3d} (Proposition~\ref{proposition:BPS3dcoincide}):
\begin{theorem}
\label{thm:comp}
 Let $Q$ be a quiver. The relative BPS Lie algebras $\BPS_{\Pi_Q,\Lie}$ and $\BPS_{\Pi_Q,\Lie}^{\td,\psi}$ are canonically isomorphic.
\end{theorem}
In particular, the BPS Lie algebra sheaf $\BPS_{\Pi_Q,\Lie}^{\td,\psi}$ is (up to isomorphism) independent of the chosen $\psi$-twist.

\subsection{Acknowledgements}
All three authors were supported by the European Research Council starter grant “Categorified Donaldson--Thomas theory” No. 759967. BD was in addition supported by a Royal Society University Research Fellowship. SSM was in addition supported by the Max Planck Institute for Mathematics and the Swiss National Science Foundation [no. 218340]. We would like to thank Tasuki Kinjo and Davesh Maulik for useful and enlightening discussions. 

\subsection{Standard notations and conventions}
\label{notations_subsec}
\begin{enumerate}[$\bullet$]
 \item If $M$ is a monoid, we denote by $\widehat{M}$ its groupification. For $(l,m)\in \BoN\times M$, we denote by $lm$ the sum of $l$ copies of $m$.
 \item We denote by $1_i$ the $i$th basis vector of the monoid $\BoN^{Q_0}$, where $i\in Q_0$.
 \item We denote by $\HO^*_{\BoC^*}$ the $\BoC^*$-equivariant cohomology ring of the point $\pt$ with coefficients in $\BoQ$: $\HO^*_{\BoC^*}\coloneqq \HO^*_{\BoC^*}(\pt,\BoQ)\cong \HO(\B\BoC^*,\BoQ)$. We denote $\ul{\HO}^*_{\BoC^*}\coloneqq \HO^*_{\BoC^*}(\pt,\ul{\BoQ})$ the corresponding pure mixed Hodge structure.
 \item Stacks are usually denoted using fraktur letters: $\mathfrak{X}$. If a stack $\FX$ has a good moduli space \cite{alper2013good}, it will be denoted using calligraphic letters: e.g. $\FX\rightarrow \CX$ denotes the morphism to the good moduli space.
 \item If $\SF$ is a complex of constructible sheaves on a space $X$ we will frequently abbreviate $\HO^*\!\SF\coloneqq \HO^*(X,\SF)$.
\item If $X$ is a smooth scheme or stack, we set $\BoQ_{X}^{\vir}=\bigoplus_{X'\in \pi_0(X)}\BoQ_{X'}[\dim X']$, the constant perverse sheaf. If each connected component of $X$ has even dimension, we set $\ulBoQ_X^{\vir}=\bigoplus_{X'\in \pi_0(X)}\ulBoQ_{X'}\otimes \BoL^{-\dim(X')/2}$, where $\BoL=\HO_{\mathrm{c}}(\BoA^1,\BoQ)$. This is a mixed Hodge module of pure weight zero.
 \item If $X$ is an irreducible complex algebraic variety of even dimension, the pure weight zero mixed Hodge module $\ul{\BoQ}_{X^{\sm}}^{\vir}=\ul{\BoQ}_{X^{\sm}}\otimes\BoL^{-\dim(X)/2}$ on $X^{\sm}$ has an essentially unique extension to a pure weight zero \emph{simple} mixed Hodge module on $X$. We denote it by $\underline{\IC}(X)$.
 \item If $X$ is a complex, separated, locally finite type scheme, we denote by $\Perv(X)$, $\CD^{\rmb}_{\rmc}(X)$, $\MHM(X)$, $\CD^{\rmb}(\MHM(X))$ the categories of perverse sheaves, mixed Hodge modules and their derived categories on the underlying reduced scheme $X_{\red}$. We let $\CD^{+}_{\rmc}(X)$ and $\CD^+(\MHM(X))$ denote the categories of locally bounded below complexes.
 \item We have the natural truncation functor $\tau^{\leq 0}\colon \CD^+(\MHM(X))\rightarrow \CD^{+}(\MHM(X))$. 
 \item For a complex of mixed Hodge modules $\SF\in\CD^+(\MHM(X))$, we denote by $\CH^i(\SF)$ its $i$th cohomology. It is a mixed Hodge module on $\CM$.
 \item If $M$ is a monoid, $A$ a $M$-graded algebra and $\psi\colon M\times M\rightarrow \BoZ/2\BoZ$ a bilinear form, we let $A^{\psi}$ be the $\psi$-twisted algebra. As $M$-graded vector spaces, $A=A^{\psi}$ and if $m\colon A\otimes A\rightarrow A$ is the multiplication of $A$, then the restriction of the multiplication on $A^{\psi}$ is $m_{a,b}^{\psi}\coloneqq(-1)^{\psi(a,b)}m_{a,b}\colon A_a\otimes A_b\rightarrow A_{a+b}$ for any $a,b\in M$.
 \item If a statement concerning a $\psi$-twisted algebra $A^{\psi}$ in some monoidal category does not actually depend on the algebra structure (and therefore, the $\psi$-twist does not play any role), we write $A^{(\psi)}$. For example, the $M$-graded character of $A$ or the purity (if $A$ is a mixed Hodge module or mixed Hodge structure) are properties independent of the multiplication.
 \item If $V$ is a $M\times \BoZ$ graded vector space (i.e. for each $m\in M$, the vector space $V_m$ is cohomologically graded), we denote by $\ch^{\BoZ}(V)$ its character: $\ch^{\BoZ}(V)=\sum_{n\in\BoZ}\sum_{m\in M}\dim(V_m^n)q^{n/2}z^{m}\in\BoN[[M\times\BoZ]]$.
 \item For $l\in\BoZ_{\geq 1}$, we let $S^l$ be the symmetric group on a set with $l$ elements. 
 \item If $X$ is an algebraic variety with an action of an algebraic group $G$, the stack quotient is simply denoted by $X/G$ (in the literature, the notation $[X/G]$ is also used).
 \item If $X$ is a scheme, we let $\Gamma(X)$ be the ring of global functions on $X$.
 \item If $A$ is a ring with an action of a group $G$ by automorphisms of ring, we let $A^G$ be the ring of invariants.
\end{enumerate}

\section{Quivers and representations}
\label{subsection:QandR}
A quiver is a directed graph, i.e. a pair of sets $Q_1$ and $Q_0$ of arrows and vertices, respectively, along with a pair of maps $s,t\colon Q_1\rightarrow Q_0$ taking an arrow to its source and target, respectively. We always assume that $Q_1$ and $Q_0$ are finite.

We allow our quivers to have loops (i.e. 1-cycles). Let $Q_0^{\im}\subset Q_0$ be the subset of vertices supporting loops, which we call \emph{imaginary vertices}. We denote by $Q^{\real}\subset Q$ the quiver obtained by replacing $Q_0$ with $Q_0\setminus Q_0^{\im}$, and removing all arrows from $Q_1$ that have source or target in $Q_0^{\im}$.

We denote by 
\begin{align*}
\chi_Q(- ,-)\colon &\BoZ^{Q_0}\times\BoZ^{Q_0}\rightarrow \BoZ\\
&(\dd',\dd'')\mapsto \sum_{i\in Q_0} d'_{i}d''_{i}-\sum_{a\in Q_1} d'_{s(a)}d''_{t(a)}
\end{align*}
the Euler form of $Q$, and define $(-,-)_Q$ to be its symmetrisation. For $\dd',\dd''\in\BoN^{Q_0}$, we define the \emph{dot product} by $\dd'\cdot\dd''=\sum_{i\in Q_0}d'_id''_i$.

Given a quiver $Q$ we denote by $\BoC Q$ the \emph{free path algebra}, i.e. the algebra having a $\BoC$-basis given by paths in $Q$, including paths $\lazy_i$ of length zero at each of the vertices $i$. The multiplication, with respect to this basis, is given by concatenation of paths.

If $A=\BoC Q/R$ is presented as the quotient of a free path algebra by a two sided ideal generated by linear combinations of paths of length of at least one, and $\rho$ is an $A$-module, we define the \emph{dimension vector} 
\[
\dim_{Q_0}(\rho)\coloneqq(\dim(\lazy_i\cdotsh \rho))_{i\in Q_0}\in \BoN^{Q_0}.
\]
For a quiver $Q$ and dimension vector $\dd\in\BoN^{Q_0}$ we define 
\[
\BoA_{Q,\dd}\coloneqq \prod_{a\in Q_1}\Hom(\BoC^{d_{s(a)}},\BoC^{d_{t(a)}})
\]
which is acted on by the gauge group
\[
\GL_{\dd}\coloneqq\prod_{i\in Q_0}\GL_{d_i}(\BoC)
\]
by change of basis. We denote by $\FM_{Q,\dd}$ the stack of $\dd$-dimensional $\BoC Q$-modules, and by $\JH_{\dd}\colon \FM_{Q,\dd}\rightarrow \CM_{Q,\dd}$ the canonical morphism to the affinization. There is an equivalence of stacks, and an isomorphism of affine varieties, respectively
\[
\FM_{Q,\dd}\simeq \BoA_{Q,\dd}/\GL_{\dd};\quad\quad\CM_{Q,\dd}\cong \Spec(\Gamma(\BoA_{Q,\dd})^{\GL_{\dd}}).
\]
The variety $\CM_{Q,\dd}$ is the coarse moduli space of $\dd$-dimensional $\BoC Q$-modules, and its points correspond to semisimple $\dd$-dimensional modules. We set $\FM_Q\coloneqq \coprod_{\dd\in\BoN^{Q_0}}\FM_{Q,\dd}$ and $\CM_Q\coloneqq \coprod_{\dd\in\BoN^{Q_0}}\CM_{Q,\dd}$. There is a direct sum operation
\[
\oplus\colon \CM_{Q}\times \CM_{Q}\rightarrow \CM_Q
\]
making $\CM_Q$ into a commutative monoid in schemes. This morphism is finite \cite[Lemma 2.1]{meinhardt2019donaldson}. 

It will often be convenient to consider some submonoid $\CM_{Q}^{\triangle}\subset \CM_{Q}$, which we will assume is saturated, meaning that the following diagram is Cartesian
\[
 \begin{tikzcd}
	\CM_Q^{\triangle}\times \CM_Q^{\triangle}& \CM_Q^{\triangle} \\
	\CM_Q\times \CM_Q& \CM_Q
	\arrow["\imath_{\triangle}\times \imath_{\triangle}"',hook,from=1-1, to=2-1]
	\arrow["i_{\triangle}",hook,from=1-2, to=2-2]
	\arrow["\oplus",from=1-1, to=1-2]
	\arrow["\oplus",from=2-1,to=2-2]
	\arrow["\lrcorner"{anchor=center, pos=0.125}, draw=none, from=1-1, to=2-2]
\end{tikzcd}
\]
where $\imath_{\triangle}\colon \CM^{\triangle}_{Q}\hookrightarrow \CM_Q$ is the inclusion, and we denote again by $\oplus$ the restriction of $\oplus$ to $\CM^{\triangle}_{Q}\times \CM^{\triangle}_{Q}$.

The connected components $\pi_0(\CM^{\triangle}_Q)$ form a monoid, and we fix a morphism of monoids $\cl\colon \pi_0(\CM^{\triangle}_Q)\rightarrow \K$ where $\K$ is a finitely generated torsion-free, commutative, cancellative monoid; equivalently, the natural map $\K\rightarrow \widehat{\K}$ to the groupification of $\K$ is injective, and the target is a lattice. We assume also that the Euler form $\chi_Q(-,-)$, which automatically descends to a bilinear form on $\pi_0(\CM^{\triangle}_Q)$, descends to a bilinear form on $\K$. The standard choice of $\K$ is $\BoN^{Q_0}$, with $\cl$ taking a connected component $\CM\subset \CM_{Q,\dd}^{\triangle}$ to $\dd$.

Given a quiver $Q$ we denote by $\overline{Q}$ the \emph{doubled quiver}, i.e. the quiver obtained by adding a new arrow $a^*$ for each arrow $a\in Q_1$, with $s(a^*)=t(a)$ and $t(a^*)=s(a)$. We define the \emph{preprojective algebra}
\[
\Pi_Q\coloneqq \BoC \overline{Q}/\big\langle \sum_{a\in Q_1}[a,a^*]\big\rangle.
\]

We denote by $\tilde{Q}$ the \emph{tripled quiver} obtained from $\overline{Q}$ by adding a loop $\omega_i$ based at $i$, for each vertex $i\in Q_0$. 
We denote by $\SG_2(\BoC Q)$ the \emph{derived preprojective algebra}, defined as follows. We give $\BoC \tilde{Q}$ a cohomological grading, by putting $a$ and $a^*$ in cohomological degree $0$, and each $\omega_i$ in cohomological degree $-1$. We define a differential on $\BoC\tilde{Q}$ by setting $d\omega_i=\lazy_i\cdotsh\sum_{a\in Q_1}[a,a^*]\cdotsh \lazy_i$ and extending via the Leibniz rule. Then $\HO^0(\SG_2(\BoC Q))\cong \Pi_Q$ and we identify the category of $\Pi_Q$-modules with the full subcategory of $\SG_2(\BoC Q)$-modules concentrated in cohomological degree $0$. For $M$ and $N$ two finite-dimensional $\Pi_Q$-modules there is an equality
\[
(M,N)_{\SG_2(\BoC Q)}=(\dim_{Q_0}(M),\dim_{Q_0}(N))_Q
\]
between the Euler form in the category of $\SG_2(\BoC Q)$-modules, and the symmetrised Euler form.

We define $\mathfrak{gl}_{\dd}\coloneqq\prod_{i\in Q_0} \mathfrak{gl}_{d_{i}}(\BoC)$, the Lie algebra of $\GL_{\dd}$. We identify $\BoA_{\overline{Q},\dd}=\mathrm{T}^*\BoA_{Q,\dd}$ via the trace pairing. Then the $\GL_{\dd}$-action is Hamiltonian, with (co)moment map
\begin{equation}
\label{mmap}
\begin{matrix}
\mu_{Q,\dd}\colon &\BoA_{\overline{Q},\dd}&\rightarrow& \mathfrak{gl}_{\dd}\\
&(\rho(a),\rho(a^*))_{a\in Q_1}&\mapsto& \sum_{a\in Q_1}[\rho(a),\rho(a^*)].
\end{matrix}
\end{equation}
The zero set $\mu_{Q,\dd}^{-1}(0)\subset \BoA_{\overline{Q},\dd}$ is identified with the subspace of $\Pi_Q$-modules. There is an equivalence
\[
\FM_{\Pi_Q,\dd}\simeq \mu_{Q,\dd}^{-1}(0)/\GL_{\dd}
\]
between the stack of $\dd$-dimensional $\Pi_Q$-modules and the stack-theoretic quotient of the zero-locus of the moment map.

\section{Generalized Kac--Moody algebras}
\label{section:GKM}
In this section, we give the constructions of the generalised Kac--Moody algebra associated to a monoid with bilinear form and a weight function following \cite{borcherds1988generalized}, and of generalised half Kac--Moody algebras in symmetric tensor categories of mixed Hodge modules/perverse sheaves. We work with monoids with bilinear forms to allow sufficient flexibility when investigating the BPS algebras of $2$-Calabi--Yau categories.

\subsection{Roots}
\label{subsection:roots}

\subsubsection{Primitive and positive roots}

Let $(M,+)$ be a monoid and $M\times M\rightarrow\BoZ$ be a bilinear form. We let
\begin{multline}
 \Sigma_{M,(-,-)}=\left\{m\in M\mid 2-(m,m)\geq 0 \text{ and } 2-(m,m)>\sum_{j=1}^s(2-(m_j,m_j))\right.\\ \left.\text{ for any nontrivial decomposition $m=\sum_{j=1}^sm_j$ with $m_j\in M\setminus \{0\}$}\right\}
\end{multline}
be the set of \emph{primitive positive roots}, and 
\[
 \Phi_{M,(-,-)}^+=\Sigma_{M,(-,-)}\cup\{lm\mid l\in\BoN_{\geq 1}, m\in\Sigma_{M,(-,-)}\text{ and }(m,m)=0\}
\]
be the set of \emph{simple positive roots}. We assume that $(m,m)\in 2\BoZ_{\leq 1}$ for $m\in\Phi_{M,(-,-)}^+$. We write the decomposition of $\Phi_{M,(-,-)}^+$ into real and imaginary roots, $\Phi_{M,(-,-)}^+=\Phi_{M,(-,-)}^{+,\real}\sqcup \Phi_{M,(-,-)}^{+,\im}$, where $\Phi_{M,(-,-)}^{+,\real}\coloneqq\{m\in\Phi_{M,(-,-)}^+\mid (m,m)=2\}$ and $\Phi_{M,(-,-)}^{+,\im}\coloneqq\{m\in\Phi_{M,(-,-)}^+\mid (m,m)\leq 0\}$. We can further decompose $\Phi_{M,(-,-)}^{+,\im}=\Phi_{M,(-,-)}^{+,\iso}\sqcup\Phi_{M,(-,-)}^{+,\hyp}$ where 
\begin{align*}
\Phi_{M,(-,-)}^{+,\iso}\coloneqq\{m\in\Phi_{M,(-,-)}^+\mid (m,m)=0\}\,;\quad
\Phi_{M,(-,-)}^{+,\hyp}\coloneqq\{m\in\Phi_{M,(-,-)}^+\mid (m,m)<0\}
\end{align*}
are the sets of isotropic and hyperbolic roots, respectively. For $m\in \Phi_{M,(-,-)}^{+,\iso}\cap \Sigma_{M,(-,-)}$ we call $m$ an \emph{indivisible isotropic root}. 

\begin{remark}
When considering Weyl group actions we assume that $M$ is finitely generated, commutative, torsion-free, and cancellative; equivalently, we require that the natural map $M\rightarrow \widehat{M}$ to the groupification is injective, and $\widehat{M}$ is a lattice. All definitions of roots make sense without these assumptions (which may not be satisfied by the monoid of connected components of the good moduli space of 2CY categories).
\end{remark}

We fix an Abelian category $\CA$, and an ambient dg subcategory $\SC$ which contains $\CA$ as a full subcategory of its homotopy category, with stack of objects $\FM_{\CA}$, such that the Euler form $(-,-)_{\SC}$ is locally constant on $\FM_{\CA}\times\FM_{\CA}$. We assume that we have a chosen monoid $\K(\CA)$ and a fixed surjective morphism of monoids
\[
\pi_0(\FM_{\CA})\rightarrow \K(\CA).
\]
As in \S \ref{subsubsection:pospartGKM} we assume that the natural morphism to the groupification $\K(\CA)\rightarrow \widehat{\K(\CA)}$ is injective, $\widehat{\K(\CA)}$ is a lattice, and the Euler form $(-,-)_{\SC}$ descends to a bilinear form on $\K(\CA)$ (and thus induces a bilinear form on $\widehat{\K(\CA)}$). We denote by $\Sigma_{\K(\CA)}$ and $\Phi_{\K(\CA)}^+$ the sets of primitive positive roots and simple positive roots, respectively, associated to the monoid $\K(\CA)$ endowed with the Euler form.

\begin{remark}
 The hypotheses on $(m,m)$ and the definition of $\Phi_{M,(-,-)}^{+,\real}$ are designed for the study of the BPS algebras of $2$-Calabi--Yau categories. It is certainly possible to weaken or generalise them to combinatorially define a broader class of GKM algebras. A possible different choice would be
\begin{align}
 \Phi^{+,\real}_{M,(-,-)}\coloneqq \{m\in M\mid (m,m)>0\}\\
 \frac{2(m,n)}{(m,m)}\in\BoZ_{\leq 0} \text{ if $m\in\Phi^{+,\real}_{M,(-,-)}$ and $n\in \Phi^+_{M,(-,-)}$}.
\end{align}
This would be useful if we had positive roots $m$ with $(m,m)>2$ but this is not the case for 2CY categories.
\end{remark}

\subsection{Generalised Kac--Moody algebras for monoids with bilinear form}
\label{subsection:GKMmonoids}
Let $M$ be a monoid with bilinear form $(-,-)$. Recall the definitions of the sets of primitive positive roots and simple positive roots from \S\ref{subsection:roots}. We let $A_{M,(-,-)}=(a_{m,n}=(m,n))_{m,n\in\Phi_{M,(-,-)}^+}$ be the \emph{Cartan matrix}. We assume that $A_{M,(-,-)}$ is symmetric, $a_{m,n}\leq 0$ if $m\neq n$ and $a_{m,m}=2$ if $a_{m,m}>0$. We let
\[
\begin{matrix}
 P&\colon& \Phi_{M,(-,-)}^+&\rightarrow& \BoN[t^{\pm 1/2}]\\
 &   & m   &\mapsto  & P_m(t^{1/2})
\end{matrix}
\]
be a weight function: we assume that $P_m(t)=1$ if $a_{m,m}=2$. We write $P_m(t^{1/2})=\sum_{j\in\BoZ}p_{m,j}t^{j/2}$, with $p_{m,j}\in\BoN$.

We let $\mathfrak{g}_{M,(-,-),P}$ be the super Lie algebra generated by $e_{m,j,l}, f_{m,j,l}, h_m$, with $m\in\Phi_{M,(-,-)}^+$, $j\in\BoZ$, and $1\leq l\leq p_{m,j}$, with the relations\footnote{Since we assume that the only strictly positive diagonal entries in the Cartan matrix are $2$, the exponents in the generalised Serre relations simplify compared to the definition for Cartan matrices under weaker conditions.}
\[
 \begin{aligned}
h_{m+n}&=h_{m}+h_{n}\\
 [h_m,h_n]&=0\\
 [h_m,e_{n,j,l}]&=a_{m,n}e_{n,j,l}&\\
 [h_m,f_{n,j,l}]&=-a_{m,n}f_{n,j,l}&\\
 [e_{m,j,l},f_{n,j',l'}]&=\delta_{m,n}\delta_{j,j'}\delta_{l,l'}h_n&\\
 \ad(e_{m,j,l})^{1-a_{m,n}}(e_{n,j',l'})&=0&\quad \text{ if $a_{m,m}=2$}\\
 \ad(f_{m,j,l})^{1-a_{m,n}}(f_{n,j',l'})&=0&\quad \text{ if $a_{m,m}=2$}\\
 [e_{m,j,l},e_{n,j',l'}]=[f_{m,j,l},f_{n,j',l'}]&=0 &\quad \text{ if $a_{m,n}=0$}
 \end{aligned}
\]
where $\delta$ denotes the Kronecker delta function. We place each $e_{m,j,l}$ in cohomological degree $j$, each $f_{m,j,l}$ in cohomological degree $-j$, and $h_m$ in cohomological degree $0$. Then clearly the above relations are homogeneous, so that the quotient Lie algebra inherits a cohomological grading. Note that because we allow generators of odd cohomological degree, signs appear due to the Koszul sign rule. For example the antisymmetry and Jacobi relations become
\[
[\alpha,\beta]+(-1)^{ab}[\beta,\alpha]=0;\quad\quad [\alpha,[\beta,\gamma]]+(-1)^{ab+ac} [\beta,[\gamma,\alpha]]+(-1)^{ac+bc}[\gamma,[\alpha,\beta]]
\]
respectively, where $\alpha,\beta$, and $\gamma$ are homogeneous of cohomological degree $a,b$ and $c$ respectively.  Note that the BPS Lie algebra associated to the category of representations of a preprojective algebra is concentrated in even cohomological degrees, so the Koszul sign rule will not appear for this set of examples.

We let $\mathfrak{n}^+_{M,(-,-),P}$ be the Lie subalgebra generated by $e_{m,j,l}$, for $m\in \Phi_{M,(-,-)}^+$, $j\in\BoZ$, $1\leq l\leq p_{m,j}$, $\mathfrak{n}^-_{M,(-,-),P}$ be the sub-Lie algebra generated by $f_{m,j,l}$, for $m\in \Phi_{M,(-,-)}^+$, $j\in\BoZ$, $1\leq l\leq p_{m,j}$, and $\mathfrak{h}_M$ be the sub-Lie algebra generated by $h_m$, $m\in M$. We have the triangular decomposition \cite{borcherds1988generalized}
\[
 \mathfrak{g}_{M,(-,-),P}=\mathfrak{n}^-_{M,(-,-),P}\oplus \mathfrak{h}_M\oplus \mathfrak{n}^+_{M,(-,-),P}.
\]
By considering the associative algebra with generators $e_{m,j,l}, f_{m,j,l}, h_m$ with the same relations, we obtain the enveloping algebra $\UEA(\mathfrak{g}_{M,(-,-),P})$ of $\mathfrak{g}_{M,(-,-),P}$. It has the triangular decomposition
\[
\UEA(\mathfrak{g}_{M,(-,-),P})\cong \UEA(\mathfrak{n}^-_{M,(-,-),P})\otimes \UEA(\mathfrak{h}_M)\otimes \UEA(\mathfrak{n}^+_{M,(-,-),P}).
\]

If the function $P$ is given by the graded dimensions of a family of $\BoZ$-graded vector spaces $V=(V_m)_{m\in\Phi_{M,(-,-)}^+}$ such that for $m\in\Phi_{M,(-,-)}^{+,\real}$, $V_m\cong\BoQ[0]=\BoQ e_{m,0,1}$ is a one-dimensional vector space in cohomological degree $0$, and $V_m=0$ for $m\notin \Phi^+_{M,(-,-)}$, the relations
\[
\begin{aligned}
\ad(e_{m,j,l})^{1-a_{m,n}}(e_{n,j',l'})&=0&\quad \text{ if $a_{m,m}=2$}\\
[e_{m,j,l},e_{n,j',l'}]&=0&\quad \text{ if $a_{m,n}=0$}
\end{aligned}
\]
for some graded bases $e_{m,j,l}$ of $V_m$ are called \emph{Serre relations associated to $V$}. We let $R_{m,n}$ be the subspace of the free associative algebra $\Free(V)$ generated by $V$, linearly generated by $\ad(e_{m,j,l})^{1-a_{m,n}}(e_{n,j',l'})$ ($j,j'\in\BoZ$, $1\leq l\leq p_{m,j}$, $1\leq l'\leq p_{n,j'}$) if $a_{m,m}=2$ and $[e_{m,j,l},e_{n,j',l'}]$ ($j,j'\in\BoZ$, $1\leq l\leq p_{m,j}$, $1\leq l'\leq p_{n,j'}$) if $a_{m,n}=0$. It is clear that $R_{m,n}$ does not depend on the choice of the graded bases of $V_m$ and $V_n$. If $I_{M,V,\Alg}$ is the two-sided ideal of $\Free(V)$ generated by $R_{m,n}$, $m,n\in\Phi_{(M,(-,-))}^+$, we have
\[
 \Free(V)/I_{M,V,\Alg}\cong \UEA(\mathfrak{n}^+_{M,(-,-),P}).
\]
Similarly, if $I_{M,V,\Lie}$ is the Lie ideal of $\Free_{\Lie}(V)$ generated by $R_{m,n}$, we have $ \Free_{\Lie}(V)/I_{M,V,\Lie}\cong\mathfrak{n}^+_{M,(-,-),P}$.  We define
\[
\Fn^+_{M,(-,-),V}\coloneqq \Free_{\Lie}(V)/I_{M,V,\Lie}.
\]
As in \S \ref{subsection:GKMtensorcategories}, one shows that $\UEA(\Fn^+_{M,(-,-),V})\cong \Free(V)/I_{M,V,\Alg}$.

\subsection{Lowest weight modules}
\label{subsection:lowestweightmodule}

Let $(M,(-,-))$ be a monoid with bilinear form. We assume that the natural map $M\rightarrow\widehat{M}$ is an embedding into a lattice. We let $P\colon\Phi^+_{M,(-,-)}\rightarrow\BoN[t^{\pm 1/2}]$ be a weight function. We define $\mathfrak{g}\coloneqq \mathfrak{g}_{M,(-,-),P}$ and $\mathfrak{b}^-\coloneqq \mathfrak{n}^-_{M,(-,-),P}\oplus\mathfrak{h}_M$. For $\mathbf{f}\in\Hom(\widehat{M},\BoZ)$, we define the \emph{Verma module} $V_{\mathfrak{g},\mathbf{f}}$ as follows:
\[
 V_{\Fg,\mathbf{f}}\coloneqq \UEA(\Fg)\otimes_{\UEA(\mathfrak{b}^-)}\BoC,
\]
where $\BoC$ is a $\UEA(\mathfrak{b}^-)$-module via the quotient $\UEA(\mathfrak{b}^-)\rightarrow\UEA(\mathfrak{h})$ and the action of $\mathfrak{h}$ on $\BoC$ is given by $h_m\cdot 1\coloneqq \mathbf{f}(m)$ for $m\in M$. We let $v\coloneqq 1\otimes 1\in V_{\mathfrak{g},\mathbf{f}}$. There is a unique maximal submodule of $V_{\mathfrak{g},\mathbf{f}}$ not containing $v$. We denote by $L_{\mathfrak{g},\mathbf{f}}$ (or $L_{\mathbf{f}}$ for short) the quotient of $V_{\mathfrak{g},\mathbf{f}}$ by this submodule. It is a simple $\mathfrak{g}$-module. It is called the \emph{simple lowest weight} $\mathfrak{g}$-module of weight $\mathbf{f}$.

\subsection{Cartan involution}\label{CI_ssec}
\label{subsection:Chevalleyinvolution}
The \emph{Cartan involution} of $\Fg_{M,(-,-),P}$ is the unique automorphism of Lie algebras $w\colon\Fg_{M,(-,-),P}\rightarrow \Fg_{M,(-,-),P}$ such that for any $m\in\Phi^+_{M,(-,-)}$, $j\in\BoZ$, $1\leq l\leq p_{m,j}$,
\[
 w(e_{m,j,l})=-f_{m,j,l};\quad\quad
 w(f_{m,j,l})=-e_{m,j,l};\quad\quad
 w(h_m)=-h_m.
\]
It is described for Kac--Moody algebras in \cite[\S1.3]{kac1990infinite}. For generalised Kac--Moody algebras, its definition appears in \cite[\S1]{borcherds1988generalized}.

\subsection{Generalised half Kac--Moody (Lie) algebras in tensor categories of mixed Hodge modules or perverse sheaves}
\label{subsection:GKMtensorcategories}
We let $(\CM,\oplus)$ be a commutative monoid in the category of complex schemes. We allow $\CM$ to have infinitely many connected components. We assume that each connected component of $\CM$ is a finite type separated scheme and that the sum $\oplus \colon \CM\times\CM\rightarrow\CM$ is a finite morphism. We let $\CM_0\subset \CM$ be the neutral component of $\CM$.  The formula $\SF \boxdot \SG \coloneqq \oplus_*(\SF \boxtimes \SG)$ defines symmetric monoidal structures on $\CD^+_{\rmc}(\CM)$ and $\D^+(\MHM(\CM))$, and by $t$-exactness of $\oplus_*$, also on $\Perv(\CM)$ and $\MHM(\CM)$. We let $T_{\SF,\SG}\colon \SF\boxdot\SG\xrightarrow{\cong}\SG\boxdot\SF$ be the natural isomorphisms for objects $\SF, \SG$ of $\CD^{\rmb}_\rmc(\CM)$ or $\CD^+(\MHM(\CM))$. See \cite{maxim2011symmetric} for details. In this section, we define (positive halves of) generalised Kac--Moody Lie algebras in the tensor category $\MHM(\CM)$. The same formulas provide a definition for generalised half Kac--Moody Lie algebras in $\Perv(\CM)$; we do not make them explicit (just apply the functor $\mathbf{rat}\colon\MHM(\CM)\rightarrow\Perv(\CM)$).

\subsubsection{Free algebras and free Lie algebras}
\label{subsubsection:freealgebras}
Let $\pi_0(\CM)$ be the monoid of connected components of $\CM$. We fix a morphism of monoids $\cl\colon \pi_0(\CM)\rightarrow M$ with finite fibres, where $M$ satisfies the condition that the natural map $M\rightarrow \widehat{M}$ is an injection into a lattice. For $m\in M$ we write $\CM_m\subset \CM$ for the union of connected components in $\cl^{-1}(m)$. Let $(-,-)\colon M\times M\rightarrow \BoZ$ be a bilinear form. The sets of primitive positive roots $\Sigma_{M,(-,-)}$ and simple positive roots $\Phi_{M,(-,-)}^+$ are defined in \S\ref{subsection:roots}. We let $A_{M,(-,-)}$ be the Cartan matrix, as defined in \S\ref{subsection:GKMmonoids}. We assume that it satisfies the assumptions of \S\ref{subsection:GKMmonoids}: $A_{M,(-,-)}$ is symmetric, $a_{m,m}\in 2\BoZ_{\leq 1}$ for all $m$, and $a_{m,n}\leq 0$ if $a_{m,m}=2$. We also assume that the fibres of the sum map $+\colon M\times M\rightarrow M$ are finite. This condition will ensure that the restriction to various connected components of $\CM$ of the constructions are actual complexes of mixed Hodge modules and not infinite direct sums of such (which would require some completion of the derived category of mixed Hodge modules).

We assume that $\CM_m\cong \pt$ for $m\in\Phi_{M,(-,-)}^{+,\real}$. Let $\underline{\SG}_m\in\MHM(\CM_m)$ for $m\in\Phi_{M,(-,-)}^+$. We assume that $\underline{\SG}_m=\underline{\BoQ}_{\CM_m}$ for $m\in\Phi_{M,(-,-)}^{+,\real}$. We let $\ul{\SG}=\bigoplus_{m\in\Phi^+_{M,(-,-)}}\ul{\SG}_m$ and consider
\[
\begin{aligned}
 \Free_{\boxdot\sAlg}(\ul{\SG})&\coloneqq \Free_{\boxdot\sAlg}\left(\bigoplus_{m\in\Phi_{M,(-,-)}^+}\ul{\SG}_m\right)\\ 
 &=\bigoplus_{r\geq 0}\bigoplus_{m_1,\hdots,m_r\in\Phi_{M,(-,-)}^+}\oplus_*(\ul{\SG}_{m_1}\boxtimes\hdots\boxtimes\ul{\SG}_{m_r})\\
           &=\bigoplus_{r\geq 0}\oplus_*(\ul{\SG}^{\boxtimes r}).
\end{aligned}
\]
The finiteness of $\oplus$ ensures that $\Free_{\boxdot\sAlg}(\ul{\SG})$ is a mixed Hodge module on $\CM$, i.e. the restriction to each connected component of $\CM$ is a genuine mixed Hodge module. The mixed Hodge module $\Free_{\boxdot\sAlg}(\ul{\SG})$ is an algebra object, with the multiplication $\mult\colon\Free_{\boxdot\sAlg}(\ul{\SG})\boxdot\Free_{\boxdot\sAlg}(\ul{\SG})\rightarrow\Free_{\boxdot\sAlg}(\ul{\SG})$ obtained via the canonical identifications
\[
 \oplus_*\!\left((\oplus_*\ul{\SG}^{\boxtimes r})\boxtimes(\oplus_*\ul{\SG}^{\boxtimes s})\right)\simeq \oplus_*\ul{\SG}^{\boxtimes (r+s)}.
\]
For $m\in\Phi_{M,(-,-)}^+$, we have by definition a natural monomorphism $\ul{\SG}_m\rightarrow \Free_{\boxdot\sAlg}(\ul{\SG})$. Moreover, the commutator $\mult-\mult\circ T_{\Free_{\boxdot\sAlg}(\ul{\SG}),\Free_{\boxdot\sAlg}(\ul{\SG})}$ gives $\Free_{\boxdot\sAlg}(\ul{\SG})$ a Lie algebra structure. We let $\Free_{\boxdot\sLie}(\ul{\SG})$ be the Lie subalgebra of $\Free_{\boxdot\sAlg}(\ul{\SG})$ generated by $\ul{\SG}$. We denote by $[-,-]$ its Lie bracket. It is referred to as \emph{the free Lie algebra generated by $\ul{\SG}$}.

For morphisms of mixed Hodge modules $\ul{\mathscr{F}}\rightarrow \Free_{\boxdot\sLie}(\ul{\SG})$ and $\ul{\mathscr{H}}\rightarrow \Free_{\boxdot\sLie}(\ul{\SG})$, we have a morphism $\ul{\SF}\boxtimes\ul{\SH}\rightarrow\Free_{\boxdot\sLie}(\ul{\SG})\boxtimes\Free_{\boxdot\sLie}(\ul{\SG})$ and hence, by applying $\oplus_*$, a morphism $\ul{\SF}\boxdot\ul{\SH}\rightarrow\Free_{\boxdot\sLie}(\ul{\SG})\boxdot\Free_{\boxdot\sLie}(\ul{\SG})$. We denote by $\ad(\ul{\SF})(\ul{\SH})\subset \Free_{\boxdot\sLie}(\ul{\SG})$ the image of the composition of this morphism with the Lie bracket $[-,-]$. We can also see $\ad(\ul{\SF})(\ul{\SH})$ as a subobject of $\Free_{\boxdot\sAlg}(\ul{\SG})$ thanks to the inclusion $\Free_{\boxdot\sLie}(\ul{\SG})\rightarrow\Free_{\boxdot\sAlg}(\ul{\SG})$. We can iterate this construction, therefore defining for any $k\geq 0$ the subobject $\ad(\ul{\SF})^k(\ul{\SH})\subset \Free_{\boxdot\sLie}(\ul{\SG})\subset \Free_{\boxdot\sAlg}(\ul{\SG})$.

\subsubsection{Serre relations and Serre ideal in the category of mixed Hodge modules}
\label{subsubsection:Serresrelationsandideal}
\begin{enumerate}
\item \label{item:rel1}For $m,n\in\Phi_{M,(-,-)}^+$ such that $a_{m,n}=0$, we set $\underline{\SR}_{m,n}=\ad(\underline{\SG}_m)(\underline{\SG}_n)\subset\Free_{\boxdot\sLie}(\underline{\SG})$,
\item \label{item:rel2}For $m,n\in\Phi_{M,(-,-)}^+$ such that $a_{m,m}=2$, we set $\underline{\SR}_{m,n}=\ad(\underline{\SG}_m)^{1-a_{m,n}}(\underline{\SG}_n)\subset \Free_{\boxdot\sLie}(\underline{\SG})$,
\item For all other $m,n\in\Phi_{M,(-,-)}^+$, we set $\underline{\SR}_{m,n}=0$.
\end{enumerate}
We note that if $a_{m,n}=0$ and $a_{m,m}=2$, both \eqref{item:rel1} and \eqref{item:rel2} give the same definition of $\underline{\SR}_{m,n}$.

We let $\CI_{\CM,\underline{\SG},\Alg}$ be the two-sided ideal of $\Free_{\boxdot\sAlg}(\underline{\SG})$ generated by $\underline{\SR}_{m,n}$, $m,n\in\Phi_{M,(-,-)}^+$. This is the two-sided ideal of $\Free_{\boxdot\sAlg}(\underline{\SG})$ defined as the image of the morphism
\[
\Free_{\boxdot\sAlg}(\underline{\SG})\boxdot \left(\left(\bigoplus_{m,n\in\Phi_{M,(-,-)}^+}\underline{\SR}_{m,n}\right)\boxdot \Free_{\boxdot\sAlg}(\underline{\SG})\right)\rightarrow\Free_{\boxdot\sAlg}(\underline{\SG})
\]
induced by $\mult\circ (\id\boxdot \mult)$. It is called the \emph{Serre ideal}.

We let $\CI_{\CM,\SG,\Lie}$ be the Lie ideal of $\Free_{\boxdot\sLie}(\underline{\SG})$ generated by $\underline{\SR}_{m,n}$, $m,n\in\Phi_{M,(-,-)}^+$. It is defined as the sum of the images of the morphisms
\[
 \left(\left(\bigoplus_{m,n\in\Phi_{M,(-,-)}^+}\underline{\SR}_{m,n}\right)\boxdot (\Free_{\boxdot\sLie}(\underline{\SG}))^{\boxdot t}\right)\rightarrow\Free_{\boxdot\sLie}(\underline{\SG}).
\]
induced by the iterated Lie bracket $[\mydots[-,-],-],\mydots ]$, for $t\geq 0$. It is also called the Serre (Lie) ideal, and the context will determine which of $\CI_{\CM,\SG,\Alg}$ and $\CI_{\CM,\SG,\Lie}$ is considered.

\subsubsection{Enveloping algebras}
\label{subsubsection:envelopingalgebra}
If $\underline{\mathscr{L}}$ is a Lie algebra object in $\MHM(\CM)$ with $\underline{\mathscr{L}}_{\CM_0}=0$, we let $\Free_{\boxdot\sAlg}(\underline{\mathscr{L}})$ be the free associative algebra generated by the underlying mixed Hodge module of $\underline{\mathscr{L}}$ defined in \S \ref{subsubsection:freealgebras}. We let $\UEA(\underline{\mathscr{L}})$ be the quotient of $\Free_{\boxdot\sAlg}(\underline{\mathscr{L}})$ by the two-sided ideal generated by the image of the morphism
\begin{equation}
\label{equation:envalgideal}
 C\coloneqq[-,-]\oplus (T_{\underline{\mathscr{L}},\underline{\mathscr{L}}}-\id_{\underline{\mathscr{L}}\boxdot\underline{\mathscr{L}}})\colon \underline{\mathscr{L}}\boxdot\underline{\mathscr{L}}\rightarrow \underline{\mathscr{L}}\oplus (\underline{\mathscr{L}}\boxdot\underline{\mathscr{L}})\subset \Free_{\boxdot\sAlg}(\underline{\mathscr{L}}).
\end{equation}
We have a canonical monomorphism of Lie algebra objects $\underline{\mathscr{L}}\rightarrow\UEA(\underline{\mathscr{L}})$ induced by the canonical inclusion $\underline{\mathscr{L}}\rightarrow\Free_{\boxdot\sAlg}(\underline{\mathscr{L}})$, considering the target as a Lie algebra via the commutator bracket.

\begin{proposition}[Universal property]
\label{proposition:univprop}
 For any associative algebra $\underline{\SA}$ in $(\MHM(\CM),\boxdot)$ and any morphism of Lie algebras $\underline{\mathscr{L}}\rightarrow\underline{\SA}$, there exists a unique morphism of algebras $\UEA(\underline{\mathscr{L}})\rightarrow\underline{\SA}$ such that the following diagram commutes:
 \[
 \begin{tikzcd}
	{\underline{\mathscr{L}}} & {\UEA(\underline{\mathscr{L}})} \\
	& \underline{\SA}
	\arrow[from=1-1, to=2-2]
	\arrow[from=1-2, to=2-2]
	\arrow[from=1-1, to=1-2]
\end{tikzcd}
 \]
\end{proposition}
\begin{proof}
 The proof follows exactly the same steps as the proof of the analogous statement for Lie algebras and their enveloping algebras in the category of vector spaces. In brief, the morphism $\underline{\mathscr{L}}\rightarrow \underline{\SA}$ induces a morphism $\Free_{\boxdot\sAlg}(\underline{\mathscr{L}})\rightarrow \underline{\SA}$. The image of $C$ \eqref{equation:envalgideal} is clearly in the kernel of this morphism. We then obtain the factorisation $\UEA(\underline{\mathscr{L}})\rightarrow \underline{\SA}$ as in the proposition. The uniqueness is immediate, as $\underline{\mathscr{L}}$ generates the algebra object $\UEA(\underline{\mathscr{L}})$.
\end{proof}

\begin{proposition}[PBW isomorphism]
Let $\underline{\mathscr{L}}\in\MHM(\CM)$ be a Lie algebra object such that $\underline{\mathscr{L}}_{\CM_0}=0$. Then the morphism $\Sym_{\boxdot}(\underline{\mathscr{L}})\rightarrow \UEA(\underline{\mathscr{L}})$ induced by the multiplication is an isomorphism.
\end{proposition}
\begin{proof}
 The proof is identical to that for ordinary Lie algebras. Namely, we can define a filtration of $\UEA(\underline{\mathscr{L}})$ by letting $\UEA(\underline{\mathscr{L}})_n$ be the image under the canonical surjection $\Free_{\boxdot\sAlg}(\underline{\mathscr{L}})\rightarrow\UEA(\underline{\mathscr{L}})$ of the sub-mixed Hodge module $\bigoplus_{m\leq n}\underline{\mathscr{L}}^{\boxdot m}$. The corresponding associated graded algebra object is isomorphic to $\Sym_{\boxdot}(\underline{\mathscr{L}})$.
\end{proof}

We now give a formal proposition whose proof follows from the universal property of the universal enveloping algebra.

\begin{proposition}
\label{proposition:comparisonLiealggeneral}
Let $\underline{\SG}$ be an object of the symmetric monoidal category $\MHM(\CM)$ such that $\underline{\SG}_{\CM_0}=0$.
Let $\Free_{\boxdot\sLie}(\underline{\SG})$ be the free Lie algebra generated by $\underline{\SG}$ and $\underline{\SR}\subset \Free_{\boxdot\sLie}(\underline{\SG})$ be a suboject. Consider the Lie ideal $\CI_{\underline{\SR},\Lie} \subset\Free_{\boxdot\sLie}(\underline{\SG})$ generated by $\underline{\SR}$. Consider also the two-sided ideal $\CI_{\underline{\SR},\Alg} \subset \Free_{\boxdot\sAlg}(\underline{\SG})$ generated by $\underline{\SR}$. 
Then the universal enveloping algebra $\UEA(\Free_{\boxdot\sLie}(\underline{\SG})/\CI_{\underline{\SR},\Lie})$ is naturally isomorphic to the algebra $\Free_{\boxdot\sAlg}(\underline{\SG})/\CI_{\underline{\SR},\Alg}$.
Moreover, the Lie subalgebra $\underline{\mathscr{L}} \subset \UEA(\Free_{\boxdot\sLie}(\underline{\SG})/\CI_{\underline{\SR},\Lie})$ generated by $\underline{\SG} \subset \UEA(\Free_{\boxdot\sLie}(\underline{\SG})/\CI_{\underline{\SR},\Lie})$ coincides with $\Free_{\boxdot\sLie}(\underline{\SG})/\CI_{\underline{\SR},\Lie}$.
\end{proposition}

\begin{proof}
Let $\underline{\SA}=\Free_{\boxdot\sAlg}(\underline{\SG})/\CI_{\underline{\SR},\Alg}$, $\underline{\Fn} = \Free_{\boxdot\sLie}(\underline{\SG})/\CI_{\underline{\SR},\Lie}$ and $\underline{\mathscr{B}} = \UEA(\underline{\Fn})$. Since $\CI_{\underline{\SR},\Lie} \subset \CI_{\underline{\SR},\Alg}$ we have a morphism of Lie algebras $\underline{\Fn} \to \underline{\SA}$. By Proposition~\ref{proposition:univprop}, it extends uniquely to a morphism of algebras $\underline{\mathscr{B}}\rightarrow \underline{\SA}$. Conversely, we have a morphism of algebras $\Free_{\boxdot\sAlg}(\SG)\rightarrow \underline{\mathscr{B}}$ induced by the natural map $\underline{\SG}\rightarrow \underline{\mathscr{B}}$. The generating subobject $\underline{\SR}$ of the ideal $\CI_{\underline{\SR},\Alg}$ is sent to zero by this morphism since $\underline{\SR} \subset \CI_{\underline{\SR},\Lie}$. We then obtain the morphism $\underline{\SA}\rightarrow \underline{\mathscr{B}}$ by quotienting the source by $\CI_{\underline{\SR},\Alg}$. The morphisms $\underline{\SA}\rightarrow \underline{\mathscr{B}}$ and $\underline{\mathscr{B}}\rightarrow \underline{\SA}$ are inverse of each other (as their composition is the identity on the generating subobject $\underline{\SG}$). 

The second part of the statement follows from the first since both Lie algebras are generated by $\underline{\SG}$.
\end{proof}

\begin{definition}[Positive part of the enveloping algebra and generalised Kac--Moody Lie algebra]
For $\CM$ and $\underline{\SG}$ as in \S\ref{subsubsection:freealgebras}, we define $\mathfrak{n}^+_{\CM,(-,-),\underline{\SG}}\coloneqq \Free_{\boxdot\sLie}(\underline{\SG})/\CI_{\CM,\underline{\SG},\Lie}$. It has an induced Lie algebra structure in $\MHM(\CM)$. We let $\UEA(\mathfrak{n}^+_{\CM,(-,-),\underline{\SG}})$ be its universal enveloping algebra in $\MHM(\CM)$.
\end{definition}

The MHM $\UEA(\mathfrak{n}^+_{\CM,(-,-),\underline{\SG}})$ is an algebra object in $\MHM(\CM)$ and it is isomorphic to $\Free_{\boxdot\sAlg}(\underline{\SG})/\CI_{\CM,\underline{\SG},\Alg}$ by Proposition~\ref{proposition:comparisonLiealggeneral}. By the same proposition, $\Fn^+_{\CM,(-,-),\underline{\SG}}$ is the Lie subalgebra object in $\UEA(\mathfrak{n}^+_{\CM,(-,-),\underline{\SG}})$ generated by $\underline{\SG}$.

\subsection{Global sections and restriction of generalised Kac--Moody algebras}
\label{subsection:globsectGKM}

Recall the following elementary lemma, an analogue of Proposition~\ref{proposition:comparisonLiealggeneral}.

\begin{lemma}
\label{lemma:envelopingLiealg}
 Let $A\coloneqq F/I$ be an associative algebra described as the quotient of the free associative algebra $F$ on a set of generators $E$ modulo a two-sided ideal $I\subset F$ generated by a subspace $V\subset F$. We let $\overline{E}$ be the image of $E$ in $A$. We let $\mathfrak{g}$ be the Lie subalgebra of $A$ generated by $\overline{E}$. We assume that $V\subset \Free_{\Lie}(E)\subset F$. Then, there is a canonical isomorphism $A\cong \UEA(\mathfrak{g})$. 
\end{lemma}

\begin{proposition}
We let $\cl\colon \pi_0(\CM)\rightarrow M$ be a morphism of monoids, where $M$ is finitely generated, commutative and cancellative and set
\[
 \begin{matrix}
  P\colon& \Phi_{M,(-,-)}^+&\rightarrow&\BoN[t^{\pm 1/2}]&\\
    &   m&\mapsto  &\ch^{\BoZ}(\HO^*(\SG_m))&\coloneqq\sum_{j\in\BoZ}\dim\HO^j(\SG_m)t^{j/2}.
 \end{matrix}
 \]
\begin{enumerate}
 \item The derived global sections $\HO^*(\UEA(\mathfrak{n}^+_{\CM,(-,-),\SG}))$ is an algebra, isomorphic to the positive part $\UEA(\mathfrak{n}^+_{M,(-,-),P})$ of the enveloping algebra $\UEA(\mathfrak{g}_{M,(-,-),P})$.
 \item The derived global sections $\HO^*(\mathfrak{n}^+_{\CM,(-,-),\SG})$ is a Lie algebra, isomorphic to the positive part $\mathfrak{n}^+_{M,(-,-),P}$ of the generalised Kac--Moody Lie algebra $\mathfrak{g}_{M,(-,-),P}$.
\end{enumerate}
\end{proposition}
\begin{proof}
 By the K\"unneth formula, $\HO^*$ is a monoidal functor. Therefore, $\HO^*(\mathfrak{n}^+_{\CM,(-,-),\SG})$ is the Lie subalgebra of $\HO^*(\UEA(\mathfrak{n}^+_{\CM,(-,-),\SG}))$ generated by $\HO^*(\SG_m)$, $m\in\Phi_{M,(-,-)}^+$. Moreover,
 \[
 \HO^*(\UEA(\mathfrak{n}^+_{\CM,(-,-),\SG}))=\Free\left(\bigoplus_{m\in\Phi_{M,(-,-)}^+}\HO^*(\SG_m)\right)/\HO^*(\CI_{\CM,\SG,\Alg}) 
 \]
and $\HO^*(\CI_{\CM,\SG,\Alg})$ is the two-sided ideal of $\Free\left(\bigoplus_{m\in\Phi_{M,(-,-)}^+}\HO^*(\SG_m)\right)$ generated by $\HO^*(\SR_{m,n})$ for $m,n\in\Phi_{M,(-,-)}^+$. By the K\"unneth formula again, using $\CM_m=\pt$ and $\SG_m=\underline{\BoQ}_{\pt}$ if $m\in\Phi_{M,(-,-)}^{+,\real}$, $\HO^*(\SR_{m,n})$ is exactly the set of relations $R_{m,n}$ associated to the $\BoZ$-graded vector spaces $\HO^*(\SG_m)$, $m\in\Phi^+_{M,(-,-)}$ (see the end of \S\ref{subsection:GKMmonoids}). This proves (1). Since $R_{m,n}\subset \Free_{\Lie}(\bigoplus_{m\in\Phi_{M,(-,-)}^+}\HO^*(\SG_m))$, (2) follows thanks to Lemma~\ref{lemma:envelopingLiealg}.
\end{proof}

\begin{proposition}
\label{proposition:discreterestriction}
Assume that $M=\pi_0(\CM)$ (so that $\pi_0(\CM)$ is a finitely generated and cancellative monoid). We assume also that there is a saturated inclusion of monoids $\imath\colon\pi_0(\CM)\rightarrow \CM$, right inverse to the projection $p\colon\CM\rightarrow\pi_0(\CM)$ (sending all points of $\CM_m$ to $m$ for $m\in\pi_0(\CM)$), and such that for any $m\in\pi_0(\CM)$, there is an action of $\BoC^*$ on $\CM_m$ such that $\iota(m)$ is the unique fixed point and the action contracts $\CM_m$ onto $\iota(m)$.
We set
\[
 \begin{matrix}
 P\colon& \Phi_{M,(-,-)}^+&\rightarrow &\BoN[t^{\pm 1/2}]&\\
     &m   &\mapsto   &\dim \HO^*(\BD\SG_m)^*&=\sum_{j\in\BoZ}\dim\HO^j(\BD\SG_m)t^{-j/2}. 
 \end{matrix}
\]
\begin{enumerate}
\item The algebra $\imath^!\UEA(\mathfrak{n}^+_{\CM,(-,-),\SG})$ is isomorphic to the universal enveloping algebra $\UEA(\mathfrak{n}^+_{\pi_0(\CM),(-,-),P})$,
 \item The Lie algebra $\imath^!\mathfrak{n}^+_{\CM,(-,-),\SG}$ is isomorphic to the positive part $\mathfrak{n}^+_{\pi_0(\CM),(-,-),P}$ of the generalised Kac--Moody Lie algebra $\mathfrak{g}_{\pi_0(\CM),(-,-),P}$.
\end{enumerate}
\end{proposition}
\begin{proof}
By base change in the diagram
\[
 \begin{tikzcd}
	{\pi_0(\CM)\times\pi_0(\CM)} & {\pi_0(\CM)} \\
	\CM\times\CM & \CM
	\arrow["\oplus", from=1-1, to=1-2]
	\arrow["\oplus", from=2-1, to=2-2]
	\arrow["\imath\times\imath"', from=1-1, to=2-1]
	\arrow["\imath", from=1-2, to=2-2]
	\arrow["\lrcorner"{anchor=center, pos=0.125}, draw=none, from=1-1, to=2-2]
\end{tikzcd}
\]
the functor $\imath^!\colon \CD^+(\MHM(\CM))\rightarrow\CD^+(\MHM(\pi_0(\CM)))$ is strictly monoidal.

 By the assumptions on $\CM$ and $\iota$, if $p\colon\CM\rightarrow \pi_0(\CM)$ is the natural projection mapping a point $x\in\CM$ to the connected component of $\CM$ it belongs to, $\iota^!\cong p_!$ (see \cite[Proof of Lemma 6.10]{davison2020bps}). By strict monoidality of $\iota^!$, $\iota^!\UEA(\mathfrak{n}^+_{\CM,(-,-),\SG})\cong \iota^!\Free_{\Alg}(\SG)/\iota^!\CI_{\CM,\SG,\Alg}\cong \Free(\bigoplus_{m\in\Phi_{M,(-,-)}^+}\iota^!\SG_m)/\iota^!\CI_{\CM,\SG,\Alg}$. Now, $\iota^!\SG_m\cong p_!\SG_m\cong \HO^{-*}(\BD\SG_m)^*$ as $\BoZ$-graded vector spaces and since for $m\in\Phi_{M,(-,-)}^{+,\real}$, $\CM_m=\pt$ and $\SG_m=\underline{\BoQ}_{\pt}$, $\iota^!(\SG_m)=\BoQ e_m$ and $\iota^!\SR_{m,n}$ is the Serre relation $R_{m,n}$ associated to $\HO^*(\BD\SG_m)^*$, $m\in\Phi_{M,(-,-)}^+$ (see the end of \S\ref{subsection:GKMmonoids}). This proves part (1). Part (2) follows using Lemma~\ref{lemma:envelopingLiealg}. 
\end{proof}

\section{Moduli stacks of objects in $2$-dimensional categories}
\label{section:modulistack2dcats}
We recall the categorical and geometric framework of \cite{davison2022BPS}. 

\subsection{Moduli stacks of objects}
\label{subsection:gsetup}
We let $\mathscr{C}$ be a $\BoC$-linear dg-category, and let $\bm{\mathfrak{M}}\subset \bm{\mathfrak{M}}_{\mathscr{C}}$ be a $1$-Artin open substack of the stack of objects of $\mathscr{C}$ (in the sense of \cite{toen2007moduli}). We denote by $\mathfrak{M}= t_0(\bm{\mathfrak{M}})$ its classical truncation. We assume that $\mathfrak{M}$ parametrises objects of an admissible (in the sense of \cite[\S 1.2]{beilinson2018faisceaux}) finite length Abelian subcategory $\mathcal{A}$ of the homotopy category $\HO^0(\mathscr{C})$. For $X=\Spec(A)$, $X$-points of $\bm{\mathfrak{M}}_{\mathscr{C}}$ correspond to pseudo-perfect $\mathscr{C}\otimes_{\BoC} A$-modules $N$, i.e. bimodules $N$ that are perfect as $A$-modules. Given a pair of such modules we obtain the dg $A$-module $\RHom_{\mathscr{C}\otimes_{\BoC} A}(N,N')$. This defines the RHom complex on $\bm{\mathfrak{M}}_{\mathscr{C}}^{\times 2}$. We define the $\RHom$ complex $\CH^{\bullet}$ on $\mathfrak{M}^{\times 2}$ by restriction.

\subsection{Assumptions on the stack of objects}
\label{subsection:assumptionsCoHA}
For the construction of relative CoHAs in \cite{davison2022BPS}, various assumptions are made on the stack of objects in $\CA$, which we now recall. In concrete applications they are all known to hold; see \cite{davison2022BPS} for a more thorough treatment of these, as well as the assumptions \ref{gms_assumption}--\ref{BPS_cat_assumption} that follow. The first three concern the stack of objects $\FM$ from \S \ref{subsection:gsetup}. Assumptions \ref{p_assumption}, \ref{q_assumption1} and \ref{ass:associativity} are necessary for the construction of the CoHA product. Assumptions \ref{gms_assumption} and \ref{BPS_cat_assumption} allow us to use the neighbourhood theorem of \cite{davison2021purity}. Assumption \ref{ds_fin} is used to defined the BPS algebra.
\begin{assumption}{1}
[Assumption for the construction of the pushforward by $p$]
\label{p_assumption}
Let $\mathfrak{Exact}$ be the (classical truncation of the) stack of short exact sequences in $\mathcal{A}$. We let $p\colon\mathfrak{Exact}\rightarrow \mathfrak{M}$ be the morphism forgetting the extreme terms of the short exact sequence. We assume that $p$ is a proper and representable morphism of stacks.
\end{assumption}

Let $\CC^{\bullet}=\CH^{\bullet}[1]$ be a shift of the RHom complex defined in \S\ref{subsection:gsetup}. The stack $\bsE=\Tot_{\FM\times\FM}(\CC^{\bullet})$ carries a universal bundle of (shifted) homomorphisms of objects in $\mathscr{C}$, so that we have a natural identification $t_0(\bsE)=\mathfrak{Exact}$, and a morphism $\bs{p}\colon \bsE\rightarrow \FM$ given by sending the morphism $f\colon \rho\rightarrow \rho'$ to $\cone(f)[1]$. Then $p=t_0(\bs{p})$ is the morphism sending a short exact sequence to its middle term.
\begin{assumption}{2}
[Assumption for the construction of the virtual pullback by $\bs{q}$]
\label{q_assumption1}
We assume that $\FM$ has the resolution property, i.e. every coherent sheaf on $\FM$ is a quotient of a vector bundle. For all $\CM_a,\CM_b\in\pi_0(\CM)$ we assume that the complex $\CC^{\bullet}$ is quasi-isomorphic to a $3$-term complex of vector bundles over $\FM_a\times\FM_b$.
\end{assumption}
We note that, for the second part of Assumption~\ref{q_assumption1} to be satisfied, it is enough for the RHom complex to be globally presented by a three-term complex of vector bundles, which is generally easy to check in examples. There are general criteria ensuring the resolution property \cite{totaro2004resolution,gross2017tensor}.

\subsection{Assumptions on the good moduli space}
\label{subsec:spaceassumptions}
The construction of the BPS algebra in \S \ref{subsection:BPSalg} requires the existence of a good moduli space, satisfying various extra assumptions. Again, in concrete applications, they are all well-known to hold; see \cite{davison2022BPS} for a leisurely discussion.

\begin{assumption}{3}
\label{gms_assumption}
The stack $\FM$ has a good moduli space $\FM\xrightarrow{\JH}\CM$ in the sense of \cite{alper2013good}, and $\CM$ is a separated scheme with finite type connected components.
\end{assumption}
The scheme $\mathcal{M}$ is endowed with a monoid structure given by the direct sum:
\begin{align*}
 \oplus\colon\mathcal{M}\times\mathcal{M}\longrightarrow\mathcal{M};\quad\quad  (x,y)\longmapsto x\oplus y
\end{align*}
obtained from the direct sum $\oplus\colon\mathfrak{M}\times\mathfrak{M}\rightarrow\mathfrak{M}$ and universality of the good moduli space $\JH\colon\mathfrak{M}\rightarrow\mathcal{M}$. Note that the following diagram commutes:
\[
 \begin{tikzcd}
	{\mathfrak{M}\times\mathfrak{M}} & {\mathfrak{Exact}} & {\mathfrak{M}} \\
	{\mathcal{M}\times\mathcal{M}} && {\mathcal{M}.}
	\arrow["\JH\times\JH"', from=1-1, to=2-1]
	\arrow["q"', from=1-2, to=1-1]
	\arrow["p", from=1-2, to=1-3]
	\arrow["\JH", from=1-3, to=2-3]
	\arrow["\oplus", from=2-1, to=2-3]
\end{tikzcd}
\]
If $\CM_a$ is a connected component of $\CM$, we define $\FM_a\coloneqq \JH^{-1}(\CM_a)$, and for $a\in \K(\CA)$, where $\cl\colon \pi_0(\CM)\rightarrow\K(\CA)$ is a morphism of monoids, we define $\FM_a$ similarly (thus $\FM_a=\bigsqcup_{a'\in \cl^{-1}(a)}\FM_{a'}$).

\begin{assumption}{4}
\label{ds_fin}
We assume that the morphism $\oplus$ is finite.
\end{assumption}
Given that objects in $\CA$ have finite length, the morphism $\oplus$ is easily seen to be quasi-finite, so Assumption \ref{ds_fin} amounts to requesting that $\oplus$ is a proper morphism.

To apply the local neighbourhood theorem (Theorem~\ref{theorem:neighbourhood}) on closed points of the moduli stacks $\Mst$ we make the following additional Assumption~\ref{BPS_cat_assumption}.

\begin{assumption}{5}
	\label{BPS_cat_assumption}
	Suppose we are in the situation of \S \ref{subsection:gsetup} and Assumption \ref{gms_assumption} is satisfied. Given a collection of points $x_1,\ldots,x_r\in\CM$ parametrising simple objects $\underline{\CF}=\{\CF_1,\ldots,\CF_r\}$ of $\CA$, We assume that the full dg subcategory of $\mathscr{C}$ containing $\underline{\CF}$ carries a right 2-Calabi--Yau (2CY) structure in the sense of \cite[\S 3.1]{brav2019relative}.
\end{assumption}

Again, this somewhat technical assumption is known to hold in examples we encounter, since by \cite[Theorem 3.1]{brav2019relative}, the existence of the required right 2CY structure will follow from the existence of a left 2CY structure on the category $\SC$. The existence of a left 2CY structure for certain categories of coherent sheaves can be found in [loc. cit], for derived preprojective algebras was shown in Keller's paper \cite{keller2011deformed}, see
\cite[Appendix~C]{davison2022BPS} for a summary.

The 2CY assumption (Assumption~\ref{BPS_cat_assumption}) then plays a crucial role in the definition of the BPS algebra appearing in \S \ref{subsection:BPSalg}.

\subsection{Serre subcategories}
\label{subsubsection:serre}
Let $\FM\xrightarrow{\JH}\CM$ be as in Assumption~\ref{gms_assumption} a good moduli space for the stack of objects $\FM$ as in \S \ref{subsection:gsetup}. Fix a locally closed submonoid $\mathcal{M}'\subset \mathcal{M}$ that is stable under direct summands, i.e. saturated (if $x\oplus y\in \mathcal{M}'$, then $x,y\in\mathcal{M}'$). We define $\mathcal{B}$ to be the full Abelian subcategory of $\mathcal{A}$ generated by the objects represented by the closed points of $\mathcal{M}'$ under taking extensions. Typically, $\mathcal{B}$ will be defined to be the full subcategory of objects of $\mathcal{A}$ satisfying some support or nilpotency condition. Since $\CM$ parametrises semisimple objects in $\CA$, $\CB$ is a Serre subcategory (an Abelian subcategory stable under extensions, subobjects and quotients).

Let $\mathcal{B}$ be a Serre subcategory of $\mathcal{A}$ corresponding to a locally closed submonoid $\mathcal{M}_{\mathcal{B}}\subset \mathcal{M}$. We define $\mathfrak{M}_{\mathcal{B}}\subset \mathfrak{M}$ to be the substack parametrising objects of $\mathcal{B}$, which we \emph{define} via the following Cartesian square
\[
\begin{tikzcd}
	{\mathfrak{M}_{\mathcal{B}}} & {\mathfrak{M}} \\
	{\mathcal{M}_{\mathcal{B}}} & {\mathcal{M}}.
	\arrow[from=1-1, to=1-2]
	\arrow["\JH", from=1-2, to=2-2]
	\arrow[from=2-1, to=2-2]
	\arrow[from=1-1, to=2-1]
	\arrow["\lrcorner"{anchor=center, pos=0.125}, draw=none, from=1-1, to=2-2]
\end{tikzcd}
\]
The morphism $\FM_{\CB} \to \CM_{\CB}$ is a good moduli space by \cite[Proposition~4.7 (i)]{alper2013good}. Let $\pi_0(\mathcal{M}_{\mathcal{B}})$ be the monoid of connected components of $\mathcal{M}_{\mathcal{B}}$. The monoid structure $\oplus: \pi_0(\Msp_{\CB})\times \pi_0(\Msp_{\CB})\rightarrow \pi_0(\Msp_\CB)$ is induced by the direct sum. The monoid of connected components of the stack $\mathfrak{M}_{\mathcal{B}}$ coincides with $\pi_0(\Msp_{\CB})$, since the fibres of $\JH$ are connected, by \cite[Theorem 4.16 (vii)]{alper2013good}. We choose a monoid $\K(\CB)$ and a surjective morphism of monoids $\cl\colon \pi_0(\mathcal{M}_{\mathcal{B}})\rightarrow \K(\CB)$ with finite fibres, where the target is a monoid for which the morphism $\K(\CB)\rightarrow \widehat{\K(\CB)}$ is an inclusion in a lattice

 For $a\in\K(\CB)$, we denote by $\Mst_{\CB,a}$ and $\Msp_{\CB,a}$ the corresponding (finite) unions of connected components.

For any object $\mathcal{F}$ of $\mathcal{B}$ (resp. $\BoC$-point $x$ of $\mathfrak{M}_{\mathcal{B}}$ or $\mathcal{M}_{\mathcal{B}}$), we let $[\mathcal{F}]$ (resp. $[x]$) be $\cl$ applied to the connected component of $x$. We will refer to $[x]$ as \emph{the class} of $x$. The Euler form automatically factors through $\pi_0(\Msp_{\CB})$ to define a bilinear form:
\begin{equation*}
(-,-)_{\mathscr{C}}\colon \pi_0(\Msp_{\CB})\times\pi_0(\Msp_{\CB})\longrightarrow\BoZ.
\end{equation*}
We assume that it moreover factors through $\cl$. The pair $(\K(\CB),(-,-)_{\SC})$ will provide the combinatorial input required in \S \ref{subsection:roots}.

\subsection{The local neighbourhood theorem}
\label{subsubsection:setuplocal}
Let $\mathscr{C}$, $\mathfrak{M}$ and $\CM$ be as above, satisfying Assumptions \ref{gms_assumption} and \ref{BPS_cat_assumption}. Then the closed points of $\mathcal{M}$ are in bijection with semisimple objects of the category $\mathcal{A}$ and the map $\JH$ sends a $\BoC$-point of $\mathfrak{M}$ corresponding to an object $\mathcal{F}$ of $\mathcal{A}$ to the $\BoC$-point of $\mathcal{M}$ corresponding to the associated graded object of $\mathcal{F}$ with respect to some Jordan--H\"older filtration. Let $\cl\colon \pi_0(\CM_\CA) \to \K(\CA)$ be a surjective morphism of monoids. 

\subsubsection{$\Sigma$-collections}
\label{subsubsection:extquiver}
Let $\mathcal{F}$ be an object in a $\BoC$-linear Abelian category $\mathcal{A}$, which we will assume is realised as a full subcategory of the homotopy category of some dg category $\SC$. We say that $\mathcal{F}$ is a \emph{$\Sigma$-object} if $\Ext^*_{\SC}(\CF,\CF)\cong \HO^*(S,\BoC)$ for $S$ a closed Riemann surface (of some genus $g\geq 0$). If $\underline{\mathcal{F}}=\{\mathcal{F}_1,\hdots,\mathcal{F}_r\}$ is a collection of objects of $\mathcal{A}$, we say that $\underline{\mathcal{F}}$ is a \emph{$\Sigma$-collection} if each $\mathcal{F}_t$ is a $\Sigma$-object and for $m\neq n$, $\Hom(\mathcal{F}_m,\mathcal{F}_n)=0$.

Let $\underline{\mathcal{F}}=\{\mathcal{F}_1,\hdots,\mathcal{F}_r\}$ be a $\Sigma$-collection of objects of $\mathcal{A}$. We define the Ext-quiver of $\underline{\mathcal{F}}$ as the quiver $\overline{Q}_{\underline{\mathcal{F}}}$ having as set of vertices $\underline{\mathcal{F}}$, and having $\ext^1(\mathcal{F}_i,\mathcal{F}_j)\coloneqq \dim(\Ext^1(\mathcal{F}_i,\mathcal{F}_j))$ arrows from $i$ to $j$. When $\ext^1(\mathcal{F}_i,\mathcal{F}_j)=\ext^1(\mathcal{F}_j,\mathcal{F}_i)$ for any $1\leq i,j\leq r$ and this number is even if $i=j$, then $\overline{Q}_{\underline{\mathcal{F}}}$ is the double of some (usually non-unique) quiver $Q_{\underline{\mathcal{F}}}$. This condition is for example satisfied for right $2$-Calabi--Yau categories as in \S\ref{subsection:gsetup}. Accordingly, we make Assumption \ref{BPS_cat_assumption}. We will refer to $Q_{\underline{\mathcal{F}}}$ as a \emph{half of the Ext-quiver of $\underline{\mathcal{F}}$}.

We have a morphism of monoids
\begin{equation}
\label{lambdaFdef}
\lambda_{\underline{\CF}}\colon \BoN^{\underline{\CF}} \longrightarrow \K(\CA);\quad\quad\mm\longmapsto \sum_{i=1}^rm_i[\CF_i].
\end{equation}
We also have a morphism of monoids
\begin{equation}
\label{equation:morphismmonoids}
 \imath_{\underline{\mathcal{F}}}\colon \BoN^{\underline{\mathcal{F}}}\longrightarrow \mathcal{M}
\end{equation}
mapping $\mm\in\BoN^{\underline{\mathcal{F}}}$ to the point of $\mathcal{M}$ corresponding to the object $\bigoplus_{i=1}^r\mathcal{F}_i^{\oplus m_i}$, i.e. $1_{\CF_i}\mapsto x_i$ if we denote by $x_i\in \CM$ the point corresponding to $\CF_i$.

If $x\in\CM_{\CA}$, it corresponds to a semisimple object $\CF=\bigoplus_{i=1}^r\CF_i^{\oplus d_i}$. We define $(Q_x,\dd_x)$ to be the pair formed by a half of the Ext-quiver of $\underline{\CF}\coloneqq\{\CF_1,\hdots,\CF_r\}$ and the dimension vector $\dd_x=(d_1,\hdots,d_r)$. Although this involves making an irrelevant choice for $Q_x$, the doubled quiver $\overline{Q_x}$ only depends on the simple direct summands of $\CF$.

\subsubsection{The local neighbourhood theorem}
We recall here the local neighbourhood theorem for $2$-Calabi--Yau categories \cite[Theorem 5.11]{davison2021purity}, which we will repeatedly use, starting in the next section to investigate the geometry of components of moduli spaces corresponding to roots for general 2CY categories.
\begin{theorem}
\label{theorem:neighbourhood}
 Let $\JH\colon\mathfrak{M}\rightarrow\mathcal{M}$ be a good moduli space, locally of finite type with reductive stabiliser groups. We assume that $\mathfrak{M}=t_0(\bm{\mathfrak{M}})$ is an Artin stack, where $\bm{\mathfrak{M}}\subset\bm{\mathfrak{M}}_{\mathscr{C}}$ is an open substack of the stack of objects of a dg-category $\mathscr{C}$. Let $x$ be a closed $\BoC$-valued point of $\mathfrak{M}$ corresponding to an object $\mathcal{F}=\bigoplus_{1\leq i\leq r}\mathcal{F}_i^{\oplus m_i}$ and assume that the full dg-subcategory of $\mathscr{C}$ containing $\mathcal{F}_i$, $1\leq i\leq r$ carries a right 2CY structure and $\underline{\mathcal{F}}=\{\mathcal{F}_i\}_{1\leq i\leq r}$ is a $\Sigma$-collection. We let $Q$ be a half of the Ext-quiver of the collection $\underline{\mathcal{F}}=\{\mathcal{F}_i:1\leq i\leq r\}$. Then, there is a pointed affine $\GL_{\mm}$-variety $(H,y)$, and a $\GL_{\mm}$-invariant analytic neighbourhood $U\subset H$ of $y$ such that if $\mathcal{U}_x$ (resp. $\mathfrak{U}_x$) denotes the image of $U$ by the canonical map $H\rightarrow H\cms\GL_{\mm}$ (resp. $H\rightarrow H/\GL_{\mm}$), there is a commutative diagram
\[
\begin{tikzcd}
	{(\mathfrak{M}_{\Pi_Q,\mm},0_{\mm})} && {(\mathfrak{U}_{x},y)} && {(\mathfrak{M},x)} \\
	\\
	{(\mathcal{M}_{\Pi_Q,\mm},0_{\mm})} && {(\mathcal{U}_{x},p(y))} && {(\mathcal{M},\JH(x))}
	\arrow["{\JH_{\mm}}"', from=1-1, to=3-1]
	\arrow["p_{\mm}", from=1-3, to=3-3]
	\arrow["\JH", from=1-5, to=3-5]
	\arrow["{\tilde{\jmath}_x}", from=1-3, to=1-5]
	\arrow["{\jmath_x}", from=3-3, to=3-5]
	\arrow["{\jmath_{0_{\mm}}}"', from=3-3, to=3-1]
	\arrow["{\tilde{\jmath}_{0_{\mm}}}"', from=1-3, to=1-1]
	\arrow["\lrcorner"{anchor=center, pos=0.125, rotate=-90}, draw=none, from=1-3, to=3-1]
	\arrow["\lrcorner"{anchor=center, pos=0.125}, draw=none, from=1-3, to=3-5]
\end{tikzcd}
\]
in which the horizontal morphisms are analytic open immersions and the two squares are Cartesian.

\end{theorem}
The original version of this theorem gives an \'etale local description of the map $(\mathfrak{M},x)\rightarrow (\mathcal{M},\JH(x))$. For applications to calculating perverse sheaves (as opposed to objects in the derived category of perverse sheaves), working in the analytic topology has some advantages. Under the assumptions given in \S\ref{subsubsection:setuplocal} Assumption \ref{gms_assumption} and \ref{BPS_cat_assumption} guarantee that the theorem applies for any $\BoC$-valued closed point $x$ of $\mathfrak{M}$.

\subsection{Determinant line bundle}
\label{subsection:determinant}

\begin{assumption}{6}
	\label{determinant_assumption}
	Suppose $\FM$ satisfies Assumptions \ref{gms_assumption} and \ref{BPS_cat_assumption} so that the local neighbourhood theorem (Theorem~\ref{theorem:neighbourhood})
	applies to  $\FM$.
	For every closed point $x \in \FM$ we denote by $i_x\colon \B \BoC^{*} \into \FM$ the inclusion induced by the scaling automorphisms  $\BoC^{*}\subset \Aut_{\FM}(x)$ at $x$.
	
	We assume there is a line bundle $\det\colon \FM \to \B \BC^{*}$, called a \emph{positive determinant line bundle} on $\FM$ such that $i_x^*\det$ has a non-zero degree, i.e., for all closed points 
	$x \in \FM$ we have $(i_x \circ \det)^*u\neq 0 \in \HO^2(\B \BoC^{*})$
	where $u \in \HO^2(\B \BoC^{*})$ is a generator of the cohomology ring $\HO^*(\B \BoC^{*})$.
\end{assumption}

Again this assumption is known to hold in the relevant examples, see \cite[Appendices~B and C]{davison2022BPS}.

\section{Roots and the geometry of the good moduli space}
\subsection{Roots for preprojective algebras}
\label{subsection:rootsforPi}
Let $Q$ be a quiver and let $\Pi_Q$ be its preprojective algebra (see \S\ref{subsection:QandR}). We let $(-,-)_{\SG_2(\BoC Q)}\colon\BoN^{Q_0}\times\BoN^{Q_0}\rightarrow\BoZ$ be the Euler form of the derived preprojective algebra $\SG_{2}(\BoC Q)$. It is the symmetrization of the Euler form of $Q$ and, when $Q$ is not Dynkin, it coincides with the Euler form of $\Pi_Q$ itself. We let
\begin{multline*}
 \Sigma_{\Pi_Q}\coloneqq \left\{\dd\in\BoN^{Q_0}\setminus\{0\}\mid \text{ for any nontrivial decomposition }\dd=\sum_{j=1}^r\dd_j, \dd_j\in\BoN^{Q_0}\setminus\{0\}\right. \\ \left.\text{ one has }(2-(\dd,\dd)_{\SG_2(\BoC Q)})>\sum_{j=1}^r(2-(\dd_j,\dd_j)_{\SG_2(\BoC Q)})\right\}.
\end{multline*}
We let $Q_0^{\real}$ be the set of \emph{real} vertices (loop-free vertices of $Q$). For $i\in Q_0^{\real}$, we have a reflection
\[
 \begin{matrix}
 s_i&\colon&\BoZ^{Q_0}&\rightarrow&\BoZ^{Q_0}\\
   &   &\dd    &\mapsto  &\dd-(1_i,\dd)_{\SG_2(\BoC Q)}1_i.
 \end{matrix}
\]
These reflections generate the \emph{Weyl group} $W_Q$ of $Q$.
The fundamental cone is 
\[
 C_{\Pi_Q}\coloneqq \{\dd\in\BoN^{Q_0}\setminus\{0\}\mid \supp(\dd) \text{ is connected and }(\dd,1_i)_{\SG_2(\BoC Q)}\leq 0\text{ for every $i\in Q_0^{\real}$}\},
\]
where the support $\supp(\dd)$ of the dimension vector $\dd\in\BoN^{Q_0}$ is defined as the full subquiver containing the vertices $i\in Q_0$ such that $d_i\neq 0$.
In the definition of the fundamental cone, one can replace $Q_0^{\real}$ by $Q_0$ since if $i$ is a vertex with at least one loop, $(1_i,\dd)\leq 0$ automatically for any $\dd\in\BoN^{Q_0}$. We embed $Q_0^{\real}$ inside $\BoN^{Q_0}$ by sending $i\mapsto 1_i$. Following \cite{crawley2001geometry}, we let
\[
 R_{\Pi_Q}\coloneqq W_Q\cdot Q_0^{\real}\cup (\pm W_Q\cdot C_{\Pi_Q})
\]
be the set of roots of the quiver $Q$. We let $R_{\Pi_Q}^+\coloneqq R_{\Pi_Q}\cap \BoN^{Q_0}$ be the set of positive roots. As in [loc. cit] we define the set
\begin{multline*}
 \Sigma'_{\Pi_Q}\coloneqq \left\{\dd\in R_{\Pi_Q}^+\mid \text{ for any nontrivial decomposition }\dd=\sum_{j=1}^r\dd_j, \dd_j\in R_{\Pi_Q}^+\setminus\{0\}\right. \\ \left.\text{ one has }(2-(\dd,\dd)_{\SG_2(\BoC Q)})>\sum_{j=1}^r(2-(\dd_j,\dd_j)_{\SG_2(\BoC Q)})\right\}.
\end{multline*}

We now prove that the set $\Sigma'_{\Pi_Q}$ (which by \cite[Theorem 1.2]{crawley2001geometry} characterises dimension vectors for which the preprojective algebra $\Pi_Q$ admits simple representations) coincides with $\Sigma_{\Pi_Q}$ (whose definition does not require the knowledge of roots of the quiver). This is a refinement of the equivalence $(3)\iff (4)$ of \cite[Theorem~1.1]{crawley2001geometry}
\begin{lemma}
 The sets $\Sigma_{\Pi_Q}$ and $\Sigma_{\Pi_Q}'$ coincide.
\end{lemma}
\begin{proof}
By changing the conditions on the dimension vectors and on the decompositions allowed, we define a third set
\begin{multline*}
 \Sigma''_{\Pi_Q}\coloneqq \left\{\dd\in \BoN^{Q_0}\setminus\{0\}\mid \text{ for any nontrivial decomposition }\dd=\sum_{j=1}^r\dd_j, \dd_j\in R_{\Pi_Q}^+\setminus\{0\}\right. \\ \left.\text{ one has }(2-(\dd,\dd)_{\SG_2(\BoC Q)})>\sum_{j=1}^r(2-(\dd_j,\dd_j)_{\SG_2(\BoC Q)})\right\}
\end{multline*}
and a fourth one
 \begin{multline*}
 \Sigma'''_{\Pi_Q}\coloneqq \left\{\dd\in R^+_{\Pi_Q}\mid \text{ for any nontrivial decomposition }\dd=\sum_{j=1}^r\dd_j, \dd_j\in \BoN^{Q_0}\setminus\{0\}\right. \\ \left.\text{ one has }(2-(\dd,\dd)_{\SG_2(\BoC Q)})>\sum_{j=1}^r(2-(\dd_j,\dd_j)_{\SG_2(\BoC Q)})\right\}.
\end{multline*}

Clearly, $\Sigma'_{\Pi_Q}\subset \Sigma''_{\Pi_Q}\supseteq \Sigma_{\Pi_Q}\supseteq \Sigma'''_{\Pi_Q}$ and $\Sigma'''_{\Pi_Q}\subset \Sigma'_{\Pi_Q}$. We eventually prove that these sets all coincide.
 
We prove that $\Sigma_{\Pi_Q}=\Sigma''_{\Pi_Q}$. We let $\dd\in\Sigma''_{\Pi_Q}$. Let $\dd=\sum_{j=1}^r\dd_j$, $\dd_j\in\BoN^{Q_0}$ be a nontrivial decomposition of $\dd$. For $1\leq j\leq r$, we let $\dd_j=\sum_{u=1}^{s_{j}}\dd_{j,u}$ be the canonical decomposition of $\dd_j$ (see \cite{crawley2001geometry} for the definition of this decomposition). Then, by definition, $\dd_{j,u}\in R^+_{\Pi_Q}$ and by \cite[Lemma 4.5]{crawley2001geometry}
\[
 2-(\dd_j,\dd_j)_{\SG_2(\BoC Q)}\leq \sum_{u=1}^{s_{j}}(2-(\dd_{j,u}\dd_{j,u})_{\SG_2(\BoC Q)})
\]
for $1\leq j\leq r$. The decomposition $\dd=\sum_{j,u}\dd_{j,u}$ is such that $\dd_{j,u}\in\Sigma_{\Pi_Q}''\subset R^+_{\Pi_Q}$, and is nontrivial. So by definition of $\Sigma''_{\Pi_Q}$, we have
\[
 2-(\dd,\dd)_{\SG_2(\BoC Q)}>\sum_{j,u}(2-(\dd_{j,u}\dd_{j,u})_{\SG_2(\BoC Q)})\geq \sum_{j=1}^r(2-(\dd_j,\dd_j)_{\SG_2(\BoC Q)}).
\]
Therefore, $\dd\in\Sigma_{\Pi_Q}$.

Now, we prove the equality $\Sigma'''_{\Pi_Q}=\Sigma_{\Pi_Q}$. Let $\dd\in \Sigma_{\Pi_Q}$. All what we have to prove is $\dd\in R^+_{\Pi_Q}$. We will actually prove that $\dd$ is in the fundamental cone $C_{\Pi_Q}$. First, the support of $\dd$ has to be connected, as otherwise we can write $\dd=\dd_1+\dd_2$ with $\dd_1,\dd_2\neq 0$ and $(\dd_1,\dd_2)_{\SG_2(\BoC Q)}=0$. We then obtain
\[
 2-(\dd,\dd)_{\SG_2(\BoC Q)}=2-(\dd_1,\dd_1)_{\SG_2(\BoC Q)}-(\dd_2,\dd_2)_{\SG_2(\BoC Q)}<2-(\dd_1,\dd_1)_{\SG_2(\BoC Q)}+2-(\dd_2,\dd_2)_{\SG_2(\BoC Q)},
\]
contradicting the condition defining $\Sigma_{\Pi_Q}$.
Let $i\in Q_0^{\real}$. If $d_i=0$, then clearly, $(\dd,1_i)_{\SG_2(\BoC Q)}\leq 0$. If $\dd=d_i1_i$, then since $(1_i,1_i)_{\SG_2(\BoC Q)}=2$, we must have $d_i=1$ and clearly, $1_i\in\Sigma'''_{\Pi_Q}$. Otherwise, we write $\dd=1_i+\dd'$, with $\dd'=\dd-1_i\in\BoN^{Q_0}\setminus\{0\}$. By definition of $\Sigma_{\Pi_Q}$, we have
\[
 2-(\dd,\dd)_{\SG_2(\BoC Q)}>2-(1_i,1_i)_{\SG_2(\BoC Q)}+2-(\dd',\dd')_{\SG_2(\BoC Q)}
\]
and this simplifies to
\[
 (1_i,\dd')_{\SG_2(\BoC Q)}<-1.
\]
From this, we obtain $(\dd,1_i)_{\SG_2(\BoC Q)}=2+(1_i,\dd')_{\SG_2(\BoC Q)}\leq 0.$
We conclude that $\dd$ is in the fundamental cone.

We deduce that $\Sigma'_{\Pi_Q}\subset\Sigma'''_{\Pi_Q}$ and the obvious inclusion $\Sigma'''_{\Pi_Q}\subset\Sigma'_{\Pi_Q}$ allows us to conclude that all four sets $\Sigma_{\Pi_Q}, \Sigma'_{\Pi_Q}, \Sigma''_{\Pi_Q}$ and $\Sigma'''_{\Pi_Q}$ coincide.
\end{proof}

\subsection{Ext-quivers and roots}
Let $\CA$ be a 2-Calabi--Yau Abelian category. We let $\underline{\CF}=\{\CF_1,\hdots,\CF_r\}$ be a collection of simple objects of $\CA$ and let $\overline{Q}_{\underline{\CF}}$ be the Ext-quiver of $\underline{\CF}$. Recall the morphism of monoids $\lambda_{\underline{\CF}}\colon \BoN^{\underline{\CF}}\rightarrow \K(\CA)$
(see \S\ref{subsubsection:extquiver}).
\begin{lemma}
\label{lemma:pullbackEulerform}
 The pull-back by $\lambda_{\underline{\CF}}$ of the Euler form of $\SC$ is the Euler form of $\SG_2(\BoC Q_{\underline{\CF}})$: for any $\dd,\ee\in\BoN^{\underline{\CF}}$, $(\dd,\ee)_{\SG_2(Q_{\underline{\CF}})}=(\lambda_{\underline{\CF}}(\dd),\lambda_{\underline{\CF}}(\ee))_{\SC}$.
\end{lemma}
\begin{proof}
 This is a direct calculation using the definition of the Ext-quiver (\S\ref{subsubsection:extquiver}).
\end{proof}

We now define the sets of primitive positive roots $\Sigma_{\CA}$ and simple positive roots $\Phi_{\CA}^+$ for $\CA$ by
\[
 \Sigma_{\CA}\coloneqq \Sigma_{\pi_0(\CM_{\CA}),(-,-)_{\SC}};\quad\quad
 \Phi_{\CA}^+\coloneqq \Phi_{\pi_0(\CM_{\CA}),(-,-)_{\SC}}^+
\]
according to \S\ref{subsection:roots} for the monoid with bilinear form $(\pi_0(\CM_{\CA}),(-,-)_{\SC})$. We also have the roots for $\K(\CA)$, $\Sigma_{\K(\CA)}$ and $\Phi^+_{\K(\CA)}$.

\begin{proposition}
\label{proposition:rootsneighbourhood}
 \begin{enumerate}
 \item $\Sigma_{\CA}$ is the set of connected components $a\in\pi_0(\CM)$ such that $\CA$ has simple objects of class $a$, or equivalently such that $\JH_{\CA,a}$ is a $\mathbf{G}_m$-gerbe over an open subset of $\CM_{\CA,a}$,
 \item $\lambda_{\underline{\CF}}^{-1}(\Sigma_{\CA})=\Sigma_{\Pi_{Q_{\underline{\CF}}}}$,
 \item $\lambda_{\underline{\CF}}^{-1}(\Phi^+_{\CA})=\Phi^+_{\Pi_{Q_{\underline{\CF}}}}$. This last equality is compatible with the decompositions of simple positive roots into real, isotropic and hyperbolic roots, that is, $\lambda_{\underline{\CF}}^{-1}(\Phi^{+,\star}_{\CA})=\Phi^{+,\star}_{\Pi_{Q_{\underline{\CF}}}}$ for $\star\in\{\real, \iso,\hyp\}$.
 \end{enumerate}
\end{proposition}
\begin{proof}
 \begin{enumerate}
 \item This is known for $\CA=\Rep\Pi_Q$ by \cite[Theorem 1.2]{crawley2001geometry}. Since the stabilizer of a simple representation is $\mathbf{G}_m$, $\CA$ has a simple object of class $a$ if and only if $\JH_a$ is a $\mathbf{G}_m$-gerb over an open subset of $\CM_a$, this condition can be checked analytically locally and so (1) is a consequence of the local neighbourhood theorem and Lemma~\ref{lemma:pullbackEulerform}
 \item By $(1)$, the set $\Sigma_{\CA}$ is the set of connected components $a\in\pi_0(\CM_{\CA})$ over which $\JH_{\CA}$ is generically a $\mathbf{G}_m$-gerbe. This condition is preserved under restriction of $\JH_{\CA}$ to an \'etale neighbourhood of $\CM_{\CA,a}$. By the local neighbourhood theorem, if $\dd\in\BoN^{}$, then an \'etale neighbourhood of $\JH_{\Pi_{Q_{\underline{\CF}}}}$ around $0\in\CM_{\Pi_{Q_{\underline{\CF}}},\dd}$ is isomorphic to an \'etale neighbourhood of $\JH_{\CA}$ around the point corresponding to $\bigoplus_{i=1}^r\CF_i^{\oplus d_i}$. The point $(2)$ follows.
 \item This is a consequence of (1) and Lemma~\ref{lemma:pullbackEulerform}.
 \end{enumerate}
\end{proof}

\subsection{Geometry of the good moduli space}
\label{subsection:geometrygoodmodulispace}
Using the local neighbourhood theorem (Theorem~\ref{theorem:neighbourhood}) and the geometric properties of the moduli spaces of representations of preprojective algebras \cite{crawley2001geometry}, we derive geometric properties of the good moduli space of $2$-Calabi--Yau categories.

We define the function
\[
 p\colon\pi_0(\CM_{\CA})\rightarrow\BoZ;\quad\quad
  a        \mapsto  2-(a,a)_{\SC}.
\]
\begin{proposition}
\label{proposition:geometrygoodmodspace}
 \begin{enumerate}
 \item \label{item:geoimrealroot}Let $a\in\Sigma_{\CA}$. Then $\CM_{\CA,a}$ is irreducible of dimension $p(a)$. In particular, for $a\in\Sigma_{\CA}$ such that $(a,a)=0$, $\dim\CM_{\CA,a}=2$ and for $a\in\Phi_{\CA}^{+,\real}$, $\CM_{\CA,a}=\pt$,
 \item \label{item:irrdim}Let $a\in\Phi_{\CA}^{+,\iso}$, $a=la'$ with $l\in\BoN$ and $a'\in\Sigma_{\CA}$. Then $\CM_{\CA,la}$ is irreducible of dimension $2l$.
 \item \label{item:reglocus} Let $a\in\Sigma_{\CA}$ be an isotropic class. Then for any $l\in \BoN$ the map
 \[
 S^l\CM_{\CA,a}\rightarrow\CM_{\CA,la}
 \]
given by the direct sum induces an isomorphism between the underlying reduced schemes.
 \item \label{item:genstab}Let $a\in\Phi_{\CA}^{+,\iso}$. We write $a=la'$ for $a'\in\Sigma_{\CA}$ and $l\in\BoN_{\geq 1}$. Then, the stabilizer of a general element $x\in\CM_{\CA,a}$ is isomorphic to $(\BoC^*)^l$.
 \end{enumerate}
\end{proposition}
\begin{proof}
The property \eqref{item:geoimrealroot} is true for $\CA=\Rep\Pi_Q$ by \cite[Theorem 1.2]{crawley2001geometry} and is easily deduced for any $\CA$ by the local neighbourhood theorem (Theorem~\ref{theorem:neighbourhood}). Similarly, \eqref{item:irrdim} is true for preprojective algebras of quivers and since irreducibility and dimension of a connected component can be checked in the neigbourhoods of closed points, we obtain the statement for any $\CA$ by Theorem \ref{theorem:neighbourhood}. \eqref{item:genstab} is a direct consequence of \eqref{item:reglocus}. Therefore, we finish by proving \eqref{item:reglocus}. The direct sum gives a finite map $\oplus\colon\CM_{\CA,a}^l\rightarrow \CM_{\CA,la}$ (by Assumption~\ref{ds_fin}). It is $S_l$-invariant and the induced map $\oplus\colon S^l\CM_{\CA,a}\rightarrow\CM_{\CA,la}$ is proper and injective.  By (1) and (2), the algebraic varieties $(S^l\CM_{\CA,a})_{\red}$ and $(\CM_{\CA,la})_{\red}$ are of the same dimension $2l$. Therefore, $\oplus$ induces a homeomorphism $(S^l\CM_{\CA,a})_{\red}\rightarrow(\CM_{\CA,la})_{\red}$ between these two algebraic varieties. By \cite{crawley2002decomposition}, the map of the Proposition $S^l\CM_{\Pi_Q,\dd}\rightarrow \CM_{\Pi_Q,l\dd}$ induces an isomorphism between the underlying reduced schemes when $\CA=\Rep\Pi_Q$ and $\dd\in\BoN^{Q_0}$ is an indivisible imaginary root and so, by the local neighbourhood theorem, the morphism $(S^l\CM_{\CA,a})_{\red}\rightarrow(\CM_{\CA,la})_{\red}$ induces isomorphisms between the completions of local rings at all closed points. Therefore, it is an isomorphism.
\end{proof}

Let $a\in\Sigma_{\CA}\cap\Phi_{\CA}^{+,\iso}$ be an indivisible isotropic root. We let $\CM_{\CA,a}^{\simp}$ be the locus of simple objects. For $l\geq 1$ we let $\CM_{\CA,la}'\subset \CM_{\CA,la}$ be the locus of semisimple objects of $\CA$ whose simple factors are of class $a$ (so $\CM_{\CA,a}^{\simp}=\CM_{\CA,a}'$). By Proposition~\ref{proposition:geometrygoodmodspace} \eqref{item:reglocus}, $\CM_{\CA,la}'$ is open in $\CM_{\CA,la}$. We let $\FM_{\CA,la}'\subset\FM_{\CA,la}$ be the open substack parametrising objects of $\CA$ whose simple subquotients are of class $a$: we have a Cartesian diagram
\[
 \begin{tikzcd}
	{\FM'_{\CA,la}} & {\FM_{\CA,la}} \\
	{\CM'_{\CA,la}} & {\CM_{\CA,la}}
	\arrow["{\JH_{\CA,la}}", from=1-2, to=2-2]
	\arrow["{\JH'_{\CA,la}}"', from=1-1, to=2-1]
	\arrow["{\tilde{\jmath}}"', from=2-1, to=2-2]
	\arrow["\jmath", from=1-1, to=1-2]
	\arrow["\lrcorner"{anchor=center, pos=0.125}, draw=none, from=1-1, to=2-2]
\end{tikzcd}
\]
and the isomorphism of Proposition~\ref{proposition:geometrygoodmodspace} \eqref{item:reglocus} induces an isomorphism of reduced schemes $(\CM'_{\CA,la})_{\red}\cong (S^l\CM_{\CA,a}')_{\red}$.

\section{The cohomological Hall algebra and the BPS algebra}
We let $\CA$ be a $2$-Calabi--Yau Abelian category, and $\JH\colon\FM_{\CA}\rightarrow\CM_{\CA}$ be its moduli stack of objects and good moduli space as in \S\ref{section:modulistack2dcats}.

\subsection{The cohomological Hall algebra}
\label{subsection:theCoHA}
In \cite[\S 7]{davison2022BPS}, we define the \emph{relative cohomological Hall algebra} of $\CA$ as an algebra structure on the complex of mixed Hodge modules $\ulrelCoHA_{\CA}\coloneqq (\JH_{\CA})_*\BD\ulBoQ_{\FM_{\CA}}^{\vir}$. Here, $\BoQ_{\FM_{\CA}}^{\vir}\coloneqq\BoQ_{\FM_{\CA}}\otimes\BoL^{(-,-)_{\SC}/2}$  where $(-,-)_{\SC}$ is the Euler form of the dg-category $\SC$.

\begin{theorem}[\cite{davison2022BPS}]
Under Assumptions \ref{p_assumption}-\ref{BPS_cat_assumption} there is a virtual pullback along $q$ and a pushforward along $p$, yielding the cohomological Hall algebra product
\begin{equation}
\label{eq:CoHAproduct}
\mathfrak{m} \colon \JH_* \BD\ulBoQ_\Mst^{\vir} \boxdot \JH_* \BD \ulBoQ_{\Mst}^{\vir}\coloneqq \oplus_*(\JH\times\JH)_* \BD\ulBoQ^{\vir}_{\Mst\times\Mst} \longto \JH_* \BD\ulBoQ_{\Mst}^{\vir}.
\end{equation}
\end{theorem}

\begin{assumption}{7}
\label{ass:associativity}
The cohomological Hall algebra product $\mathfrak{m}$ endows the mixed Hodge module complex $\JH_* \BD \ul{\BoQ}_{\Mst}^{\vir}$ 
with the structure of an associative algebra in the symmetric monoidal category $(\D^{+}(\MHM(\CM)),\boxdot)$.
\end{assumption}

\begin{remark}
Assumption~\ref{ass:associativity} is shown to hold for 2-Calabi--Yau categories of geometric and algebraic origin in the appendix of \cite{davison2022BPS}.
\end{remark}

This means that that the morphism \eqref{eq:CoHAproduct} $
 \mult\colon\ulrelCoHA_{\CA}\boxdot\ulrelCoHA_{\CA}\rightarrow\ulrelCoHA_{\CA}$ satisfies the standard axioms of associativity and unitality in monoidal categories, where the tensor structure $\boxdot$ on $\CD^+(\MHM(\CM))$ is defined as in \S\ref{subsection:GKMtensorcategories}. The product is defined via (virtual) pullback and pushforward from the stack of short exact sequences of objects in $\CA$. The mixed Hodge module $\ulBoQ_{\FM_{\CA}}^{\vir}\in\CD^+(\MHM(\FM_{\CA}))$ is the twisted constant mixed Hodge module $\ulBoQ_{\FM_{\CA,a}}\otimes\mathbf{L}^{(a,a)_{\SC}/2}$ on the connected component $\FM_{\CA,a}$ of $\FM_{\CA}$. We are committing here a standard abuse of notation: while we do not define anywhere the derived category of mixed Hodge modules on the stack $\FM_{\CA}$, the direct image $(\JH_{\CA})_*\BD\ulBoQ_{\FM_{\CA}}^{\vir}$ is still defined, as an unbounded complex of mixed Hodge modules: see \cite[\S 2.2]{davison2020cohomological} for details.

The \emph{absolute cohomological Hall algebra} is obtained by taking the derived global sections $\HO^*(\ulrelCoHA_{\CA})$, with the multiplication $\HO^*(\mult)$. If $\CB\subset \CA$ is a Serre subcategory as in \S\ref{subsubsection:serre} and $\imath\colon \CM_{\CB}\hookrightarrow\CM_{\CA}$ is the inclusion of the submonoid parametrising semisimple objects of $\CA$ contained in $\CB$, we let $\FM_{\CB}$ be the stack defined by the following Cartesian diagram:
\[
 \begin{tikzcd}
	{\mathfrak{M}_{\mathcal{B}}} & {\mathfrak{M}_{\CA}} \\
	{\mathcal{M}_{\mathcal{B}}} & {\mathcal{M}_{\CA}.}
	\arrow[from=1-1, to=1-2]
	\arrow["\JH", from=1-2, to=2-2]
	\arrow["\imath",from=2-1, to=2-2]
	\arrow[from=1-1, to=2-1]
	\arrow["\lrcorner"{anchor=center, pos=0.125}, draw=none, from=1-1, to=2-2]
\end{tikzcd}
\]
Then, by base change, the morphism $\imath^!\mult$ gives $\imath^!\ulrelCoHA_{\CA}\in\CD^+(\MHM(\CM_{\CB}))$ an algebra structure. It is the \emph{relative cohomological Hall algebra} of $\CB$, and $\HO^*(\imath^!\ulrelCoHA_{\CA})$ is the \emph{absolute cohomological Hall algebra} of $\CB$. By base change in the diagram
\[
 \begin{tikzcd}
	{\CM_{\CB}\times \CM_{\CB}} & {\CM_{\CB}} \\
	{\CM_{\CA}\times \CM_{\CA}} & {\CM_{\CA}}
	\arrow["\imath", from=1-2, to=2-2]
	\arrow["\oplus", from=1-1, to=1-2]
	\arrow["\oplus"', from=2-1, to=2-2]
	\arrow["{\imath\times \imath}"', from=1-1, to=2-1]
	\arrow["\lrcorner"{anchor=center, pos=0.125}, draw=none, from=1-1, to=2-2],
\end{tikzcd}
\]
which we implicitly assumed to be Cartesian (we say that $\CM_{\CB}$ is a \emph{saturated submonoid} of $\CM_{\CA}$) the functor $\imath^!\colon \CD^+(\MHM(\CM_{\CA}))\rightarrow \CD^+(\MHM(\CM_{\CB}))$ is strictly monoidal.

The cases most relevant to this paper are the following.
\begin{enumerate}
 \item $\CB$ is the Serre subcategory of $\CA$ generated by a collection of pairwise non-isomorphic simple objects of $\CA$, $\underline{\CF}=\{\CF_1,\hdots,\CF_r\}$. Then, $\CM_{\CB}\cong \BoN^{\underline{\CF}}$ is a discrete submonoid of $\CM_{\CA}$. Therefore we may identify $\HO^*(\imath^!\ulrelCoHA_{\CA})=\imath^!\ulrelCoHA_{\CA}$. It is called the cohomological Hall algebra of $\underline{\CF}$. If $\CA=\Rep\Pi_Q$ and $\underline{\CF}$ is the collection of nilpotent one-dimensional representations of $\Pi_Q$, $\underline{\CF}=\CN=\{S_i\colon i\in Q_0\}$, we obtain the \emph{nilpotent cohomological Hall algebra}, $\ulrelCoHA_{\Pi_Q}^{\CN}$.
 
 \item $\CA=\Rep\Pi_Q$ and $\CB$ is the category of \emph{strictly semi-nilpotent representations} of $\Pi_Q$\footnote{A representation $M$ of $\Pi_Q$ is called \emph{strictly semi-nilpotent} if there exists a filtration $0=M_0\subset M_1\subset\hdots\subset M_r=M$ with consecutive subquotients supported at a single vertex, such that for $\alpha\in Q_1$, $x_{\alpha}M_i\subset M_i$ and for $\alpha\in Q_1^*$, or $x_{\alpha}$ not a loop, $x_{\alpha}M_i\subset M_{i-1}$.} (in the sense of \cite{bozec2020number}, see also \cite[\S 7.2.2]{davison2022BPS}). Then, $\CM_{\Pi_Q}^{\SSN}\coloneqq\CM_{\CB}$ is the closed submonoid of $\CM_{\CA}$ of semisimple representations of $\Pi_Q$ such that the only arrows of $\overline{Q}$ acting possibly nontrivially are the loops in $Q_1\subset \overline{Q}_1=Q_1\sqcup Q_1^*$. We call $\ulrelCoHA_{\Pi_Q}^{\SSN}\coloneqq \imath^!\ulrelCoHA_{\CA}$ the \emph{relative strictly semi-nilpotent cohomological Hall algebra} and $\HO^*(\ulrelCoHA_{\Pi_Q}^{\SSN})$ is the \emph{strictly semi-nilpotent cohomological Hall algebra}.
\end{enumerate}
The proof of Theorem~\ref{theorem:BPSalgisenvBorcherds} uses in a crucial way that one can describe completely explicitly the \emph{degree zero strictly semi-nilpotent cohomological Hall algebra} $\HO^0(\ulrelCoHA_{\Pi_Q}^{\SSN})$ -- see Theorem~\ref{theorem:degree0SSN}.

\subsection{The BPS algebra}
\label{subsection:BPSalg}
We let $\CA$, $\JH_{\CA}\colon\FM_{\CA}\rightarrow\CM_{\CA}$, $\ulrelCoHA_{\CA}=(\JH_{\CA})_*\BD\ulBoQ_{\FM_{\CA}}^{\vir}\in\CD^+(\MHM(\CM_{\CA}))$ and $\mult\colon\ulrelCoHA_{\CA}\boxdot\ulrelCoHA_{\CA}\rightarrow \ulrelCoHA_{\CA}$ be as in \S\ref{subsection:gsetup} and \S\ref{subsection:theCoHA}. In particular, we make Assumptions \ref{p_assumption}--\ref{BPS_cat_assumption}.

\begin{lemma}
\label{lemma:nonnegativeperverse}
 The complex of mixed Hodge modules $\ulrelCoHA_{\CA}$ is concentrated in nonnegative cohomological degrees: $\mathcal{H}^i(\ulrelCoHA_{\CA})=0$ for $i<0$.  The mixed Hodge module $\CH^{0}(\ulrelCoHA_{\CA})$ is semisimple.
\end{lemma}
\begin{proof}
 This is part of \cite[Theorem A]{davison2020bps}, which shows, in particular, that the mixed Hodge module $\CH^{0}(\ulrelCoHA_{\CA})$ is pure of weight zero.
\end{proof}

It follows that $\tau^{\leq 0}\ulrelCoHA_{\CA}\cong \CH^0(\ulrelCoHA_{\CA})$. Since $\boxdot$ is bi-exact, it also follows that $\tau^{\leq 0}(\ulrelCoHA_{\CA}\boxdot \ulrelCoHA_{\CA})\cong \CH^0(\ulrelCoHA_{\CA}\boxdot \ulrelCoHA_{\CA})\cong \CH^0(\ulrelCoHA_{\CA})\boxdot \CH^0(\ulrelCoHA_{\CA})$.  By taking the degree zero cohomology mixed Hodge modules, we obtain an associative algebra $(\CH^{0}(\ulrelCoHA_{\CA}),\CH^0(\mult))$ in $(\MHM(\mathcal{M}),\boxdot)$.  The operation $\CH^0(\mult)$ is associative because $\mult$ is.

We set $\ul{\BPS}_{\CA,\Alg}\coloneqq\CH^{0}(\ulrelCoHA_{\CA})$ with its algebra structure given above. This algebra object is called \emph{the relative BPS algebra}. Its cohomology (derived global sections) is denoted $\underline{\aBPS}_{\CA,\Alg}:=\HO^*(\ul{\BPS}_{\CA,\Alg})$ and called \emph{the BPS algebra}.

\begin{proposition}
The natural morphism of algebras $\underline{\aBPS}_{\CA,\Alg}\rightarrow \HO^*(\ulrelCoHA_{\CA})$ is injective.
\end{proposition}
\begin{proof}
 The decomposition theorem for 2CY categories of \cite[Theorem B]{davison2021purity} asserts in particular that the underlying complex of mixed Hodge modules $\ulrelCoHA_{\CA}$ is the direct sum of its shifted cohomology mixed Hodge modules and so the natural map $\ul{\BPS}_{\CA,\Alg}\rightarrow \ulrelCoHA_{\CA}$ coming from the adjunction $\tau^{\geq 0}\rightarrow \id$ admits a retraction. Therefore, the map $\underline{\rmBPS}_{\CA,\Alg}\rightarrow\HO^*(\ulrelCoHA_{\CA})$ obtained by taking derived global sections also admits a retraction and is injective.
\end{proof}

\subsection{$\psi$-twists}
\subsubsection{Twisted multiplication}
\label{subsubsection:twistmultiplication}
As in \cite{davison2020cohomological,davison2022BPS}, we work with a twist of the multiplication of the cohomological Hall algebra by a bilinear form $\psi$. This twist is necessary to have the structure of an enveloping algebra on the BPS algebra (it appears already in the PBW theorem for CoHAs of Jacobi algebras, see \cite[\S 1.6]{davison2020cohomological}). In vague terms (made precise in \cite{hennecart2022geometric}), this twist exchanges the specialisation at $q=-1$ of the quantum group $\UEA_q(\mathfrak{g})$ with the enveloping algebra $\UEA(\mathfrak{g})$ (the specialisation $q=1$). In this section, we explain this twist.

As in \S\ref{subsubsection:BPSalg2CY} we choose a monoid $\K(\CA)$ and fix a morphism of monoids $\cl\colon\pi_0(\CM_{\CA})\rightarrow \K(\CA)$ through which the Euler form factors. We assume again that $\K(\CA)$ can be embedded in a lattice and that $\cl$ has finite fibres. We let $\psi\colon\K(\CA)\times\K(\CA)\rightarrow\BoZ/2\BoZ$ be a bilinear form such that for any $a,b\in\K(\CA)$,
\[
 \psi(a,b)+\psi(b,a)\equiv (a,b)_{\SC}\pmod{2}.
\]
Such a bilinear form exists since for any $a\in\pi_0(\CM_{\CA})$, $(a,a)_{\SC}\equiv 0\pmod{2}$. Indeed, it follows that any matrix $M$ representing the form induced by $(-,-)_{\SC}$ on $\K(\CA)$ has even entries on the diagonal, and we may define $\psi(-,-)$ via the matrix obtained by replacing all entries in $M$ above the diagonal by zeros.

For such a bilinear form $\psi$, we let $\ulrelCoHA_{\CA}^{\psi}$ be the relative cohomological Hall algebra with the multiplication $\mult_{a,b}^{\psi}\coloneqq(-1)^{\psi(a,b)}\mult_{a,b}\colon \ulrelCoHA_{\CA}^{\psi}\boxtimes\ulrelCoHA_{\CA}^{\psi}\rightarrow\ulrelCoHA_{\CA}^{\psi}$. As complexes of mixed Hodge modules on $\CM_{\CA}$, $\ulrelCoHA_{\CA}^{\psi}=\ulrelCoHA_{\CA}$. The twisted multiplication induces twisted multiplications $\mult^{\psi}$ on the relative BPS algebra $\underline{\BPS}_{\CA,\Alg}$, on the absolute CoHA $\HO^*\!\!\ulrelCoHA_{\CA}$ and the absolute BPS algebra $\underline{\rmBPS}_{\CA,\Alg}$. We denote by $\ul{\BPS}_{\CA,\Alg}^{\psi}$, $\HO^*\!\!\ulrelCoHA_{\CA}^{\psi}$ and $\underline{\rmBPS}_{\CA,\Alg}^{\psi}$ the algebra objects with the twisted multiplications.

\subsubsection{Induced twist for Ext-quivers}
\label{subsubsection:inducedtwists}
Let $\CA$ be an Abelian subcategory of a 2CY category as in \S\ref{section:modulistack2dcats}. If $\underline{\CF}$ is a collection of simple objects of $\CA$, we have the Ext-quiver $\overline{Q}_{\underline{\CF}}$ (\S\ref{subsubsection:extquiver}) and the preprojective algebra $\Pi_{Q_{\underline{\CF}}}$. Using the morphism of monoids \eqref{equation:morphismmonoids}, we obtain by pullback a bilinear form $\psi$ on $\BoN^{(Q_{\underline{\CF}})_0}$. By Lemma~\ref{lemma:pullbackEulerform}, for any $\dd,\dd'\in\BoN^{(Q_{\underline{\CF}})_0}$, $\psi(\dd,\dd')+\psi(\dd',\dd)\equiv (\dd,\dd')_{\SG_2(\BoC Q_{\underline{\CF}})}\pmod{2}$.

\subsection{Primitive summands}
\label{subsec:generators}
\subsubsection{Real and hyperbolic primitive summands}
In this section, we define the Chevalley generators for primitive positive roots $m\in \Sigma_{\K(\CA)}$.

\begin{lemma}[{\cite[Theorem 6.6]{davison2021purity}}]
\label{lemma:inclusionICBPSalg}
 For any $m\in \K(\CA)$ such that $\JH_{\CA,m}$ is a $\mathbf{G}_m$-gerb over an open subset (that is, $m\in\Sigma_{\CA}$, see \S\ref{subsection:rootsforPi}), we have a canonical monomorphism
 \[
 \ul{\IC}(\CM_{\CA,m})\rightarrow \ulBPS_{\CA,\Alg}^{(\psi)}.
 \]
\end{lemma}

\subsubsection{Isotropic primitive summands}
\label{subsubsection:isotropicprimitivesummands}
In this section, we describe the generators for isotropic positive roots.

Let $a\in\Sigma_{\CA}$ be isotropic. We let
\[
 \Delta_l\colon\CM'_{\CA,a}\rightarrow \CM'_{\CA,la}
\]
be the diagonal inclusion (see \S\ref{subsection:geometrygoodmodulispace} for the meaning of the prime symbols).

For $a\in\Sigma_{\CA}$, such that $(a,a)_{\SC}=0$ ($a$ is isotropic), we let $\FM'_{\CA,\bullet a}=\bigsqcup_{l\geq 0}\FM'_{\CA,la}$ and similarly for $\CM'_{\CA,\bullet a}$. We let $\JH'_{\CA,\bullet a}\colon \FM'_{\CA,\bullet a}\rightarrow \CM'_{\CA,\bullet a}$ be the restriction of the map $\JH_{\CA}$. We define $\ulrelCoHA'^{\psi}_{\CA}\coloneqq \tilde{\jmath}^*\ulrelCoHA^{\psi}_{\CA}$, with $\tilde{\jmath}$ defined as in \S \ref{subsection:geometrygoodmodulispace}. The mixed Hodge module complex $\ulrelCoHA'^{\psi}_{\CA}$ inherits a multiplication morphism given by restricting $\mathfrak{m}^{\psi}\colon\ulrelCoHA^{\psi}_{\CA}\boxdot \ulrelCoHA^{\psi}_{\CA}\rightarrow \ulrelCoHA^{\psi}_{\CA}$.
\begin{theorem}
\label{theorem:BPSisotropic}
 Let $a\in\Sigma_{\CA}$ be isotropic. We have a PBW isomorphism of mixed Hodge module complexes
 \begin{equation}
 \label{IsotPBW}
\Sym_{\boxdot}\left(\bigoplus_{l\geq 1}(\Delta_l)_*\ulBoQ_{\CM'_{\CA,a}}\otimes\BoL^{-1}\otimes \HO^*_{\BoC^*}\right)\cong (\JH'_{\CA,\bullet a})_*\BD\ulBoQ_{\FM'_{\CA,\bullet a}}^{\vir}=\ulrelCoHA'^{\psi}_{\CA},
 \end{equation}
 i.e. a morphism $\bigoplus_{l\geq 1}(\Delta_l)_*\underline{\BoQ}_{\CM'_{\CA,a}}\otimes\BoL^{-1}\otimes \HO^*_{\BoC^*}\rightarrow \ulrelCoHA'_{\CA}$ inducing the above morphism via the product structure on the target. The morphism \eqref{IsotPBW} induces an isomorphism of algebra objects in $(\MHM(\CM'_{\CA,\bullet a}),\boxdot)$:
\[
 \CH^{0}((\JH'_{\CA,\bullet a})_*\underline{\BoQ}_{\FM'_{\CA,\bullet a}})=\tau^{\leq 0}(\JH'_{\CA,\bullet a})_*\BD\underline{\BoQ}_{\FM'_{\CA,\bullet a}}^{\vir}\cong \Sym_{\boxdot}\left(\bigoplus_{l\geq 1}(\Delta_l)_*\ulBoQ_{\CM'_{\CA,a}}\otimes\BoL^{-1}\right).
\]
\end{theorem}

\begin{proof}
We have a Cartesian diagram
\[
 \begin{tikzcd}
	{\FM'_{\CA,a}} & {\FM'_{\CA,la}} \\
	{\CM'_{\CA,a}} & {\CM'_{\CA,la}}
	\arrow["{\Delta_l}"', from=2-1, to=2-2]
	\arrow["{\JH'_{\CA,la}}", from=1-2, to=2-2]
	\arrow["{q_l}"', from=1-1, to=2-1]
	\arrow["{\tilde{\Delta_l}}", from=1-1, to=1-2]
	\arrow["\lrcorner"{anchor=center, pos=0.125}, draw=none, from=1-1, to=2-2].
\end{tikzcd}
\]
The objects of $\CA$ parametrised by the closed points $x$ in the image of $\Delta_l$ are of the form $S^{\oplus l}$ where $S$ is a simple object of $\CA$ of class $a$. The Ext-quiver of $S$ is the double of the Jordan quiver (it has one vertex -- corresponding to $S$, and $1=\frac{1}{2}(2-(a,a)_{\SC})$ arrows) and therefore, by the local neighbourhood theorem (Theorem~\ref{theorem:neighbourhood}), the map $q_l$ is analytically locally around $x$ modeled on the map $\BA^2\times\FN_l\rightarrow\BA^2$ where $\FN_l$ is the stack of pairs of commuting nilpotent $l\times l$ matrices.

For the first statement, we can now run the arguments of the proof of \cite[Proposition A.4]{davison2021nonabelian}. That the analogue in our situation of the morphism $\Psi_X$ of \emph{loc.cit.} is an isomorphism follows from the local neighbourhood theorem and the case of the Jordan quiver (or torsion sheaves on $\BA^2$, see \cite[Proposition A.1]{davison2021nonabelian} and the references therein). For the second statement, again we appeal to \cite[Proposition A.4]{davison2021nonabelian}, except in [loc. cit] the $\psi$-twist is not considered, or equivalently we consider the $\psi$-twist for which $\psi(a,a)=0$, and not $\psi'$ satisfying $\psi'(a,a)=1$ (the only nontrivial choice). On the other hand, the commutator for $\tau^{\leq 0}(\ulrelCoHA'^{\psi}_{\CA})$ vanishes if and only if the commutator for $\tau^{\leq 0}(\ulrelCoHA'^{\psi'}_{\CA})$ vanishes, and the twisted case follows from the untwisted case.
\end{proof}

The following two corollaries are direct consequences of Theorem~\ref{theorem:BPSisotropic}.
\begin{corollary}
\label{corollary:restrictionBPSiso}
For any $a\in\Sigma_{\CA}$ such that $(a,a)_{\SC}=0$ and any $r\geq 1$, we have 
\[
\tilde{\jmath}^*\ulBPS_{\CA,\Alg,ra}^{(\psi)}\cong \left(\Sym_{\boxdot}\left(\bigoplus_{l\geq 1}(\Delta_l)_*\ulBoQ_{\CM'_{\CA,a}}\otimes\BoL^{-1}\right)\right)_{ra}.
\]
\end{corollary}

\begin{corollary}
\label{corollary:isoprimsummand}
Choose $m\in\K(\CA)$ such that $m=rm'$ with $r\in\BoN_{\geq 1}$, $(m',m')=0$ and $m'\in\Sigma_{\CA}$. Then, we have a canonical monomorphism $(u_m)_*\ul{\IC}(\CM_{\CA,m'})\rightarrow\ul{\BPS}_{\CA,\Alg}^{(\psi)}$.
 \end{corollary}
 \begin{proof}
 This follows from Corollary~\ref{corollary:restrictionBPSiso} and the fact that $\tilde{\jmath}_{!*}\tilde{\jmath}^*\ulBPS_{\CA,\Alg,rm'}^{(\psi)}$ is a direct summand of $\ulBPS_{\CA,\Alg,rm'}^{(\psi)}$ by semisimplicity of $\ulBPS_{\CA,\Alg,rm'}^{(\psi)}$ (Lemma~\ref{lemma:nonnegativeperverse}).
 \end{proof}
\subsubsection{Generators}
\label{subsection:generators}

In this section, we define a mixed Hodge module $\underline{\SG}_{\CA}$ on $\CM_{\CA}$ which will be used for the study of the cohomological Hall algebra of $\CA$, eventually playing the role of the Chevalley generators of (half) a generalised Kac--Moody algebra. More precisely, $\underline{\SG}\coloneqq\bigoplus_{m\in\Phi^+_{\CA}}\underline{\SG}_{\CA,m}$ where
\[
 \underline{\SG}_{\CA,m}\coloneqq
 \begin{cases}
 \underline{\IC}(\CM_{\CA,m})&\text{ if $m\in\Sigma_{\CA}$}\\
 (u_m)_*\underline{\IC}(\CM_{\CA,m'})&\text{ if $m=lm'$ with $m'\in\Sigma_{\CA}$ isotropic and $l\geq 2$ },
 \end{cases}
\]
where 
\begin{align*}
u_m&\coloneqq \CM_{\CA,m'}\rightarrow \CM_{\CA,m}\\
&x\mapsto x^{\oplus l}
\end{align*}
is the direct sum map. We implicitly use that the decomposition $m=lm'$ is unique as by Proposition~\ref{proposition:geometrygoodmodspace}\eqref{item:reglocus}, $l$ is the dimension of the stabiliser of a general element of $\CM_{\CA,la}$ and $(\CM_{\CA,a})_{\red}$ can be identified with the small diagonal inside $\CM_{\CA,a}$, that is the closure of the locus of points having $\GL_l(\BoC)$ as stabiliser.

For $m\in\K(\CA)$, we let $\underline{\SG}_{\CA,m}\coloneqq \bigoplus_{m'\in\cl^{-1}(m)}\underline{\SG}_{\CA,m'}$. We set $\underline{\SG}\coloneqq \bigoplus_{m\in\Phi_{\K(\CA)}^+}\underline{\SG}_{\CA,m}$.

\subsubsection{Primitivity in the preprojective case}
\label{subsubsec:ppre}

We finish this section by showing that, at least in the case $\CA=\Rep(\Pi_Q)$, the summands $\underline{\SG}_{\CA,m}$ of the BPS algebra that we have defined deserve to be called primitive.
\begin{proposition}
\label{proposition:directcomplementprimitive}
Let $Q$ be a quiver. We let
\[
\underline{\SG}_{\dd} \coloneqq \underline{\SG}_{\Pi_Q,\vec{d}}= \begin{cases}
\ul{\IC}(\CM_{\Pi_Q,\dd}) &\text{ if $\dd\in\Sigma_{\Pi_Q}$}\\
(u_{\dd})_*\ul{\IC}(\CM_{\Pi_Q,\delta}) &\text{ if $\dd=l\delta$ with $l\geq 1$ and $\delta\in\Sigma_{\Pi_Q}$ is isotropic.}
\end{cases}
\]
Here, $u_{\dd}\colon \CM_{\Pi_Q,\delta}\rightarrow \CM_{\Pi_Q,l\delta}$; $x\mapsto x^{\oplus l}$. We let $\underline{\SG}=\bigoplus_{\dd\in\Phi_{\Pi_Q}^+}\underline{\SG}_{\dd}$. Then, there exists a canonical decomposition $\ulBPS_{\Pi_Q,\Alg}\cong \underline{\SF}\oplus\underline{\SG}$ with $\underline{\SF}\in\MHM(\CM_{\Pi_Q})$ such that for any $\dd,\dd'\in\BoN^{Q_0}\setminus \{0\}$, the multiplication
\[
 \mult^{\psi}\colon \ulBPS_{\Pi_Q,\Alg,\dd}^{\psi}\boxdot\ulBPS_{\Pi_Q,\Alg,\dd'}^{\psi}\rightarrow \ulBPS_{\Pi_Q,\Alg,\dd+\dd'}^{\psi}
\]
factors through the natural inclusion of $\underline{\mathscr{F}}$.
\end{proposition}

\begin{proof}
 If $\dd\in\BoN^{Q_0}\setminus\Phi_{\Pi_Q}^+$, we set $\underline{\SG}_{\dd}=0$, and $\underline{\SF}_{\dd}=\ulBPS_{\Pi_Q,\Alg,\dd}^{(\psi)}$.
 
 If $\dd\in\Sigma_{\Pi_Q}$, by Lemma~\ref{lemma:inclusionICBPSalg}, $\ul{\IC}(\CM_{\Pi_Q,\dd})$ is a direct summand appearing with multiplicity one in the mixed Hodge module $\ulBPS_{\Pi_Q,\Alg,\dd}^{(\psi)}$. Therefore, it has a canonical complement $\underline{\SF}_{\dd}$.
 
 If $\dd=l\delta$ with $l\geq 1$ and $\delta\in\Sigma_{\Pi_Q}$ isotropic, then the support of $\delta$ is an affine subquiver of $Q$. By the description of the BPS sheaf for affine quivers \cite[\S 7.2]{davison2020bps}, and the relative cohomological integrality theorem, we have a canonical isomorphism of mixed Hodge modules $\ulBPS_{\Pi_Q,\Alg,\dd}^{(\psi)}=(\Sym_{\boxdot}(\ulBPS_{\Pi_Q,\Lie}))_{\dd}$. It follows from Theorem~\ref{theorem:BPSisotropic} that $(u_{\dd})_*\ul{\IC}(\CM_{\Pi_Q,\delta})$ is the unique direct summand of $\ulBPS_{\Pi_Q,\Alg,\dd}^{(\psi)}$ supported on the small diagonal. This proves that it has a canonically defined complement $\underline{\SF}_{\dd}\subset \ulBPS_{\Pi_Q,\Alg,\dd}^{(\psi)}$.
 
 By the same arguments as in \cite[Theorem 7.1]{davison2020bps} for hyperbolic positive roots and \cite[\S 7.2]{davison2020bps} for $Q$ affine (which amount to support considerations), the multiplication map
 \[
 \mult^{\psi}\colon \ulBPS_{\Pi_Q,\Alg,\dd}^{\psi}\boxdot \ulBPS_{\Pi_Q,\Alg,\dd'}^{\psi}\rightarrow\ulBPS_{\Pi_Q,\Alg,\dd+\dd'}^{\psi}
 \]
factors through the inclusion of $\underline{\SF}_{\dd+\dd'}$ for any pair of nonzero $\dd,\dd'\in\BoN^{Q_0}\setminus\{0\}$.
\end{proof}

\section{The BPS (Lie) algebra for 2-Calabi--Yau categories -- Proof of Theorem~\ref{theorem:BPSalgisenvBorcherds}}
\label{section:BPSalgCohint}

We let $\CA$ be a $2$-Calabi--Yau Abelian category as in \S\ref{subsection:gsetup}, satisfying Assumptions~\ref{p_assumption}-\ref{BPS_cat_assumption} of \S\S\ref{subsection:assumptionsCoHA},\ref{subsec:spaceassumptions}. We fix a bilinear form $\psi$ as in \S\ref{subsubsection:twistmultiplication} and we denote by the same letter $\psi$ the induced bilinear form on $\BoN^{Q_0}$ if $Q=(Q_0,Q_1)$ is an Ext-quiver of a collection of simple objects of $\CA$ as in \S\ref{subsubsection:inducedtwists}.

\subsection{Local neighbourhoods and $\Sigma$-collections}
\label{subsection:locneighbSigmacoll}
Let $\underline{\CF}=\{ \CF_1,\ldots,\CF_r \}$ be a $\Sigma$-collection in $\CA$, satisfying the assumptions of Theorem~\ref{theorem:neighbourhood}. Define the classes $a_i = [\CF_i] \in \K(\CA)$ and let $x_i \in \Msp_{\CA,a_i}$ be the $\BoC$-point corresponding to $\CF_i$. The inclusions $x_i \into \Msp_{\CA}$ induce a monoid morphism $\imath_{\underline{\CF}} \colon \BoN^{\underline{\CF}} \into \Msp_{\CA}$ sending $1_{\CF_i}\mapsto x_i$ (see \S\ref{subsubsection:extquiver}).

Let $Q$ be half of the Ext-quiver of the collection $\underline{\CF}$. 
The inclusions $\{S_i\} \into \Msp_{\Pi_{Q},1_i}$ of the simple nilpotent representations of $\Pi_Q$ induce a monoid morphism $\imath_{\CN} \colon \BoN^{\underline{\CF}} \into \Msp_{\Pi_{Q}}$.
By \cite[Lemma 3.4]{davison2022BPS}, $\imath_{\CN}^!$ and $\imath_{\underline{\CF}}^!$ are strictly monoidal.

For $\vec{m}\in \BoN^{\underline{\CF}}$ we write $a_{\vec{m}} \coloneqq \sum_{i}m_ia_i$. Pick for every $\vec{m} \in \BoN^{\underline{\CF}}$ an analytic Ext-quiver neighbourhood
$\CU_{\vec{m}}$ of the point $x_\vec{m} \in \Msp_{\CA,a_{\vec{m}}}$ as in Theorem~\ref{theorem:neighbourhood},
corresponding to the semisimple object $\bigoplus_{i} \CF_i^{\oplus m_i}$ of
class $a_{\vec{m}}$. Set $\CU = \bigsqcup_{\vec{m} \in \BoN^{I}} \CU_{\vec{m}}$.
By Theorem~\ref{theorem:neighbourhood} we have a commutative diagram of analytic spaces 
\begin{equation}
\label{equation:diagramneighbourhood}
\begin{tikzcd}
&\BoN^{\underline{\CF}} \ar[dl,"\imath_{\CN}"',hook']\ar[d,"y",hook] \ar[dr,"\imath_{\underline{\CF}}",hook] & \\
 \Msp_{\Pi_Q} & \CU \ar[l,"\jmath_{\CN}"',hook']\ar[r,"\jmath_{\underline{\CF}}",hook] & \Msp_{\CA}.
\end{tikzcd}
\end{equation}
The horizontal morphisms $\jmath_{\CN},\jmath_{\underline{\CF}}$ are analytic-open embeddings such that $\jmath_{\CN}(\CU)$ (resp. $\jmath_{\underline{\CF}}(\CU)$) is an analytic neighbourhood of $\imath_{\CN}(\BoN^{\underline{\CF}})$ (resp. $\imath_{\underline{\CF}}(\BoN^{\underline{\CF}})$).

\subsection{Restriction of the BPS algebra sheaf}
\begin{lemma}
\label{lemma:restBPSalgfibre}
 If $\ulBPS_{\Pi_Q,\Alg}^{(\psi)}$ denotes the BPS algebra sheaf on $\CM_{\Pi_Q}$ and $\ulBPS_{\CA,\Alg}^{(\psi)}$ the BPS algebra sheaf on $\CM_{\CA}$, we have a canonical isomorphism $\iota_{\CN}^!\ulBPS_{\Pi_Q,\Alg}^{(\psi)}\cong\iota_{\underline{\CF}}^!\ulBPS_{\CA,\Alg}^{(\psi)}$ in $\CD^+(\MHM(\BoN^{\underline{\CF}}))$ (i.e. of $\BoN^{\underline{\CF}}$-graded complexes of mixed Hodge structures).
\end{lemma}
\begin{proof}
 With notation as in \eqref{equation:diagramneighbourhood} we have isomorphisms
 \[
  \imath_{\CN}^!\ulBPS_{\Pi_Q,\Alg,\dd}^{\psi}\cong y^!\jmath_{\CN}^!\ulBPS_{\Pi_Q,\Alg,\dd}^{\psi}\cong y^!\CH^{0}((p_{\dd})_*\BD\ulBoQ_{\mathfrak{U}}^{\vir})\cong y^!\jmath_{\underline{\CF}}^!\ulBPS_{\CA,\Alg,a_{\dd}}^{\psi} \cong\imath_{\underline{\CF}}^!\ulBPS_{\CA,\Alg,a_{\dd}}^{\psi}
 \]
with $p_{\dd}$ as in Theorem \ref{theorem:neighbourhood}, since cohomology functors commute with open immersions.
\end{proof}

\subsection{Compatibility of generators with local neighbourhoods}
Recall the definition of the generating mixed Hodge modules $\underline{\SG}_{\CA,m}$ for $m\in\Phi^+_{\K(\CA)}$ and $\underline{\SG}_{\CA}$ in \S\ref{subsection:generators}. We consider
\[
\Free_{\boxdot\sAlg}(\underline{\SG}_{\CA})\coloneqq \Free_{\boxdot\sAlg}\left(\bigoplus_{m\in\Phi^+_{\K(\CA)}}\underline{\SG}_{\CA,m}\right).
\]
We let $\CI_{\CA,\underline{\SG}_{\CA},\Alg}$ be the Serre ideal, as defined in \S\ref{subsubsection:Serresrelationsandideal}, for the monoid $\CM_{\CA}$ and the mixed Hodge modules $\underline{\SG}_{\CA,m}$.

We let $\imath_{\underline{\CF}}\colon\BoN^{\underline{\CF}}\rightarrow \CM_{\CA}$ and $\imath_{\CN}\colon \BoN^{\underline{\CF}}\rightarrow \CM_{\Pi_{Q_{\underline{\CF}}}}$ be as in \S\ref{subsection:locneighbSigmacoll}, and define $\lambda_{\CF}$ as in \eqref{lambdaFdef}.

\begin{lemma}
\label{lemma:isoroots}
 Let $\vec{m}\in\Phi_{\Pi_Q}^{+,\iso}$. We let $\vec{m}=l\vec{m'}$ with $\vec{m'}\in\Sigma_{\Pi_Q}$. Then $\lambda_{\underline{\CF}}(\vec{m})=l\lambda_{\underline{\CF}}(\vec{m'})$ is the decomposition of $\lambda_{\underline{\CF}}(\vec{m})$ of the form $l'a$ with $a\in \Sigma_{\CA}$. Moreover, for $m\in\Phi_{\CA}^{+,\iso}$, $m=lm'$ with $l$ maximal, the map $u_{m}$ is a closed immersion whose image is the closure of the subset of $\CM_{\CA,m}$ of semisimple objects whose stabilizer is isomorphic to $\GL_l(\BoC)$.
\end{lemma}
\begin{proof}
 The first assertion, regarding the decomposition of $\lambda_{\underline{\CF}}(\mm)$, comes from Proposition~\ref{proposition:geometrygoodmodspace}. By properness of $\oplus$ (Assumption~\ref{ds_fin}), the image of the morphism $u_m$ from \S \ref{subsection:generators} is closed. The second statement is true for $\CA=\Rep\Pi_Q$ by \cite[Theorem 3.4]{crawley2002decomposition}. Then the local neighbourhood theorem (Theorem~\ref{theorem:neighbourhood}) implies that for any $\CA$, the closure of the set of points with stabiliser $\GL_l(\BoC)$ is irreducible and $2$-dimensional. The second claim follows.
\end{proof}

\begin{lemma}
 \label{lemma:restrootsgens}
 Let $\dd\in\Sigma_{\Pi_Q}$ be isotropic, so $(\dd,\dd)=0$. Then,
 \[
 \imath_{\CN}^!\underline{\SG}_{\Pi_Q,l\dd}\cong \iota_{\underline{\CF}}^!\underline{\SG}_{\CA,l\lambda_{\underline{\CF}}(\dd)}.
 \]
\end{lemma}
\begin{proof}
 This is a consequence of Lemma~\ref{lemma:isoroots}. We let $\CU_{\Delta}\subset\CU_{l\dd}$ be the closure of the locus of closed points having $\GL_l(\BoC)$ as stabilizer. In diagram \eqref{equation:diagramneighbourhood}, we have
 \[
 \jmath_{\CN}^!\underline{\SG}_{\Pi_Q,l\dd}\cong \ul{\IC}(\CU_{\Delta})\cong \jmath_{\underline{\CF}}^!\underline{\SG}_{\CA,l\lambda_{\underline{\CF}}(\dd)}
 \]
and so by applying $y^!$, we obtain the lemma.
\end{proof}

\begin{proposition}
\label{proposition:compatibilityFreealgebras}
 We have a canonical isomorphism of algebras
 \[
\beta\colon \iota_{\underline{\CF}}^!\Free_{\boxdot\sAlg}(\underline{\SG}_{\CA})\cong \iota_{\CN}^!\Free_{\boxdot\sAlg}(\underline{\SG}_{\Pi_{Q}})
 \]
in $\CD^+(\MHM(\BoN^{\underline{\CF}}))$.
\end{proposition}
\begin{proof}
 This comes from the natural isomorphism of $\BoN^{\underline{\CF}}$-graded vector spaces
\begin{equation*}
(\imath_{\CN})^!\left(\bigoplus_{\vec{m} \in \Phi_{\Pi_Q}^{+\real}\sqcup\Phi_{\Pi_Q}^{+,\hyp}} \ul{\IC}(\Msp_{\Pi_Q,\vec{m}})\right) \cong \imath_{\underline{\CF}}^!\left(\bigoplus_{a \in \Phi_{\CA}^{+,\real}\sqcup\Phi_{\CA}^{+,\hyp}} \ul{\IC}(\Msp_{\CA,a}) \right)
\end{equation*}
obtained using Proposition~\ref{proposition:rootsneighbourhood} and the diagram \eqref{equation:diagramneighbourhood}, the isomorphism of $\BoN^{\underline{\CF}}$-graded vector spaces
\begin{equation*}
(\imath_{\CN})^!\left(\bigoplus_{\vec{m} \in \Phi_{\Pi_Q}^{+\iso}} \underline{\SG}_{\Pi_Q\vec{m}}\right) \cong \imath_{\underline{\CF}}^!\left(\bigoplus_{a \in \Phi_{\CA}^{+,\iso}} \underline{\SG}_{\CA,a} \right)
\end{equation*}
obtained using Lemma~\ref{lemma:restrootsgens} and diagram \eqref{equation:diagramneighbourhood}, and the fact that $\imath^!_{\underline{\CF}}$ and $\imath^!_{\CN}$ are strictly monoidal functors (\S\ref{subsection:theCoHA}), and so they commute with the formation of the free algebras.
\end{proof}

\subsection{Compatibility of Serre relations with local neighbourhoods}
\label{subsection:compatibilityserrerelationslocal}
Recall the ideals of Serre relations for $\CA$ and $\Pi_Q$, $\CI_{\CA,\underline{\SG}_{\CA},\Alg}$ and $\CI_{\Pi_Q,\underline{\SG}_{\Pi_Q},\Alg}$, defined in \S \ref{subsection:GKMtensorcategories}. We shorten the notation: $\underline{\CI}_{\CA,\Alg}\coloneqq \CI_{\CA,\underline{\SG}_{\CA},\Alg}$ and $\underline{\CI}_{\Pi_Q,\Alg}\coloneqq \CI_{\Pi_Q,\underline{\SG}_{\Pi_Q},\Alg}$. Likewise, we shorten $\underline{\mathfrak{n}}^+_{\CA}\coloneqq\mathfrak{n}^+_{\CM_{\CA},\underline{\SG}_{\CA}}$ and $\underline{\mathfrak{n}}^+_{\Pi_Q}\coloneqq\mathfrak{n}^+_{\CM_{\Pi_Q},\underline{\SG}_{\Pi_Q}}$.
\begin{proposition}
\label{proposition:compatibilitySerreideals}
 We have a canonical isomorphism
 \[
 \iota_{\underline{\CF}}^!\underline{\CI}_{\CA,\Alg}\cong \iota_{\CN}^!\underline{\CI}_{\Pi_{Q},\Alg}.
 \]
\end{proposition}
\begin{proof}
 $\iota^!_{\underline{\CF}}\underline{\CI}_{\CA,\Alg}$ is the ideal of $\Free_{\boxdot\sAlg}(\underline{\SG}_{\CA})$ generated by Serre relations relative to $\iota^!_{\underline{\CF}}\underline{\SG}_{\CA,m}$, $m\in \Phi^+_{\CA}$ while $\iota_{\CN}^!\underline{\CI}_{\Pi_{Q_{\underline{\CF}}},\Alg}$ is the ideal of $\Free_{\boxdot\sAlg}(\underline{\SG}_{\Pi_{Q}})$ generated by Serre relations relative to $\iota_{\CN}^!\underline{\SG}_{\Pi_Q,\dd}$, $\dd\in \Phi^+_{\Pi_Q}$. By Lemma~\ref{lemma:pullbackEulerform}, Proposition~\ref{proposition:rootsneighbourhood} and Proposition~\ref{proposition:compatibilityFreealgebras}, for $\dd,\ee\in \Phi^+_{\Pi_{Q}}$, we have $\iota_{\CN}^!\underline{\SR}_{\dd,\ee}\cong\iota^!_{\underline{\CF}}\underline{\SR}_{\psi_{\underline{\CF}(\dd)},\psi_{\underline{\CF}}(\ee)}$ (with notation as in \S\ref{subsubsection:Serresrelationsandideal}). This proves the lemma. 
\end{proof}

\begin{corollary}
\label{corollary:riuea}
 We have a canonical isomorphism
\[
 \iota_{\underline{\CF}}^!\UEA(\underline{\mathfrak{n}}^+_{\CA})\cong \iota_{\CN}^!\UEA(\underline{\mathfrak{n}}^+_{\Pi_{Q}}).
\]
\end{corollary}
\begin{proof}
 This is a straightforward combination of Propositions \ref{proposition:compatibilityFreealgebras} and \ref{proposition:compatibilitySerreideals}, by taking quotients of free algebras by their respective Serre ideals.
\end{proof}

We let $\Upsilon_{\CA}\colon\Free_{\boxdot\sAlg}(\underline{\SG}_{\CA})\rightarrow \ulBPS_{\CA,\Alg}^{\psi}$ be the morphism of algebra objects in $\MHM(\CM_{\CA})$ and $\MHM(\CM_{\Pi_Q})$ extending the morphisms $\underline{\SG}_m\rightarrow\ulBPS_{\CA,\Alg}^{\psi}$ ($m\in\Phi_{\CA}^+$) given in Lemma~\ref{lemma:inclusionICBPSalg} (for real and hyperbolic simple positive roots) and Corollary~\ref{corollary:isoprimsummand} (for isotropic simple positive roots). We define the morphism $\Upsilon_{\Pi_Q}\colon\Free_{\boxdot\sAlg}(\underline{\SG}_{\Pi_{Q}})\rightarrow \ulBPS_{\Pi_Q,\Alg}^{\psi}$ in the same way (for the particular case $\CA=\Rep\Pi_Q$).
\begin{proposition}
\label{proposition:compatibilitymultfibre}
 We have a commutative diagram
 \[
 \begin{tikzcd}
	{\iota_{\CN}^!
		\Free_{\boxdot\sAlg}(\underline{\SG}_{\Pi_{Q_{\underline{\CF}}}}) 
			}
	 & {\iota_{\CN}^!\ulBPS_{\Pi_{Q_{\underline{\CF}}},\Alg}^{\psi}} \\
	{\iota_{\underline{\CF}}^!
			\Free_{\boxdot\sAlg}(\underline{\SG}_{\CA})
		} & {\iota_{\underline{\CF}}^!\ulBPS_{\CA}^{\psi}}
	\arrow[from=1-1, to=2-1, "\beta"]
	\arrow[from=1-2, to=2-2]
	\arrow[from=2-1, to=2-2, "\Upsilon_{\CA}"]
	\arrow[from=1-1, to=1-2, "\Upsilon_{\Pi_Q}"]
\end{tikzcd}
 \]
where vertical arrows are isomorphisms, given respectively by Proposition~\ref{proposition:compatibilityFreealgebras} and Lemma~\ref{lemma:restBPSalgfibre}.
\end{proposition}
\begin{proof}
 The left-most vertical arrow is an isomorphism by Proposition~\ref{proposition:compatibilityFreealgebras}. The right-most vertical arrow is an isomorphism by Lemma~\ref{lemma:restBPSalgfibre}. By compatibility with restriction to monoids of the multiplication proven in \cite[Corollary 7.3]{davison2022BPS}, the right-most vertical morphism is a morphism of algebras.  Since all four morphisms are morphisms of algebra objects, commutativity can be checked on the generating objects $\iota_{\CN}^!\underline{\SG}_{\Pi_{Q_{\underline{\CF}}}}$, and so holds by the construction of $\beta$.
\end{proof}

\subsection{Loop equivariance}
Let $\overline{L}\subset \overline{Q}_0$ be the set of loops of the quiver $\overline{Q}$. The group $\BG_a^{\overline{L}}$ acts on $\FM_{\Pi_Q}$ in the following way. If $\underline{x}\coloneqq(x_{\alpha})_{\alpha\in \overline{Q}_1}\in\mu^{-1}_{Q,\dd}(0)$ is a $\dd$-dimensional representation of $Q$ and $\underline{g}\in\BG_a^{\overline{L}}$, then 
 \[
 (\underline{g}\cdot \underline{x})_{\alpha}\coloneqq
 \left\{
 \begin{aligned}
 &x_{\alpha}+g_{\alpha}\cdotsh\Id_{d_{s(\alpha)}\times d_{s(\alpha)}}\text{ if $\alpha\in \overline{L}$}\\
 &x_{\alpha}\text{ otherwise}.
 \end{aligned}
 \right.
\]
This action commutes with the $\GL_{\dd}$-action on $\mu_{\dd}^{-1}(0)$ and so descends to $\BG_a^{\overline{L}}$-actions on $\FM_{\Pi_Q,\dd}$ and $\CM_{\Pi_Q,\dd}$, and therefore on $\FM_{\Pi_Q}$ and $\CM_{\Pi_Q}$.

\begin{lemma}
 \label{lemma:loopequivariant}
 The underlying complexes of perverse sheaves of the relative cohomological Hall algebra $\ulrelCoHA_{\Pi_Q}^{(\psi)}$, the relative BPS algebra $\ulBPS_{\Pi_Q,\Alg}^{(\psi)}$ and the free algebra $\Free_{\boxdot\sAlg}(\underline{\SG}_{\Pi_Q})$ are $\BG_{a}^{\overline{L}}$-equivariant.
\end{lemma}
\begin{proof}
 The equivariance of the relative cohomological Hall algebra $\ulrelCoHA_{\Pi_Q}^{(\psi)}$ and of the relative BPS algebra $\ulBPS_{\Pi_Q, \Alg}^{(\psi)}$ are direct consequences of the $\BG_a^{\overline{L}}$-equivariance of $\JH_{\Pi_Q}$. The equivariance of $\Free_{\boxdot\sAlg}(\underline{\SG}_{\Pi_Q})$ comes from the $\BG_a^{\overline{L}}$ stability of the smooth locus of $\CM_{\Pi_Q,\dd}$ (for $\dd\in \Sigma_{\Pi_Q}$) and the $\BG_a^{\overline{L}}$-stability of the image of $u_{\dd}\colon \CM_{\Pi_Q,\dd}\rightarrow \CM_{\Pi_Q,l\dd}$ for $\dd\in\Sigma_{\Pi_Q}$ (for the case when $\dd$ is isotropic).
\end{proof}

Let $L \subset \overline{L}$ be the set of loops of the quiver $Q$. The group $\BG_a^{L}$ is naturally a subgroup of $\BG_a^{\overline{L}}$ and hence acts on $\Mst_{\Pi_Q}$ via the above action.

\subsection{The degree zero BPS algebra of the semi-nilpotent stack}

\subsubsection{Borcherds--Bozec Lie algebra of a quiver}
\label{subsubsection:borcherdsbozec}
Let $Q=(Q_0,Q_1)$ be a quiver. We consider the monoid $M=\BoN^{Q_0}$. It is endowed with the symmetrisation of the Euler form of $Q$, $(\dd,\ee)=2\sum_{i\in Q_0}\dd_i\ee_i-\sum_{i\xrightarrow{\alpha}j\in Q_1}(\dd_i\dd_j+\dd_j\dd_i)$ (which is the Euler form of the Ginzburg dg-algebra $\SG_2(\BoC Q)$). Recall the set of positive simple roots $\Phi^+_{M,(-,-)}$ as defined in \S\ref{subsection:roots}. We have a decomposition $Q_0=Q_0^{\real}\sqcup Q_0^{\im}$, where $Q_0^{\real}$ is the set of vertices of $Q$ carrying no loops and $Q_0^{\im}$ is the set of vertices of $Q$ carrying at least one loop. We let $I_{\infty}=(Q_0^{\real}\times\{1\})\sqcup(Q_0^{\im}\times \BoZ_{\geq 1})$. We have an injection $I_{\infty}\rightarrow \Phi_{M,(-,-)}^{+}$ sending $(i',n)\in I_{\infty}$ to $n1_{i'}$. We consider the weight function
\[
 \begin{matrix}
 P&\colon&\Phi_{M,(-,-)}^+&\rightarrow&\BoN\\
 &    &m        &\mapsto  &\left\{\begin{aligned}
                     &1\text{ if $m\in I_{\infty}$}\\
                     &0 \text{ otherwise.}
                    \end{aligned}\right.

 \end{matrix}
\]
The \emph{Borcherds--Bozec Lie algebra} associated to $Q$ is $\mathfrak{g}_{Q}\coloneqq \mathfrak{g}_{M,(-,-),P}$, following \S\ref{section:GKM}. We extend the bilinear form $(-,-)$ to $\BoZ^{(I_{\infty})}$ by pullback via the natural morphism $\BoZ^{(I_{\infty})}\rightarrow \BoZ^{Q_0}$, $(i',n)\mapsto n1_{i'}$. The positive part $\mathfrak{n}_Q^+$ of $\mathfrak{g}_Q$ admits the following description by generators and relations. It is generated by $e_{(i',n)}$ for $(i',n)\in I_{\infty}$ with the relations
\[
 \begin{aligned}
 \ad(e_j)^{1-(j,i)}(e_i)&=0&\text{ for $j\in Q_0^{\real}\times \{1\}$, $i\neq j$}\\
 [e_i,e_j]&=0 &\text{ if $(i,j)=0$}.
 \end{aligned}
\]

\subsubsection{The cohomological Hall algebra of the strictly semi-nilpotent stack}
Recall from \S\ref{subsection:theCoHA} the stack of strictly semi-nilpotent representations of $\Pi_Q$, $\FM_{\Pi_Q}^{\SSN}$. The absolute cohomological Hall algebra of the category of strictly semi-nilpotent representations of $Q$ is
\[
 \HO^*\!\!\!\mathscr{A}_{\Pi_Q}^{\SSN,\psi}:=\bigoplus_{\dd\in \BoN^{Q_0}}\HO_*^{\BoMo}(\mathfrak{M}_{\Pi_Q,\dd}^{\SSN},\BoQ^{\vir}),
\]
and has been defined and studied in \cite{schiffmann2020cohomological}.

We let $\HO^0\!\SA_{\Pi_Q}^{\SSN,\psi}$ be the \emph{degree zero strictly semi-nilpotent cohomological Hall algebra}. It has a combinatorial basis given by fundamental classes of irreducible components of $\mathfrak{M}_{\Pi_Q}^{\SSN}$, $[\mathfrak{M}^{\SSN}_{\Pi_Q,\dd,c}]$. This algebra is identified with the enveloping algebra of the (positive half of the) Borcherds--Bozec Lie algebra. This is a result proven in detail in \cite[Theorem 1.3]{hennecart2022geometric}:

\begin{theorem}
\label{theorem:degree0SSN}
 There is an isomorphism of algebras
 \[
 \UEA(\mathfrak{n}_Q^+)\rightarrow \HO^0\!\mathcal{A}_{\Pi_Q}^{\SSN,\psi}
 \]
 sending the generator $e_{(i',n)}$ to $[\mathfrak{M}^{\SSN}_{n1_{i'},(n)}]$, where $\FM_{n1_{i'},(n)}^{\SSN}$ is the closed substack of $\FM_{\Pi_Q}$ of $n1_{i'}$-dimensional representations for which the arrows of $Q_1^{*}$ act by $0$.

\end{theorem}

Recall the definition of $\CM_{\Pi_Q}^{\SSN}$ in \S\ref{subsection:theCoHA}: this is the subscheme of $\CM_{\Pi_Q}$ of semisimple $\Pi_Q$-representations such that the only arrows of $\overline{Q}$ acting nontrivially are loops in $Q_1\subset \overline{Q}_1$. Let $\imath_{\SSN}\colon\mathcal{M}_{\Pi_Q}^{\SSN}\rightarrow\mathcal{M}_{\Pi_Q}$ be the natural inclusion.

\begin{lemma}[\cite{davison2020bps} \S 6.4]
 The natural map $\ptau_{\leq 0}\mathscr{A}_{\Pi_Q}^{(\psi)}\rightarrow \mathscr{A}_{\Pi_Q}^{(\psi)}$ induces a map $\imath_{\SSN}^!\ptau_{\leq 0}\mathscr{A}_{\Pi_Q}^{(\psi)}\rightarrow \imath_{\SSN}^!\mathscr{A}_{\Pi_Q}^{(\psi)}=:\mathscr{A}_{\Pi_Q}^{\SSN,(\psi)}$ which in turn induces an isomorphism
 \[
 \HO^0(\imath_{\SSN}^!\ptau_{\leq 0}\mathscr{A}_{\Pi_Q})\rightarrow \HO^0(\imath_{\SSN}^!\mathscr{A}_{\Pi_Q}).
 \]
\end{lemma}

\begin{corollary}
\label{corollary:degree0SSNvanish}
Let $Q$ be a quiver. Then the morphism 
\begin{equation}
\label{equat:cor}
 \HO^0(\imath_{\SSN}^!\Upsilon_{\Pi_Q})\colon \HO^0\!(\imath_{\SSN}^!(\Free_{\boxdot\sAlg}(\SG_{\Pi_Q})))\rightarrow \HO^0(\mathscr{A}_{\Pi_Q}^{\SSN,\psi})
\end{equation}
obtained by applying the functor $\HO^0(\imath_{\SSN}^!(-))$ to the natural morphism $\Upsilon_{\Pi_Q}\colon\Free_{\boxdot\sAlg}(\SG_{\Pi_Q})\rightarrow\SA_{\Pi_Q}^{\psi}$ (see \S\ref{subsection:compatibilityserrerelationslocal}) vanishes on $\HO^0(\imath_{\SSN}^!\CI_{\Pi_Q,\Alg})$ and induces an isomorphism
\[
 \HO^0\!(\imath_{\SSN}^!(\Free_{\boxdot\sAlg}(\SG_{\Pi_Q})))/\HO^0(\imath_{\SSN}^!\CI_{\Pi_Q,\Alg})\rightarrow \HO^0(\mathscr{A}_{\Pi_Q}^{\SSN,\psi}).
\]
\end{corollary}
\begin{proof}
 The functor $\imath_{\SSN}^!$ is strictly monoidal (\S\ref{subsection:theCoHA}), so it commutes with the operator $\Free$. By Proposition \ref{proposition:directcomplementprimitive}, the subobject $\SG_{\Pi_Q}\subset\BPS_{\Pi_Q,\Alg}^{\psi}$ admits a direct sum complement $\SF\subset \mathcal{BPS}_{\Pi_Q,\Alg}^{\psi}$ such that the multiplication map $\mult_{\dd,\ee}^{\psi}\colon \mathcal{BPS}_{\Pi_Q,\Alg,\dd}^{\psi}\boxtimes\mathcal{BPS}_{\Pi_Q,\Alg,\ee}^{\psi}\rightarrow \mathcal{BPS}_{\Pi_Q,\Alg,\dd+\ee}^{\psi}$ factors through the inclusion $\SF_{\dd+\ee}\subset\mathcal{BPS}_{\Pi_Q,\Alg,\dd+\ee}^{\psi}$ if $\dd\neq0 \neq \ee$. Consequently, the multiplication map $\HO^0(\imath_{\SSN}^!\mult_{\dd,\ee}^{\psi})\colon\HO^0(\mathscr{A}_{\Pi_Q,\dd}^{\SSN,\psi})\otimes\HO^0(\mathscr{A}_{\Pi_Q,\ee}^{\SSN,\psi})\rightarrow\HO^0(\mathscr{A}_{\Pi_Q,\dd+\ee}^{\SSN,\psi})$ factors through $\HO^0(\imath_{\SSN}^!\SF_{\dd+\ee})$. 
 
 If $\dd=n1_{i'}$ for some $i'\in Q_0^{\hyp}$, $n\geq 1$, then, $\HO^0(\imath_{\SSN}^!\mathcal{IC}(\mathcal{M}_{\Pi_Q,n1_{i'}}))=\BoQ e_{i',n}$ is one-dimensional \cite[\S 6.4.4]{davison2020bps}. If $i'\in Q_0^{\real}$, $\CM_{\Pi_Q,1_{i'}}=\pt$ and therefore, $\HO^0(\imath_{\SSN}^!\mathcal{IC}(\mathcal{M}_{\Pi_Q,1_{i'}}))=\BoQ e_{i',1}$ is one-dimensional. If $\dd=ne_{i'}$ with $i'\in Q_0^{\iso}$, $\HO^0(\imath_{\SSN}^!\SG_{n1_{i'}}))=\BoQ e_{i',n}$ is one-dimensional by the description of the BPS Lie algebra sheaf for the Jordan quiver \cite[\S 7.2]{davison2020bps} and given that in this case, $\SG_{n1_{i'}}=\BPS_{\Pi_Q,\Lie,n1_{i'}}$ (see also \S\ref{subsubsection:isotropicprimitivesummands} which applies to $\Pi_Q$ for $Q$ the Jordan quiver). By Theorem~\ref{theorem:degree0SSN}, if $\dd\neq n1_{i'}$ for $(i',n)\in I_{\infty}$, then $\HO^0(\imath_{\SSN}^!\SF_{\dd})=\HO^0(\mathscr{A}_{\Pi_Q,\dd}^{\SSN,\psi})$. If moreover $\dd\in\Phi^+_{\Pi_Q}$, $\HO^0(\imath_{\SSN}^!\SG_{\dd})$ is a direct summand of $\HO^0(\mathscr{A}_{\Pi_Q,\dd}^{\SSN})$ with complement $\HO^0(\imath_{\SSN}^!\SF_{\dd})$. It therefore vanishes. The domain of \eqref{equat:cor} is then equal to $\HO^0\!\left(\Free\left(\bigoplus_{i\in I_{\infty}}\BoQ e_i\right)\right)$. By strict monoidality of $\imath_{\SSN}^!$, $\imath_{\SSN}^!\CI_{\Pi_Q,\Alg}$ is generated by the Serre relations for $e_{(i',n)}$, $(i',n)\in I_{\infty}$, given in \S\ref{subsubsection:borcherdsbozec}. By Theorem~\ref{theorem:degree0SSN}, $\HO^0(\imath_{\SSN}^!\Upsilon_{\Pi_Q})$ vanishes on $\HO^0(\imath_{\SSN}^!\CI_{\Pi_Q,\Alg})$.
\end{proof}

\subsection{The BPS Lie algebra of a 2-Calabi--Yau category}
\label{subsection:mainproof}
The proof of Theorem~\ref{theorem:BPSalgisenvBorcherds} is given in this section. 

\subsubsection{Factorization through the Serre ideal}
\label{subsubsection:factorizationSerreideal}
We adopt ideas from the proof of \cite[Theorem~10.2]{davison2022BPS}. In comparison to the case of totally negative $2$-Calabi--Yau categories, there is an extra step which is to show that the morphism $\Free_{\boxdot\sAlg}(\underline{\SG})\rightarrow\ul{\BPS}_{\CA,\Alg}^{\psi}$ factors through the quotient by the Serre ideal $\Free_{\boxdot\sAlg}(\SG)\rightarrow \Free_{\boxdot\sAlg}(\underline{\SG})/\underline{\CI}_{\CA,\Alg}$.

\begin{lemma}
\label{lemma:supportnotkernel}
 Let $(Q,\dd)$ be such that $\Upsilon_{\Pi_Q,\dd}\colon \Free_{\boxdot\sAlg}(\underline{\SG}_{\Pi_Q})_{\dd}\rightarrow\ulBPS_{\Pi_Q,\Alg,\dd}^{\psi}$ does not vanish on $\underline{\CI}_{\Pi_Q,\Alg,\dd}$. Then, if $\underline{\CT}$ is a mixed Hodge module which is a simple summand of $\underline{\CI}_{\Pi_Q,\Alg,\dd}$ not contained in the kernel of $\Upsilon_{\Pi_Q,\dd}$, then
\[\supp\ul{\CT}\subset \overline{\{x\in\CM_{\Pi_Q,\dd}\mid \Upsilon_{\Pi_{Q_x},\dd_x}\text{ does not vanish on $\underline{\CI}_{\Pi_{Q_x},\Alg,\dd_x}$}\}},
\]
where the bar indicates the Zariski closure.
\end{lemma}
\begin{proof}
 Since the source and the target of $\Upsilon_{\Pi_Q,\dd}$ are semisimple mixed Hodge modules, we may prove the lemma in terms of perverse sheaves after applying the functor $\rat$. We let $T$ be an open subset of $\supp(\CT)$ such that $\CT_{|T}$ is given by a local system shifted by $\dim \supp\CT$. We let $x\in T$. Then, by Proposition~\ref{proposition:compatibilitymultfibre} and Proposition~\ref{proposition:compatibilitySerreideals}, $\iota_{\CN}^!\Upsilon_{\Pi_{Q_x},\dd_x}$ does not vanish on $\iota_{\CN}^!\CI_{\Pi_{Q_x},\Alg,\dd_x}$ and therefore, $\Upsilon_{\Pi_{Q_x},\dd_x}$ does not vanish on $\CI_{\Pi_{Q_x},\Alg,\dd_x}$. This completes the proof of the lemma.
\end{proof}

\begin{proposition}
\label{proposition:factorisationthroughquotient}
 The morphism $\Upsilon_{\CA}\colon\Free_{\boxdot\sAlg}(\underline{\SG}_{\CA})\rightarrow\ul{\BPS}_{\CA,\Alg}^{\psi}$ factors through $\Free_{\boxdot\sAlg}(\underline{\SG}_{\CA})\rightarrow \Free_{\boxdot\sAlg}(\underline{\SG}_{\CA})/\underline{\CI}_{\CA,\Alg}$.
\end{proposition}
\begin{proof}
\emph{Step 1: Reduction to preprojective algebras of quivers}
 
If there exists a simple mixed Hodge module $\ul{\SF}\subset \CI_{\CA,\Alg}$ such that the restriction of $\Upsilon_{\CA}$ to $\ul{\SF}$ is non-zero, then we let $x$ be a point such that $\imath_x^!\ul{\SF}\neq 0$ where $\imath_x\colon \{x\}\rightarrow \CM_{\CA}$ is the inclusion (existence of such a point under these assumptions is standard, see e.g. \cite[Corollary 10.7]{davison2022BPS}). The point $x$ corresponds to a semisimple object $\CF=\bigoplus_{j=1}^r\CF^{\oplus m_j}$. We let $\underline{\CF}=\{\CF_1,\hdots\CF_r\}$ be the corresponding collection of simple objects of $\CA$. We let $\iota_{\underline{\CF}}$ be as in diagram \eqref{equation:diagramneighbourhood}. Then, $\iota_{\underline{\CF}}^!\Upsilon_{\CA}$ does not vanish on $\iota_{\underline{\CF}}^!\ul{\SF}$. But, by Proposition~\ref{proposition:compatibilitymultfibre}, if $Q_{\underline{\CF}}$ is a quiver such that its double $\overline{Q_{\underline{\CF}}}$ is the Ext-quiver of $\underline{\CF}=\{\CF_1,\hdots,\CF_r\}$, $\iota_{\CN}^!\Upsilon_{\Pi_{Q_{\underline{\CF}}}}$ does not vanish on $\iota_{\CN}^!\ul{\CI}_{\Pi_{Q_{\underline{\CF}}},\Alg}$. This implies that $\Upsilon_{\Pi_{Q_{\underline{\CF}}}}$ does not vanish on $\ul{\CI}_{\Pi_{Q_{\underline{\CF}}},\Alg}$. So if we can prove Proposition~\ref{proposition:factorisationthroughquotient} for $\CA=\Rep\Pi_Q$ for any quiver $Q$, it is true for any Abelian category $\CA$ satisfying the assumptions of \S\ref{section:modulistack2dcats}.
 
\emph{Step 2: Proof for preprojective algebras of quivers}
We turn to the proof of Proposition~\ref{proposition:factorisationthroughquotient} for $\CA=\Rep(\Pi_Q)$.  We prove by induction the following claim.
\smallbreak
\emph{Claim 1:} For every quiver $Q$ and dimension vector $\vec{d} \in \BoN^{Q_0}$ supported on the whole of $Q$, the morphism
\begin{equation*}
\Upsilon_{\Pi_Q,\vec{d}} \colon \bigg(\Free_{\boxdot\sAlg}\bigg(\bigoplus_{\vec{f} \in \Phi^+_{\Pi_Q}} \underline{\SG}_{\Pi_Q,\vec{f}} \bigg)\bigg)_{\vec{d}} \longto \CH^{0}(\JH_{\vec{d},\ast} \BD\ul{\BoQ}_{\Mst_{\Pi_Q,\vec{d}}}^{\vir})^{\psi}=\ulBPS_{\Pi_Q,\Alg,\dd}^{\psi}
\end{equation*}
factors through the quotient $\Free_{\boxdot\sAlg}(\underline{\SG}_{\Pi_Q})_{\vec{d}}\rightarrow \Free_{\boxdot\sAlg}(\ul{\SG}_{\Pi_Q})_{\vec{d}}/\ul{\CI}_{\Pi_Q,\Alg,\dd}$.
\smallbreak
We let $\QuivDim$ be the set of pairs $(Q,\vec{d})$ of a quiver $Q=(Q_0,Q_1)$ and a dimension vector $\vec{d}\in\BoN^{Q_0}$ supported on the whole of $Q$. We define
 \begin{equation*}
\begin{matrix}
\mu\colon &\QuivDim &\rightarrow& \BoZ_{>0} \times \BoZ_{<0}& \\
&(Q,\vec{d}) &\mapsto& (\abs{\vec{d}},-\abs{Q_0})& = \left(\sum_{i\in Q_0}d_i,-\abs{Q_0}\right).
\end{matrix}
\end{equation*}
We prove the claim by induction on $\QuivDim$ with respect to the partial order given by the pull-back by $\mu$ of the lexicographical order on $\BoZ_{>0} \times \BoZ_{<0}$. More precisely, we have
\[
 (Q,\dd)\leq (Q',\dd') \iff
 \left\{
 \begin{aligned}
 &\abs{\vec{d}}\leq\abs{\vec{d'}} \\
 &\text{or } \abs{\vec{d}}=\abs{\vec{d'}} \text{ and } -\abs{Q_0}\leq -\abs{Q_0'}.
 \end{aligned}
 \right.
\]

We do this by proving the following claim.
\smallbreak
\emph{Claim 2:} If $\Upsilon_{\Pi_{Q'_x},\vec{m}_x}$ vanishes on $\ul{\CI}_{\Pi_{Q'_x},\Alg,\dd_x}$ for all half Ext quivers with dimension vectors $(Q'_x,\vec{m}_x) \in \QuivDim$ for closed points $x \in \Msp_{\Pi_Q,\vec{d}}$ such that $\mu(Q'_x,\vec{m}_x) < \mu(Q,\vec{d})$, then $\Upsilon_{\Pi_{Q},\vec{d}}$ vanishes on $\ul{\CI}_{\Pi_Q,\Alg,\dd}$.
\smallbreak
Indeed, suppose $\Upsilon_{\Pi_Q,\vec{d}}$ does not vanish on $\ul{\CI}_{\Pi_Q,\Alg,\dd}$. Let $\ul{\CT}\subset \ul{\CI}_{\Pi_Q,\Alg,\dd}$ be a simple direct summand such that the restriction of $\Upsilon_{\Pi_Q,\dd}$ to $\ul{\CT}$ is nonzero.
 
We have the following chain of inclusions
\begin{align*}
\supp(\ul{\CT}) &\subset \overline{ \{ x \in \Msp_{\Pi_{Q},\vec{d}} \mid \Upsilon_{\Pi_x,\vec{m}_x} \text{ does not vanish on $\ul{\CI}_{\Pi_{Q_x},\Alg,\dd_x}$} \} } 
&\qquad \text{(by Lemma~\ref{lemma:supportnotkernel})}\\
&\subset \overline{\{ x \in \Msp_{\Pi_Q,\vec{d}} \mid \mu(Q_x,\vec{m}_x) = \mu(Q,\dd) \}} 
 &\qquad \text{(by hypothesis and \cite[Lemma 10.11(ii)]{davison2022BPS}})\\
&=\BG_a^{\overline{L}} \cdot 0_{\vec{d}} &\qquad \text{(by \cite[Lemma 10.11(iii)]{davison2022BPS}).}
\end{align*}
In words, the hypothesis guarantees that the support of $\ul{\CT}$ is contained in the $\BG_a^{\bar{L}}$-orbit $\bar{Z} = \BG_a^{\bar{L}}\cdot 0_{\vec{d}} \cong \BoC^{\bar{L}}$ of $0_\vec{d}$.
Since $\ul{\CT}$ is a $\BG_a^{\bar{L}}$-equivariant mixed Hodge module (Lemma~\ref{lemma:loopequivariant}), its support is $\BG_a^{\bar{L}}$-stable and hence equal to the entire orbit $\bar{Z}$. 

The groups $\BG_a^{\bar{L}}$ and $\BG_a^{L}$ are contractible, and so are any of their orbits. Therefore taking total cohomology induces equivalences 
$\Perv_{\BG_a^{\bar{L}}}(\bar{Z}) \simeq \Vect(\BoQ) \simeq \Perv_{\BG_a^{L}}(Z)$ with the category of $\BoQ$-vector spaces, where $Z$ is the $\BG_a^{L}$-orbit of $0_{\vec{d}}$. Let $\imath_{\SSN}\colon \Msp_{\Pi_{Q},\vec{d}}^{\SSN} \into \Msp_{\Pi_Q,\vec{d}}$ be the inclusion.
We have a cartesian square
\begin{equation*}
\begin{tikzcd}
\BoC^{L} \cong Z \ar[d]\ar[r] & \Msp_{\Pi_Q,\vec{d}}^{\SSN} \ar[d] \\
\BoC^{\bar{L}}\cong \bar{Z} \ar[r] & \Msp_{\Pi_Q,\vec{d}}.
\end{tikzcd}
\end{equation*}
Since $\ul{\CT}$ is simple, we have $\CT\coloneqq\rat(\ul{\CT}) = \BoQ_{\BG_a^{\overline{L}}}[2\abs{L}]$ and $\imath_{\SSN,\vec{d}}^{!}\rat(\CT) = \BoQ_{\BG_a^{L}}$. 

Hence $\HO^0({\imath_{\SSN,\vec{d}}^!\CT})\cong \BoQ$ is a summand of $\HO^{0}(\imath_{\SSN}^{!}\CI_{\Pi_Q,\dd})$, which contradicts Corollary~\ref{corollary:degree0SSNvanish}. This finishes the proof.
\end{proof}

We let \[
\overline{\Upsilon}_{\CA}\colon \Free_{\boxdot\sAlg}(\ul{\SG}_{\CA})/\ul{\CI}_{\CA,\Alg}\rightarrow \ul{\BPS}_{\CA,\Alg}
\]
be the morphism obtained using Proposition~\ref{proposition:factorisationthroughquotient}.

We now have the generalisation to arbitrary quivers of \cite[Corollary 8.9]{davison2022BPS}.

\begin{corollary}
\label{corollary:degree0SSN}
Let $Q$ be a quiver. The morphism 
\[
 \HO^0(\imath_{\SSN}^!\overline{\Upsilon}_{\Pi_Q})\colon \HO^0\!(\imath_{\SSN}^!(\Free_{\boxdot\sAlg}(\underline{\SG}_{\Pi_Q})/\underline{\CI}_{\Pi_Q,\Alg}))\rightarrow \HO^0(\ulrelCoHA_{\Pi_Q}^{\SSN,\psi})
\]
is an isomorphism.
\end{corollary}
\begin{proof}
We may prove the corollary after applying the functor $\rat$. By strict monoidality of $\imath_{\SSN}^!$ and $\HO^0$ applied to complexes with vanishing derived global sections in negative degrees, we have 
$$\HO^0\!(\imath_{\SSN}^!(\Free_{\boxdot\sAlg}(\SG_{\Pi_Q})/\CI_{\Pi_Q,\Alg}))\cong \HO^0\!(\imath_{\SSN}^!(\Free_{\boxdot\sAlg}(\SG_{\Pi_Q})))/\HO^0(\imath_{\SSN}^!(\CI_{\Pi_Q,\Alg})).$$
By Corollary~\ref{corollary:degree0SSNvanish}, this can again be identified with $\Free(\bigoplus_{i\in I_{\infty}}\BoQ e_i)/\HO^0(\imath_{\SSN}^!\CI_{\Pi_Q,\Alg})$. By the proof of Corollary~\ref{corollary:degree0SSNvanish}, $\HO^0(\imath_{\SSN}^!\CI_{\Pi_Q,\Alg})$ is the ideal of Serre relations of $e_i$ as defined in \S\ref{subsubsection:borcherdsbozec}. The morphism $\HO^0(\imath_{\SSN}^!\overline{\Upsilon}_{\Pi_Q})$ can be identified with the morphism of Theorem~\ref{theorem:degree0SSN}: it is an isomorphism.
\end{proof}

\subsubsection{Proof of Theorem~\ref{theorem:BPSalgisenvBorcherds}}
\label{subsubsection:BPSalgebra2CYcats}
Theorem~\ref{theorem:BPSalgisenvBorcherds} can now be proven using the same inductive argument as the proof of Proposition~\ref{proposition:factorisationthroughquotient} applied to the morphism $\overline{\Upsilon}_{\CA}\colon \Free_{\boxdot\sAlg}(\underline{\SG}_{\CA})/\underline{\CI}_{\CA,\Alg}\rightarrow\ul{\BPS}_{\CA,\Alg}^{\psi}$. As for the proof of Proposition~\ref{proposition:factorisationthroughquotient}, we first reduce to the case of quivers and then induct on the set $\QuivDim$ defined in \S\ref{subsubsection:factorizationSerreideal}. We first reduce to showing the following claim.
\smallbreak
\emph{Claim 1:}
 If $\overline{\Upsilon}_{\Pi_Q}$ is an isomorphism for any quiver $Q$, then $\overline{\Upsilon}_{\CA}$ is an isomorphism for any $2$-Calabi--Yau category as in \S\ref{section:modulistack2dcats}.
\smallbreak
 To prove this, we take a nonzero direct summand $\ul{\CT}\subset \ul{\CK}\oplus \ul{\CC}$ where $\ul{\CK}=\ker\overline{\Upsilon}_{\CA}$ and $\ul{\CC}=\coker\overline{\Upsilon}_{\CA}$. We take $x\in\supp\ul{\CT}$ such that for $\imath_x\colon\{x\}\rightarrow\CM_{\CA}$, $\imath_x^!\ul{\CT}\neq 0$. Then $\Upsilon_{\Pi_{Q_x}}$ is not an isomorphism, where $Q_x$ is defined in \S\ref{subsubsection:extquiver} as a half of the Ext-quiver of the collection $\underline{\CF}$ of simple summands of the semisimple object of $\CA$ corresponding to $x$. In more detail, we have the following commutative diagram obtained using Proposition~\ref{proposition:compatibilitymultfibre} and Proposition~\ref{proposition:compatibilitySerreideals} and the strict monoidality of $\imath_{\CN}^!$, $\imath_{\underline{\CF}}^!$:
 \[
  \begin{tikzcd}
	{\iota_{\CN}^!(\Free_{\boxdot\sAlg}(\ul{\SG}_{\Pi_{Q_{\underline{\CF}}}})/\ul{\CI}_{\Pi_Q,\Alg})} & {\iota_{\CN}^!\ulBPS_{\Pi_{Q_{\underline{\CF}}},\Alg}^{\psi}} \\
	{\iota_{\underline{\CF}}^!(\Free_{\boxdot\sAlg}(\ul{\SG}_{\CA})/\ul{\CI}_{\CA,\Alg})} & {\iota_{\underline{\CF}}^!\ulBPS_{\CA}^{\psi}}
	\arrow[from=1-1, to=2-1]
	\arrow[from=1-2, to=2-2]
	\arrow["\imath_{\underline{\CF}}^!\overline{\Upsilon}_{\CA}",from=2-1, to=2-2]
	\arrow["\imath_{\CN}^!\overline{\Upsilon}_{\Pi_{Q_x}}",from=1-1, to=1-2]
\end{tikzcd}
 \]
By Propositions \ref{proposition:compatibilitySerreideals} and \ref{proposition:compatibilitymultfibre}, both vertical arrows are isomorphisms. It follows that the bottom horizontal arrow $\imath_{\underline{\CF}}^!\overline{\Upsilon}_{\CA}$ is an isomorphism if and only if the top horizontal arrow $\imath_{\CN}^!\overline{\Upsilon}_{\Pi_{Q_x}}$ is an isomorphism. Since $\imath_{\underline{\CF}}^!\ul{\CT}$ is in the kernel or cokernel of $\imath_{\underline{\CF}}^!\overline{\Upsilon}_{\CA}$, $\imath_{\CN}^!\overline{\Upsilon}_{\Pi_{Q_x}}$ is not an isomorphism, and therefore, $\overline{\Upsilon}_{\Pi_{Q_x}}$ itself is not an isomorphism. This proves Claim 1.

Then, the fact that $\overline{\Upsilon}_{\Pi_Q}$ is an isomorphism for any quiver $Q$ is proved by induction on $\QuivDim$ as in Step $2$ of the proof of Proposition~\ref{proposition:factorisationthroughquotient}, by appealing to Corollary~\ref{corollary:degree0SSN}. More precisely, we have the following claim whose proof is identical to that of Lemma~\ref{lemma:supportnotkernel}.
 \smallbreak
 \emph{Claim 2:} Let $(Q,\dd)\in\QuivDim$ be such that 
 \[
 \overline{\Upsilon}_{\Pi_Q,\dd}\colon \left(\Free_{\boxdot\sAlg}(\ul{\SG}_{\Pi_Q})/\ul{\CI}_{\Pi_Q,\Alg}\right)_{\dd}\rightarrow \ul{\BPS}_{\Pi_Q,\Alg,\dd}^{\psi}
 \]
  is not an isomorphism. Then, if $\ul{\CT}$ is a mixed Hodge module which is a simple summand of $\ker\overline{\Upsilon}_{\Pi_Q,\dd}\oplus\coker\overline{\Upsilon}_{\Pi_Q,\dd}$, then $\supp\ul{\CT}\subset\overline{\{x\in\CM_{\Pi_Q,\dd}\mid \overline{\Upsilon}_{\Pi_{Q_x},\dd_x} \text{is not an isomorphism}\}}$. 
  \smallbreak
  Then, we have the claim whose proof is the same as the one for Claim~2 of the Step~2 of the proof of Proposition~\ref{proposition:factorisationthroughquotient}:
 \smallbreak
 \emph{Claim 3:} If $\overline{\Upsilon}_{\Pi_{Q'_x},\vec{m}_x}$ is an isomorphism for all half Ext-quivers with dimension vector $(Q'_x,\vec{m}_x)\in\QuivDim$ for closed points $x\in\CM_{\Pi_Q,\dd}$ such that $\mu(Q_x,\vec{m}_x)<\mu(Q,\dd)$, then $\overline{\Upsilon}_{\Pi_Q,\dd}$ is an isomorphism.
 
 This claim implies by induction that for any $\dd\in\BoN^{Q_0}$, $\overline{\Upsilon}_{\Pi_Q,\dd}$ is an isomorphism, finishing the proof of Theorem~\ref{theorem:BPSalgisenvBorcherds}; for the base case, we need to show that $\overline{\Upsilon}_{\Pi_Q,\dd}$ is an isomorphism when $Q$ has a single vertex, and $\dd=1$, which is clear.
\qed

\subsubsection{The BPS Lie algebra}
\label{subsubsection:BPSLiealgebra}
Let $\CA$ be a 2-Calabi--Yau category satisfying our Assumptions \ref{p_assumption}--\ref{BPS_cat_assumption} in \S\ref{subsection:assumptionsCoHA} and \S\ref{subsec:spaceassumptions}as well as Assumption \ref{ass:associativity}.  In \S\ref{subsubsection:BPSalgebra2CYcats}, we proved Theorem~\ref{theorem:BPSalgisenvBorcherds}, that is that we have an isomorphism in $\MHM(\CM_{\CA})$
\[
 \ul{\BPS}_{\CA,\Alg}^{\psi}\cong\UEA(\mathfrak{n}^+_{\CM_{\CA},\ul{\SG}_{\CA}})
\]
where $\ul{\SG}_{\CA}$ is as in \S\ref{subsection:generators}, motivating the following definition:
\begin{definition}
\label{def:BPSheafalg}
The BPS \emph{Lie} algebra sheaf of $\CA$ is 
\[
\ul{\BPS}_{\CA,\Lie}\coloneqq\mathfrak{n}^+_{\CM_{\CA},\ul{\SG}_{\CA}}.
\]
In words, this is the generalised half Kac--Moody Lie algebra object in the tensor category of mixed Hodge modules on the coarse moduli space (as defined in \S \ref{subsection:GKMtensorcategories}), with generating mixed Hodge modules defined in \S \ref{subsection:generators}.
If $\imath_{\triangle}\colon \CM_{\CA}^{\triangle}\hookrightarrow \CM_{\CA}$ is the inclusion of a saturated monoid (in schemes) corresponding to a Serre subcategory, we define
\[
\ul{\Fn}_{\CA}^{\triangle,+}\coloneqq\HO^*(\imath_{\triangle}^!\mathfrak{n}^+_{\CM_{\CA},\ul{\SG}_{\CA}})
\]
and in particular $\ul{\Fn}_{\CA}^+\coloneqq\HO^*(\mathfrak{n}_{\CM_{\CA},\ul{\SG}_{\CA}})$. We denote by $\Fg^{\triangle}_{\CA}\supset \Fn_{\CA}^{\triangle,+}$ the extension to the generalised Kac--Moody Lie algebra as defined in Theorem~\ref{corollary:absBor}.
\end{definition}

\subsection{Verdier duality and the compactly supported CoHA}
\begin{corollary}
 The BPS algebra sheaf $\ulBPS_{\CA,\Alg}$ and the BPS Lie algebra sheaf $\ulBPS_{\CA,\Lie}$ are Verdier self-dual mixed Hodge modules on $\CM_{\CA}$. More precisely, each simple direct summand of $\ulBPS_{\CA,\Alg}$ (and hence of $\ulBPS_{\CA,\Lie}$) is Verdier self-dual.
\end{corollary}
\begin{proof}
 By Theorem \ref{theorem:BPSalgisenvBorcherds}, the relative BPS algebra is a quotient of $\Free_{\boxdot\sAlg}(\ul{\SG}_{\CA})$ and so (by finiteness of $\oplus$ which gives the commutation of $\boxdot$ with Verdier duality), it suffices to prove the Verdier self-duality of each simple direct summand of $\Free_{\boxdot\sAlg}(\ul{\SG}_{\CA})$. Such a simple direct summand appears in the pushforward of the $\ul{\IC}$ mixed Hodge module by the finite map $\oplus'\colon \prod_{i=1}^r\CM_{\CA,m_i}\times\prod_{j=1}^s\CM_{\CA,n_j}\rightarrow\CM_{\CA}$, $((x_i)_{1\leq i\leq r},(y_j)_{1\leq j\leq s})\mapsto \bigoplus_{i=1}^rx_i^{l_i}\oplus\bigoplus_{j=1}^sy_j$, where $m_i\in\Sigma_{\K(\CA)}\cap\Phi^{+,\iso}_{\K(\CA)}$, $n_j\in\Sigma_{\K(\CA)}\setminus (\Sigma_{\K(\CA)}\cap\Phi^{+,\iso}_{\K(\CA)})$ and $l_i\in\BoZ_{\geq 1}$. The map $\oplus'$ is small and is a covering over an open subspace of its image with structure group a product of symmetric groups. The corollary then follows from the fact that all representations of symmetric groups are self-dual.
\end{proof}

Recall the class map $\cl\colon\CM_{\CA}\rightarrow\K(\CA)$ (\S \ref{subsubsection:twistmultiplication}). The absolute CoHA can be seen as $\cl_*\ulrelCoHA_{\CA}\in\CD^+(\MHM(\K(\CA)))$ endowed with the multiplication $\cl_*\mult$. We also consider the variant using the functor $\cl_!$. By the properness of $\oplus$ (implying $\oplus_*= \oplus_!$), $\cl_!$ gives a monoidal functor $\cl\colon \CD^+(\MHM(\CM_{\CA}))\rightarrow\CD^+(\MHM(\K(\CA)))$. We can then consider the \emph{compactly supported CoHA} of $\CA$, defined as $\cl_!\ulrelCoHA_{\CA}$. A subalgebra is given by the compactly supported BPS algebra $\cl_!\ulBPS_{\CA,\Alg}$, and we define $\cl_!\ulBPS_{\CA,\Lie}$ to be the compactly supported BPS Lie algebra. A corollary of Theorems \ref{theorem:BPSalgisenvBorcherds} and \ref{theorem:relPBW} is the following.
\begin{corollary}
 \begin{enumerate}
 \item The compactly supported BPS algebra $\cl_!\BPS_{\CA,\Alg}$ is isomorphic to $\UEA(\Fn^+_{\K(\CA),(-,-)_{\SC},\IP_{\rmc}})$ where
\[
 \begin{matrix}
  \IP_{\rmc}&\colon&\Phi^+_{\K(\CA)}&\rightarrow&\BoN[t^{\pm1/2}]\\
  &&m&\mapsto&\IP_{\rmc}(\SG_{\CA,m})&=\sum_{j\in\BoZ}\dim\HO^j(\SG_{\CA,m})t^{-j/2}.
 \end{matrix}
\]
\item We have a PBW isomorphism
\[
 \Sym\left(\Fn^+_{\K(\CA),(-,-)_{\SC},\IP_{\rmc}}\otimes\HO^*_{\BoC^*}\right)\rightarrow\cl_!\SA_{\CA}.
\]
 \end{enumerate}
\end{corollary}
\begin{proof}

Both statements are straightforward consequences of Theorem \ref{theorem:BPSalgisenvBorcherds}, and the description of the weight function uses that $\SG_{\CA,m}$ is Verdier self-dual and so $\HO^*_{\rmc}(\SG_{\CA,m})\cong \HO^*(\SG_{\CA,m})^*$ (the second star in the R.H.S. indicates the graded linear dual).

\end{proof}

\subsection{Weyl-group invariance}
\label{subsection:weylinvariance}
Let $\CA$ be a 2-Calabi--Yau category as in Definition \ref{def:BPSheafalg}, and let $\ul{\SG}_{\CA}\in \MHM(\CM_{\CA})$ be defined as above (see \S \ref{subsection:generators}). Let $\imath_{\triangle}\colon \CM_{\CA}^{\triangle}\hookrightarrow \CM_{\CA}$ be the inclusion of monoids corresponding to a Serre subcategory, and let us assume that if $\CM_{\CA,m}\subset \CM_{\CA}$ is a connected component corresponding to a real simple root $m$, then $\CM_{\CA,m}^{\triangle}=\CM_{\CA,m}$. It follows that the weight function $P\colon \Phi_{M,(-,-)}^+\rightarrow\BoN[t^{\pm 1/2}]$ defining the generalised Kac--Moody Lie algebra $\Fg^{\triangle}_{\CA}$ satisfies $P(m)=1$, as required by the definition of GKM algebras in \S\ref{subsection:GKMmonoids}.

We define the graph $\Gamma$ as follows: the vertices $V$ of $\Gamma$ are labelled by real simple roots $a\in\Phi_{\CA}^{+,\real}$ of $\CA$. If $i$ and $j$ are two vertices of $\Gamma$, then we have $-(i,j)_{\SC}$ edges between them. Let $W_{\Gamma}$ be the Weyl group associated to the graph $\Gamma$: it is the subgroup of the group of linear automorphisms of $\BoC^V$ generated by the reflections
\[
s_i(\lambda)=\lambda -\lambda^T C_{i}1_i
\]
where $C_{i}$ is the $i$th column of the Cartan matrix defined by $C_{ij}=2\delta_{ij}-\#\{\textrm{edges between }i\textrm{ and }j\}$. Let $\widehat{M}$ denote the groupification of $\K(\CA)$. The group $W_{\Gamma}$ acts similarly on $\widehat{M}\otimes_{\BoZ}\BoC$ by reflections:
\[
s_i(m)=m-(m,i)_{\SC}i.
\]

\begin{proposition}
\label{proposition:weylInv}
Let $\CA$ be a 2-Calabi--Yau category, and let $\imath_{\triangle}\colon\CM^{\triangle}_{\CA}\hookrightarrow \CM_{\CA}$ be as above. Let $m,m'\in\K(\CA)$ satisfy $w(m)=m'$, where $w\in W_{\Gamma}$ and $W_{\Gamma}$ is the Weyl group defined above. Then there is an isomorphism of weight spaces
\begin{equation}
\label{WeylKacIso}
\Fn_{\CA,m}^{\ttBPS,\triangle,+}\cong \Fn_{\CA,m'}^{\ttBPS,\triangle,+}.
\end{equation}
\end{proposition}
\begin{proof}
Let $\Fg\subset \Fg_{\CA}^{\ttBPS,\triangle}$ be the Kac--Moody Lie algebra generated by the subspace 
\[
\bigoplus_{m\in \Phi_{\K(\CA),(-,-)}^{+,\real}}\left(\Fg_{\CA,m}^{\ttBPS,\triangle}+\Fg_{\CA,-m}^{\ttBPS,\triangle}\right).
\]
In other words, we set $\Fg$ to be the Kac--Moody Lie algebra associated to the (half) Ext quiver of the set of isomorphism classes of objects in $\CA$ corresponding to real simple roots. Then $\Fg_{\CA}^{\ttBPS,\triangle}$ is an integrable $\Fg$-representation, and the claim follows by general theory of Kac--Moody Lie algebras \cite[Proposition 3.7a)]{kac1990infinite}.
\end{proof}

\section{Comparison with the 3D BPS Lie algebra and cohomological integrality for 2CY categories}
\label{section:comparison}
\subsection{BPS Lie algebras for Jacobi algebras}
\label{subsection:DMreminder}

The BPS Lie algebra associated to a quiver with general potential, or Jacobi algebra, was defined in \cite[\S 6.1, Corollary 6.11]{davison2020cohomological}. We briefly recall the definitions, and refer to \cite{kontsevich2010cohomological,davison2017critical,davison2020cohomological} for background and full details. While the theory we present here can readily be lifted to categories of (monodromic) mixed Hodge modules/mixed Hodge structures, as we have done in the rest of the paper, we describe everything at the level of perverse sheaves/vector spaces to avoid the introduction of \emph{monodromic} mixed Hodge modules, which become necessary to deal with vanishing cycle functors and the square-root of the Tate twist.

Let $Q$ be a quiver, which we assume to be symmetric (i.e. for every pair of vertices $i,j\in Q_0$ we assume that there are as many arrows from $i$ to $j$ as there are from $j$ to $i$). We fix a bilinear form $\psi\colon \BoZ^{Q_0}\times \BoZ^{Q_0}\rightarrow \BoZ$ satisfying the condition
\[
\psi(\dd'',\dd')+\psi(\dd',\dd'')\equiv\chi_Q(\dd',\dd')\chi_Q(\dd'',\dd'')+\chi_Q(\dd',\dd'')\quad \pmod{2},
\]
as in \cite[\S2.6]{kontsevich2010cohomological}, or \cite[\S1.6]{davison2020cohomological}. We next pick a \emph{potential} $W\in\BoC Q/[\BoC Q,\BoC Q]$. The potential $W$ defines a function $\Tr(W)$ on $\FM_{Q,\dd}$, sending a $\BoC Q$-module $\rho$ to $\Tr(\overline{W}\cdot \colon\rho\rightarrow\rho)$ where $\overline{W}$ is any lift of $W$ to $\BoC Q$. This function factors through $\JH\colon \FM_{Q,\dd}\rightarrow \CM_{Q,\dd}$. We define
\[
\relCoHA_{Q,W,\dd}\coloneqq \JH_*\phip{\Tr(W)}\BoQ^{\vir}_{\FM_{Q,\dd}}
\]
where $\BoQ_{\FM_{Q,\dd}}^{\vir}\coloneqq\BoQ_{\FM_{Q,\dd}}[-\chi_Q(\dd,\dd)]$, and $\relCoHA_{Q,W}\coloneqq\bigoplus_{\dd\in \BoN^{Q_0}}\relCoHA_{Q,W,\dd}$, with $\phip{\Tr(W)}$ the perverse vanishing cycles functor for the function $\Tr(W)$. Via pushforward and pullback in vanishing cycle cohomology along morphisms from the stack of short exact sequences of $\BoC Q$-modules, along with the Thom--Sebastiani isomorphism, the complex $\relCoHA_{Q,W}$ carries the structure of an algebra object in $\CD^+(\Perv(\CM_Q))\cong\CD^+_{\rmc}(\CM_Q)$, see \cite{kontsevich2010cohomological,davison2017critical}. The monoidal structure on this category is defined via the direct image along the direct sum map $\oplus\colon \CM_Q\times\CM_Q\rightarrow \CM_Q$. As usual, we denote by $\relCoHA_{Q,W}^{\psi}$ the $\psi$-twisted CoHA, in which we twist the multiplication morphism $\relCoHA_{Q,W,\dd'}\boxdot\relCoHA_{Q,W,\dd''}\rightarrow \relCoHA_{Q,W,\dd'+\dd''}$ by $(-1)^{\psi(\dd',\dd'')}$.

For $\dd\in\BoN^{Q_0}$ we define 
\[
\BPS_{Q,W,\dd}\coloneqq (\ptau^{\leq 1}\!\JH_*\phip{\Tr(W)}\BoQ_{\FM_{Q,\dd}}^{\vir})[1].
\]
The following theorem is the basis for the definition of the BPS Lie algebra associated to a quiver with potential \cite[Theorem C]{davison2020cohomological}.
\begin{theorem}[{\cite[Theorem C, Corollary 6.11]{davison2020cohomological}}]
\label{theorem:3dBPS_def}
The complex $\BPS_{Q,W,\dd}$ is a perverse sheaf. Setting $\Fg_{Q,W,\dd}\coloneqq \HO^*\!\BPS_{Q,W,\dd}[-1]$ and $\Fg_{Q,W}\coloneqq \bigoplus_{\dd\in\BoN^{Q_0}}\Fg_{Q,W,\dd}$, the natural morphism $\Fg_{Q,W}\rightarrow \HO^*\!\relCoHA_{Q,W}$ is an injection, and the image is closed under the commutator Lie bracket for the algebra $\HO^*\!\relCoHA_{Q,W}^{\psi}$.
\end{theorem}
The resulting $\BoN^{Q_0}$-graded Lie algebra is denoted $\Fg_{Q,W}^{\psi}$. We include the superscript $\psi$ since a priori this algebra depends on the bilinear form $\psi$ (in contrast with the Lie algebra $\Fn_{\CA}^+$ we have associated with a 2CY category $\CA$). We use the notation $\mathfrak{g}_{Q,W}$ as opposed to $\Fn^{+}_{Q,W}$, following \cite{davison2020cohomological}. In the case of the tripled quiver with potential \S\ref{subsection:2d3d}, the Lie algebra $\mathfrak{g}_{\tilde{Q},\tilde{W}}^{\psi}$ appears a posteriori to be a half of a GKM Lie algebra (by Proposition \ref{proposition:BPS3dcoincide} and our main result Theorem \ref{corollary:absBor}), and so one could instead use the letter $\Fn$, but we stick with $\Fg$ for consistency with \cite{davison2020cohomological}.

\subsection{Tripled quiver with potential}
\label{subsection:2d3d}
Let $Q$ be a quiver, and denote by $\tilde{Q}$ the tripled quiver of $Q$ defined in \S\ref{subsection:QandR}. We denote by 
\begin{equation}
\label{tWdef}
\tilde{W}=\left(\sum_{i\in Q_0}\omega_i\right)\left(\sum_{a\in Q_1}[a,a^*]\right)
\end{equation}
the \emph{canonical cubic potential}. We denote by 
\[
h\colon \CM_{\tilde{Q}}\rightarrow\CM_{\overline{Q}}
\]
the forgetful map, forgetting the additional loops $\omega_i$ at the vertices. The dimensional reduction isomorphism \cite[Theorem A.1]{davison2017critical} is a natural isomorphism
\[
\Lambda_{\dd} \colon h_*\relCoHA_{\tilde{Q},\tilde{W},\dd}\cong \relCoHA_{\Pi_Q,\dd}
\]
and $\Lambda=\bigoplus_{\dd\in\BoN^{Q_0}}(-1)^{\binom{\lvert \dd\lvert}{2}}\Lambda_{\dd}$ is an isomorphism of algebra objects \cite{davison2017appendix,yang2018cohomological}. Note that 
\[
\chi_{\tilde{Q}}(\dd',\dd')\chi_{\tilde{Q}}(\dd'',\dd'')+\chi_{\tilde{Q}}(\dd',\dd'')\equiv(\dd',\dd'')_Q\quad \pmod{2}
\]
so in particular we may retain our assumed choice of $\psi$: we set $\psi(\dd',\dd'')=\chi_{Q}(\dd',\dd'')$. 

We have the following theorem from \cite{davison2020bps}
\begin{theorem}[{\cite[Theorem 4.1 and 6.1]{davison2020bps} }]
\label{theorem:lessperverse}
The complex $\BPS^{\td}_{\Pi_Q,\Lie}\coloneqq h_*\BPS_{\tilde{Q},\tilde{W}}[-1]=h_*\ptau^{\leq 1}\!\JH_*\phip{\Tr(\tilde{W})}\BoQ_{\FM_{\tilde{Q}}}^{\vir} $ is a perverse sheaf. Moreover, the morphism
\[
\BPS^{\td}_{\Pi_Q,\Lie}\rightarrow \BPS_{\Pi_Q,\Alg}^{\psi}
\]
is an inclusion of perverse sheaves, and the image is closed under the commutator Lie bracket on the target. We denote by $\BPS^{\td,\psi}_{\Pi_Q,\Lie}$ the resulting Lie algebra object. The resulting morphism
\[
\UEA(\BPS^{\td,\psi}_{\Pi_Q,\Lie})\rightarrow \BPS_{\Pi_Q,\Alg}^{\psi}
\]
is an isomorphism of algebra objects. There is a natural isomorphism of Lie algebras $\HO^*\!\BPS^{\td,\psi}_{\Pi_Q,\Lie}\cong \Fg_{\tilde{Q},\tilde{W}}^{\psi}$.
\end{theorem}
In particular, we have a morphism
\[
 \PBW_{\Pi_Q}\colon \Sym_{\boxdot}(\BPS_{\Pi_Q,\Lie}^{\td})\rightarrow \BPS_{\Pi_Q,\Alg}^{\psi}
\]
obtained by evaluating symmetric tensors with the aid of the algebra structure on the target.

\begin{proposition}[{\cite[Formula 67]{davison2020bps}}]
\label{prop:3dPBW}
 The map $\PBW_{\Pi_Q}$ is an isomorphism of perverse sheaves.
\end{proposition}

\begin{proposition}
\label{proposition:BPS3dcoincide}
There is a natural morphism
 \[
 \BPS_{\Pi_Q,\Lie}\rightarrow\BPS_{\Pi_Q,\Lie}^{\td,\psi}
 \]
 in $\Perv(\CM_{\Pi_Q})$. It is an isomorphism of Lie algebra objects. Taking derived global sections, there is an isomorphism of Lie algebras $ \Fn^{\ttBPS,+}_{\Pi_Q}\cong \Fg_{\tilde{Q},\tilde{W}}^{\psi}$.
\end{proposition}
\begin{proof}
 By Theorem~\ref{theorem:BPSalgisenvBorcherds} there is a monomorphism of Lie algebra objects
 \begin{equation}
 \label{equation:inclusionBPS}
 \BPS_{\Pi_Q,\Lie}\rightarrow \BPS_{\Pi_Q,\Alg}^{\psi}
 \end{equation}
and by \cite[Formula 67]{davison2020bps} there is a monomorphism
 \begin{equation}
 \label{equation:inclusionBPS3d}
 \BPS_{\Pi_Q,\Lie}^{\td,\psi}\rightarrow \BPS_{\Pi_Q,\Alg}^{\psi}.
 \end{equation}
In both instances we consider the target as a Lie algebra object via the commutator bracket. By Theorem~\ref{theorem:BPSalgisenvBorcherds} again, $\BPS_{\Pi_Q,\Lie}$ can be written as a quotient of $\Free_{\boxdot\sLie}(\SG_{\Pi_Q})$, and in particular for each $\dd\in\Phi^+_{\Pi_Q}$ the semisimple perverse sheaf $\BPS_{\Pi_Q,\Lie}$ contains a unique copy of $\SG_{\dd}$. Since $\BPS_{\Pi_Q,\Lie}^{\td}$ contains $\SG_{\dd}$ for each $\dd\in\Phi^+_{\Pi_Q}$ \cite[\S 7]{davison2020bps}, the morphism \eqref{equation:inclusionBPS} factors uniquely through \eqref{equation:inclusionBPS3d}. We therefore have monomorphisms of Lie algebra objects.
\begin{equation}
\label{equation:monoBPSLieAlgs}
 \BPS_{\Pi_Q,\Lie}\rightarrow \BPS_{\Pi_Q,\Lie}^{\td,\psi}\rightarrow\BPS_{\Pi_Q,\Alg}^{\psi}.
\end{equation}
By evaluating symmetric products using the algebra structure on $\BPS_{\Pi_Q,\Alg}^{\psi}$, we obtain from \eqref{equation:monoBPSLieAlgs} the morphisms of perverse sheaves
\begin{equation}
\label{equation:Sym_mono}
 \Sym_{\boxdot}(\BPS_{\Pi_Q,\Lie})\rightarrow\Sym_{\boxdot}(\BPS_{\Pi_Q,\Lie}^{\td})\rightarrow \BPS_{\Pi_Q,\Alg}^{\psi}.
\end{equation}

The composition $\Sym_{\boxdot}(\BPS_{\Pi_Q,\Lie})\rightarrow \BPS_{\Pi_Q,\Alg}^{\psi}$ is an isomorphism by Theorem~\ref{theorem:BPSalgisenvBorcherds} and the second arrow in \eqref{equation:Sym_mono} is the morphism $ \PBW_{\Pi_Q}$, which is an isomorphism by Proposition~\ref{prop:3dPBW}. Therefore the first arrow in \eqref{equation:Sym_mono} is an isomorphism, and consequently the first arrow in \eqref{equation:monoBPSLieAlgs} is also an isomorphism.
\end{proof}

Let $Q$ be a quiver, and let $\Gamma$ be the underlying graph of $Q^{\real}$. The Weyl group $W_{\Gamma}$ acts on $\BoC^{Q_0}$, and its image is generated by the reflections
\[
s_i\colon \dd\mapsto \dd-(\dd,1_i)_{Q}1_i
\]
with $i$ a vertex that supports no 1-cycles, i.e. $i\in Q^{\real}_0$ (see also \S \ref{subsection:weylinvariance}).
\begin{corollary}
\label{corollary:kacWeyl}
Let $Q$ be a quiver, and let $\dd,\dd'\in\BoN^{Q_0}$ be dimension vectors such that $w\cdot \dd=\dd'$ for some $w\in W_{\Gamma}$. Then there is an isomorphism
\[
\Fg_{\tilde{Q},\tilde{W},\dd}\cong \Fg_{\tilde{Q},\tilde{W},\dd'}.
\]
\end{corollary}
\begin{proof}
This follows from combining Propositions \ref{proposition:weylInv} and \ref{proposition:BPS3dcoincide}.
\end{proof}

Using the characterisation \eqref{equation:Ngraded} of Kac polynomials in terms of characteristic functions of BPS cohomology, we obtain a new proof of the Weyl-invariance of Kac polynomials
\[
A_{Q,\dd}(q)=A_{Q,w\cdot \dd}(q)
\]
which is part of Theorem 1 of Kac's original paper \cite{kac1983root}.

\subsection{Cohomological integrality and PBW isomorphism for $2$-Calabi--Yau categories}

\subsubsection{Restriction of BPS Lie algebras}

We let $\underline{\CF}=\{ \CF_1,\ldots,\CF_r \}$ be a $\Sigma$-collection in an Abelian category $\CA$, satisfying the assumptions of Theorem~\ref{theorem:neighbourhood}, where $\CA$ also satisfies Assumptions \ref{p_assumption}--\ref{BPS_cat_assumption} from \S \ref{subsection:assumptionsCoHA} and \S \ref{subsec:spaceassumptions}, as well as Assumption \ref{ass:associativity}, so that we can define the BPS Lie algebra for $\CA$. We let $Q$ be a half Ext-quiver of the collection $\underline{\CF}$.
\begin{lemma}
 Let $\ulBPS_{\Pi_Q,\Lie}$ be the BPS Lie algebra sheaf on $\CM_{\Pi_Q}$ and let $\ulBPS_{\CA,\Lie}$ be the BPS Lie algebra sheaf on $\CM_{\CA}$. Then with the notations of diagram \eqref{equation:diagramneighbourhood}, we have a canonical isomorphism of Lie algebras $\imath_{\CN}^!\ulBPS_{\Pi_Q,\Lie}\cong\imath_{\underline{\CF}}^!\ulBPS_{\CA,\Lie}$.
\end{lemma}
\begin{proof}
 Once again, we may prove this lemma after applying the functor $\rat$. Then it follows from Lemma~\ref{lemma:restBPSalgfibre}, Proposition~\ref{proposition:rootsneighbourhood} and Lemma~\ref{lemma:restrootsgens}. More precisely, $\imath_{\CN}^!\BPS_{\Pi_Q,\Lie}^{\psi}$ is the sub-Lie algebra of $\iota_{\CN}^!\BPS_{\Pi_Q,\Alg}^{\psi}$ generated by $\iota_{\CN}^!\SG_{\Pi_Q,\dd}$ for $\dd\in\Phi^+_{\Pi_Q}$ and $\imath_{\underline{\CF}}^!\BPS_{\CA,\Lie}^{\psi}$ is the sub-Lie algebra of $\imath_{\underline{\CF}}^!\BPS_{\CA,\Alg}^{\psi}$ generated by $\iota^!_{\underline{\CF}}\SG_{\CA,m}$, $m\in\Phi^+_{\CA}$, and for $\dd\in\Phi^+_{\Pi_Q}$, $\iota_{\CN,\dd}^!\SG_{\Pi_Q,\dd}\cong\iota_{\underline{\CF},\lambda_{\underline{\CF}}(\dd)}^!\SG_{\CA,\lambda_{\underline{\CF}}(\dd)}$ as in the proof of Proposition \ref{proposition:compatibilityFreealgebras} (where $\lambda_{\underline{\CF}}$ is defined in \S\ref{subsubsection:extquiver}), which moreover satisfy the same Serre relations (using Lemma \ref{lemma:pullbackEulerform}).
\end{proof}

\subsubsection{Proof of the cohomological integrality theorem for 2CY categories}
\label{subsubsection:fullPBW}
We let $\CA$ be a 2CY Abelian category satisfying Assumptions~\ref{p_assumption}-\ref{BPS_cat_assumption}. We also assume that the moduli stack of objects $\FM_{\CA}$ carries a positive determinant line bundle (Assumption~\ref{determinant_assumption}).

We prove Theorem~\ref{theorem:relPBW}. We first construct the morphism $\Psi_{\CA}$.

Composing the adjunction morphism $\ul{\BPS}_{\CA,\Alg}=\tau^{\leq 0}\ulrelCoHA_{\CA}\rightarrow \ulrelCoHA_{\CA}$ with the inclusion $\ul{\BPS}_{\CA,\Lie}\rightarrow \ul{\BPS}_{\CA,\Alg}^{\psi}$ provided by Theorem \ref{theorem:BPSalgisenvBorcherds}, we have a canonical morphism
\[
 \ulBPS_{\CA,\Lie}\rightarrow \ulrelCoHA_{\CA}.
\]
We fix a positive determinant line bundle on $\FM_{\CA}$. This means that we fix a morphism $\det\colon \FM_{\CA}\rightarrow \B \BoC^*$ such that for every point $x\in \FM_{\CA}$ representing an object $\CF$, if we let $\imath\colon \B \BoC^*\rightarrow \FM_{\CA}$ be the morphism corresponding to the inclusion of groups $\BoC^*\subset \Aut_{\CA}(\CF)$, and write $\HO_{\BoC^*}=\BoQ[u]$, then $\imath^*\det^*u=m_{\CF} u$ for some $m_{\CF}\in\BoZ_{\geq 1}$. Fixing a determinant line bundle on $\FM_{\CA}$, the algebra $\HO^*_{\BoC^*}$ acts on $\ulrelCoHA_{\CA}$, and we obtain a morphism
\[
 \ulBPS_{\CA,\Lie}\otimes\HO^*_{\BoC^*}\rightarrow \ulrelCoHA_{\CA}^{\psi}.
\]
The algebra structure on $\ulrelCoHA_{\CA}^{\psi}$ now induces by iterated multiplications a morphism
\[
 \Psi^{\psi}_{\CA}\colon\Sym_{\boxdot}(\ulBPS_{\CA,\Lie}\otimes\HO^*_{\BoC^*})\rightarrow \ulrelCoHA_{\CA}^{\psi}
\]
in $\CD^+(\MHM(\CM_{\CA}))$.

\begin{proof}[Proof of Theorem~\ref{theorem:relPBW}]
We want to prove that $\Psi^{\psi}_{\CA}$ is an isomorphism (or equivalently by conservativity of the functor $\rat$, that $\rat(\Psi_{\CA}^{\psi})$ is an isomorphism). We let $\ul{\CC}\coloneqq\cone(\Psi_{\CA})$. If $\ul{\CC}$ is non-zero, then there exists a closed point $x\in\CM_{\CA}$ such that, with $\imath_x$ the inclusion of $x$, $\imath_x^!\ul{\CC}\neq 0$. We let $\underline{\CF}$ be the collection of simple direct summands of $\CA$ corresponding to $x$. Then, $\iota_{\underline{\CF}}^!\ul{\CC}$ is non-zero, so $\imath_{\underline{\CF}}^!\Psi_{\CA}^{\psi}$ is not an isomorphism. But $\imath_{\underline{\CF}}^!\Psi_{\CA}^{\psi}$ can be identified with the morphism $\imath_{\CN}^!\Psi_{\Pi_{Q_{\underline{\CF}}}}^{\psi}$ and so the latter cannot be an isomorphism. Then $\Psi_{\Pi_{Q_{\underline{\CF}}}}^{\psi}$ is not an isomorphism. This is in contradiction with the PBW theorem\footnote{We remark that although the proof of \cite[Theorem D]{davison2016integrality}, and the proof of \cite[Theorem C]{davison2020cohomological}, on which it is based, are written for a specific choice of positive determinant bundle, the proof works without modification for an arbitrary positive choice.} for preprojective algebras \cite[Theorem D]{davison2016integrality}, using the isomorphism of Proposition~\ref{proposition:BPS3dcoincide}.
\end{proof}

\section{Cuspidal polynomials of quivers}
\label{section:Cuspidalpolsquivers}

\subsection{Cuspidal polynomials}
\label{subsection:cuspidals}
Our main theorem (Theorem \ref{corollary:absBor}) can be applied to proving the positivity of absolutely cuspidal polynomials of quivers \cite{bozec2019counting}, a strengthening of the Kac positivity conjecture \cite[Conjecture 2]{kac1983root}, in full generality.

Let $Q=(Q_1,Q_0,s,t\colon Q_1\rightarrow Q_0)$ be a quiver. In the early '90s Ringel and Green defined \cite{ringel1990hall,ringel1992hall,green1995hall} the \emph{Hall bialgebra} of $Q$ over a finite field $\BoF_q$. As a vector space over $\BoC$, it has a basis given by the set of isomorphism classes of representations of $Q$ over $\BoF_q$:
\[
 H_{Q,\BoF_q}:=\bigoplus_{[M]\in\Rep_Q(\BoF_q)/\sim}\BoC[M]
\]
where $\Rep_Q(\BoF_q)$ denotes the category of finite-dimensional modules and the product is the usual Hall-type product: we recommend \cite{schiffmann2006lectures} for background. Green \cite{green1995hall} defined a coproduct $\Delta$ on $ H_{Q,\BoF_q}$, so that $H_{Q,\BoF_q}$ has the structure of a twisted Hopf algebra (see Xiao's paper \cite{xiao1997drinfeld} for an explicit description of the antipode). The twist, which will not be used in this paper, is explained in [loc. cit.]. 

Sevenhant and Van den Bergh proved \cite{sevenhant2001relation} that $H_{Q,\BoF_q}$ has the structure of the positive part of the quantised enveloping algebra of a generalised Kac--Moody algebra with deformation parameter specialised at $\sqrt{q}$. The generators are given by the space of \emph{cuspidal functions}. These are by definition the primitive elements of the coproduct:
\[
 H_{Q,\BoF_q}^{\cusp}=\{f\in H_{Q,\BoF_q}\mid \Delta(f)=f\otimes 1+1\otimes f\}.
\]
They satisfy Serre relations, see \cite[Theorem 3.4]{hennecart2021isotropic} for a detailed formulation. In \cite{sevenhant2001relation} it is assumed that the quiver is loop-free, although this assumption can be removed.

For $\dd\in\BoN^{Q_0}$, let $M_{Q,\dd}(q)$ be the number of isomorphism classes of $\BoF_q$--representations of $Q$ of dimension vector $\dd$, and let $I_{Q,\dd}(q)$ be the number of isomorphism classes of \emph{indecomposable} $\BoF_q$--representations of $Q$ of dimension vector $\dd$. We denote by $A_{Q,\dd}(q)$ the \emph{Kac polynomial}, counting the number of isomorphism classes of \emph{absolutely indecomposable} $\BoF_q$--representations of $Q$ of dimension vector $\dd$. By Kac \cite{kac1983root}, all of these functions are polynomials in $q$. Moreover, by \cite{hausel2013positivity} (or \cite{davison2018purity,dobrovolska2016moduli} for different approaches), the coefficients of $A_{Q,\dd}(q)$ are nonnegative, as originally conjectured by Kac in \cite{kac1983root}.

Given a $\BoN^{Q_0}$-graded vector space $V$, we define $\ch(V)=\sum_{\dd\in \BoN^{Q_0}}\dim(V_{\dd})z^{\dd}$. If $V$ also carries a cohomological $\BoZ$-grading, we define 
\[
\ch^{\BoZ}(V)\coloneqq\sum_{n\in \BoZ}\sum_{\dd\in\BoN^{Q_0}}\dim(V_{\dd}^n)q^{n/2}z^{\dd}.
\]

The $\BoN^{Q_0}$-graded character of the Hall algebra is given by the formula
\begin{equation}
\label{equation:gradedcharacterHallalgebra}
 \ch(H_{Q,F_q})=\sum_{\dd\in\BoN^{Q_0}}M_{Q,\dd}(q)z^{\dd}=\Exp_{z}\left(\sum_{\dd>0}I_{Q,\dd}(q)z^{\dd}\right)=\Exp_{q,z}\left(\sum_{\dd>0}A_{Q,\dd}(q)z^{\dd}\right),
\end{equation}
where $\Exp_z$ and $\Exp_{q,z}$ denote the plethystic exponentials, see \cite[\S 1.5]{bozec2019counting}. The second equality follows from the Krull--Schmidt property of the category of representations of the quiver and the third from Galois descent for quiver representations.

The space $H_{Q,\BoF_q}^{\cusp}$ is naturally graded by the dimension vector: $H_{Q,\BoF_q}^{\cusp}=\bigoplus_{\dd\in\BoN^{Q_0}}H_{Q,\BoF_q}^{\cusp}[\dd]$. There has been a growing interest in understanding this space and to find a parametrisation of cuspidal functions \cite{bozec2019counting, hennecart2021isotropic}. The first step was to compute its dimension. We have the following result.

\begin{theorem}[{\cite[Theorem 1.1]{bozec2019counting}}]
 The dimension $\dim_{\BoC} H_{Q,\BoF_q}^{\cusp}[\dd]$ is given by a polynomial with rational coefficients $C_{Q,\dd}(q)\in \BoQ[q]$.
\end{theorem}
The vector space $H_{Q,\BoF_q}^{\cusp}[\dd]$ is ungraded, and its dimension is determined by $q$, via the evaluation of the polynomial $C_{Q,\dd}(q)$. We will see that this same polynomial is obtained as the characteristic function of a \textit{fixed} cohomologically graded vector space.

Bozec and Schiffmann combinatorially defined a new family of polynomials $(C_{Q,\dd}^{\rmabs}(q))_{\dd\in\BoN^{Q_0}}$ from the family $(C_{Q,\dd}(q))_{\dd\in\BoN^{Q_0}}$, which enjoy more favourable properties. We denote by $\BoN^{Q_0}_{\prim}$ the set of indivisible dimension vectors, i.e. of $\dd\in\BoN^{Q_0}$ which cannot be written $l\dd'$ with $l\in\BoZ_{\geq 2}$ and $\dd'\in\BoN^{Q_0}$. Then 
\begin{itemize}
\item $C_{Q,\dd}^{\rmabs}(q)=C_{Q,\dd}(q) \text{ if $\chi_Q(\dd,\dd)<0$}$,
\item $\Exp_{z}\left(\sum_{l\in\BoZ_{>0}}C_{Q,l\dd}(q)z^{l \dd}\right)=\Exp_{q,z}\left(\sum_{l\in\BoZ_{>0}}C_{Q,l\dd}^{\rmabs}(q)z^{l \dd}\right)$ if $\dd\in(\BoN^{Q_0})_{\prim}$ and $\chi_Q(\dd,\dd)=0$,
\item $C_{Q,\dd}^{\rmabs}(q)=0$ otherwise.
\end{itemize}
The definition is motivated by the fact that if there exists a $\BoN\times \BoN^{Q_0}$-graded GKM Lie algebra $\mathfrak{g}_{Q}^{\BoN}$ associated with the monoid $(\BoN^{Q_0},(-,-))$ (in the sense of \S\ref{subsection:GKMmonoids}), whose positive part $\mathfrak{n}^{\BoN,+}_{Q}$ has graded character
\[
 \ch^{\BoZ}(\mathfrak{n}_Q^{\BoN,+})=\sum_{\dd\in\BoN^{Q_0}}A_{Q,\dd}(q)z^{\dd},
\]
then the $\BoN\times\BoN^{Q_0}$-graded dimension of the space $\bigoplus_{\dd\in\BoN^{Q_0}\setminus\{0\}}V_{\dd}$ of simple roots of $\Fg_Q^{\BoN}$ is given by the generating series
\begin{equation}
\label{equation:charBKM}
 \ch^{\BoZ}(V_{\dd})=\sum_{\dd\in\BoN^{Q_0}\setminus\{0\}}C^{\rmabs}_{Q,\dd}(q)z^{\dd},
\end{equation}
see \cite{bozec2019counting}.

\begin{theorem}[{\cite[Theorem 1.4, Theorem 1.6]{bozec2019counting}}]
 The polynomials $C_{Q,\dd}^{\rmabs}(q)$ have integer coefficients.
\end{theorem}

In \cite{bozec2019counting} Bozec and Schiffmann conjecture the following, see also \cite[Conjecture 3.4]{schiffmann2018kac}.

\begin{conjecture}[]
\label{conjecture:BozecSchiffmann}
 For any $\dd\in\BoN^{Q_0}$, $C_{Q,\dd}^{\rmabs}(q)\in\BoN[q]$.
\end{conjecture}
This conjecture is known for isotropic dimension vectors \cite{deng2003new,bozec2019counting, hennecart2021isotropic}, and Theorem \ref{theorem:ICNQV} establishes it in full generality. For isotropic dimension vectors, we have the following, from \cite[\S 1.2 ii)]{bozec2019counting}.
\begin{proposition}
\label{proposition:isotropiccuspidalpolynomials}
 Let $\dd\in\BoN^{Q_0}$ be isotropic: $\chi_Q(\dd,\dd)=0$. We have:
 \[
 C^{\rmabs}_{Q,\dd}(q)=
 \left\{
 \begin{aligned}
 &q \text{ if the support of $\dd$ is an affine quiver or the Jordan quiver}\\
 &0 \text{ otherwise.}
 \end{aligned}
 \right.
 \]

\end{proposition}

Let $Q$ be a quiver. Assuming that the coefficients of the polynomials $C_{Q,\dd}^{\rmabs}(q)$ are nonnegative, Bozec and Schiffmann prove that the $\BoN^{Q_0}\times \BoN$-graded Lie algebra $\mathfrak{n}_Q^{\BoN,+}\coloneqq\mathfrak{n}^+_{\BoN^{Q_0},(-,-),C_{Q,\dd}^{\rmabs}}$ is such that 
\begin{equation}
\label{equation:charnQN}
 \ch^{\BoZ}(\UEA(\mathfrak{n}^{\BoN,+}_Q))=\Exp_{q,z}\left(\sum_{\dd\in\BoN^{Q_0}\setminus\{0\}}A_{Q,\dd}(q)z^{\dd}\right).
\end{equation}

By the graded PBW theorem, this is equivalent to the equality \eqref{equation:charBKM}. Conversely, if such a Borcherds Lie algebra $\mathfrak{g}_Q^{\BoN}$ exists, then the polynomials $C^{\rmabs}_{Q,\dd}(q)$ have nonnegative coefficients: they are given by the equality

\begin{equation}
\label{equation:Ngraded}
 \sum_{\dd\in\BoN^{Q_0}}C_{Q,\dd}^{\rmabs}(q)z^{\dd}=\ch^{\BoZ}(\mathfrak{n}_Q^{\BoN,+}/[\mathfrak{n}_Q^{\BoN,+},\mathfrak{n}_Q^{\BoN,+}]).
\end{equation}

By \cite[\S 1.2]{davison2020bps}, the 3d BPS Lie algebra , $\mathfrak{g}_{\tilde{Q},\tilde{W}}\cong\HO^*\BPS_{\Pi_Q,\Lie}^{\td}$, is a $\BoN^{Q_0}\times 2\BoN_{\leq 0}$-graded Lie algebra with character
 
 \begin{equation}
 \label{equation:charBPS}
 \ch^{\BoZ}(\mathfrak{g}_{\tilde{Q},\tilde{W}})=\sum_{\dd\in\BoN^{Q_0}}A_{Q,\dd}(q^{-1})z^{\dd}
 \end{equation}
which by Proposition~\ref{proposition:BPS3dcoincide} is isomorphic to $\Fn^{\ttBPS,+}_{\Pi_Q}$. Via this isomorphism $\mathfrak{g}_{\tilde{Q},\tilde{W}}$ is identified with the positive part of a GKM Lie algebra associated to the monoid $\BoN^{Q_0}$ with the bilinear form $(-,-)_Q$, symmetrisation of the Euler form of $Q$, and with graded dimension of the spaces of generators given by
\begin{equation}
\label{equation:BPSchar}
 \ch^{\BoZ}(\mathfrak{g}_{\tilde{Q},\tilde{W}}/[\mathfrak{g}_{\tilde{Q},\tilde{W}},\mathfrak{g}_{\tilde{Q},\tilde{W}}])=\sum_{\vec{d}\in\Sigma_{\Pi_Q}}\IP(\mathcal{M}_{\Pi_Q,\dd},q)z^{\dd}+\sum_{\substack{\vec{d}\in\Sigma_{\Pi_Q}\\(\vec{d},\vec{d})=0\\l\geq 2}}\IP(\mathcal{M}_{\Pi_Q,\vec{d}},q)z^{l\vec{d}},
\end{equation}
the intersection Poincar\'e polynomials of $\mathcal{M}_{\Pi_Q,\dd}$ for various $\dd$. 

We finally come to the main combinatorial consequence of Theorem \ref{corollary:absBor}:
\begin{theorem}
\label{theorem:ICNQV}
Conjecture~\ref{conjecture:BozecSchiffmann} is true: for any $\vec{d}\in\BoN^{Q_0}$, $C^{\rmabs}_{Q,\dd}(q)\in\BoN[q]$. Moreover, for any $\dd\in\BoN^{Q_0}$, $\dd\in\Sigma_Q$,
 \[
 C^{\rmabs}_{Q,\dd}(q^{-2})=\IP(\mathcal{M}_{\Pi_Q,\dd},q).
 \]
\end{theorem}
\begin{proof}
This comes from the comparison of Equations \eqref{equation:Ngraded} and \eqref{equation:BPSchar}, given that the character of the GKM Lie algebra $\mathfrak{n}_{\Pi_Q}^{\aBPS,+}$ is given by Equation \eqref{equation:charBPS}, which is \eqref{equation:charnQN} up to the change of variables $q\leftrightarrow q^{-2}$.
\end{proof}

\begin{corollary}
 \begin{enumerate}
 \item $C^{\rmabs}_{Q,\dd}(q)\neq 0$ if and only if $\dd\in\Phi^+_{\Pi_Q}$,
 \item If $\dd\in\Phi^+_{\Pi_Q}$, $\deg C_{Q,\dd}^{\rmabs}(q)=1-\chi_Q(\dd,\dd)$ and $C_{Q,\dd}^{\rmabs}(q)$ is a monic polynomial.
 \end{enumerate}
\end{corollary}
\begin{proof}
 The first statement follows immediately from Theorem \ref{theorem:ICNQV}. The second statement follows from the fact that for $\dd\in\Sigma_{\Pi_Q}$, the singular Nakajima quiver variety $\CM_{\Pi_Q,\dd}$ is irreducible of dimension $2-(\dd,\dd)_Q$ by \cite{crawley2001geometry} and for $\dd\in\Sigma_{\Pi_Q}$, $(\dd,\dd)_Q=0$ and $l\geq 1$, $\CM_{\Pi_Q,l\dd}$ is irreducible of dimension $2l$ by \cite{crawley2002decomposition}.
\end{proof}

\subsection{Intersection Poincar\'e polynomials of singular quiver varieties}

The canonical decomposition of a dimension vector $\dd\in\BoN^{Q_0}$ is the unique decomposition of $\dd$ as a sum of primitive positive roots (elements of $\Sigma_{\Pi_Q}$) such that any other decomposition as a sum of primitive positive roots is a refinement of it. By \cite{crawley2002decomposition}, it is well-defined and unique. Moreover, it gives a product decomposition of the corresponding quiver variety \cite{crawley2002decomposition}.

\begin{corollary}
\label{corollary:IPallNQV}
 Let $\dd\in\BoN^{Q_0}$ be a dimension vector and $\dd=\sum_{j=1}^rm_j\dd_j$ be the canonical decomposition of $\dd$, where the $\dd_j$ are pairwise distinct, $\dd_j\in\Sigma_{\Pi_Q}$ and $m_j\geq 0$. We have
 \[
 \IP(\CM_{\Pi_Q,\dd})(t)=\prod_{j=1}^r\frac{1}{m_j!}\left(\left(\frac{\mathrm{d}}{\mathrm{d}u}\right)^{m_j}\Exp_{t,u}(C_{Q,\dd_j}^{\rmabs}(t^{-2})u)\right)_{u=0}
 \]
where $u$ is an additional formal variable and $\Exp_{t,u}$ is the plethystic exponential.
\end{corollary}
\begin{proof}
This follows from the product decomposition of quiver varieties \cite[Theorem 1.1]{crawley2002decomposition} and the formula for intersection Poincar\'e polynomials of symmetric products.
\end{proof}

The canonical decomposition of a dimension vector $\dd\in\BoN^{Q_0}$ can be determined algorithmically. Indeed, for a given $\dd\in\BoN^{Q_0}$, it is possible to determine the set of primitive positive roots $\dd'\leq \dd$ (this amounts to checking a certain finite number of inequalities for all decompositions of $\dd'$). Then, Crawley-Boevey proves that the canonical decomposition is such that any other decomposition of $\dd$ as a sum of elements of $\Sigma_{\Pi_Q}$ refines it. Therefore, we can start from the trivial decomposition $\dd=\sum_{j=1}^r\dd_j\coloneqq \sum_{i\in Q_0}d_i1_i$ where each $\dd_j$ is equal to $1_i$ for some vertex $i\in Q_0$. We can then iteratively combine some of the dimension vectors $\dd_j$ when their sum is in $\Sigma_{\Pi_Q}$. In conclusion, intersection Poincar\'e polynomials of all quiver varieties can be algorithmically computed using our results.

\subsection{(Semi)-nilpotent Kac and cuspidal polynomials}
Bozec, Schiffmann and Vasserot defined in \cite{bozec2020number} various families of counting polynomials for quiver representations over finite fields, obtained by imposing appropriate nilpotency conditions. They related the polynomials obtained this way to the count of points of some quiver varieties. In this section, we recall these definitions and we define the corresponding \emph{cuspidal} and \emph{absolutely cuspidal} polynomials, following \cite{bozec2019counting}. We use the general results from this paper to prove the positivity of all of the coefficients of these polynomials.

\subsubsection{Nilpotent and semi-nilpotent representations of quivers}
\label{subsubsection:seminilpotentrepquivers}
We recall various notions of nilpotency for quiver representations from \cite{bozec2020number}. Let $Q=(Q_0,Q_1)$ be a quiver. A representation $M$ of $Q$ is called
\begin{enumerate}
 \item \emph{nilpotent} if there exists a flag of $Q_0$-graded vector spaces $\{0\}=M_0\subset M_1\subset\hdots\subset M_r=M$ such that the action of the arrows of $Q$ sends $M_i$ to $M_{i-1}$;
 \item \emph{$1$-nilpotent} if there exists a flag of $Q_0$-graded vector spaces $\{0\}=M_0\subset M_1\subset \hdots\subset M_r=M$ such that loops of $Q$ send $M_i$ to $M_{i-1}$ (with no conditions for other arrows).
\end{enumerate}

Let $\overline{Q}=(Q_0,\overline{Q}_1)$ be the double of the quiver $Q$. We have $\overline{Q}_1=Q_1\sqcup Q_1^*$, and each arrow $\alpha\in Q_1$ has an opposite arrow $\alpha^*\in Q_1^*$. A representation $M$ of $\overline{Q}$ is called
\begin{enumerate}
 \item \label{item:seminilp}\emph{semi-nilpotent} if there exists a flag of $Q_0$-graded vector spaces $\{0\}=M_0\subset M_1\subset\hdots\subset M_r=M$ such that arrows $\alpha\in Q_1$ send $M_i$ to $M_{i-1}$ and arrows $\alpha\in Q_1^*$ send $M_i$ to $M_i$.
 \item \label{item:strictlyseminilp}\emph{strictly semi-nilpotent} if there exists a flag of $Q_0$-graded vector spaces $\{0\}=M_0\subset M_1\subset\hdots\subset M_r=M$ such that consecutive subquotients are supported at a single vertex, each arrow $\alpha\in Q_1$ sends $M_i$ to $M_{i-1}$ and arrows $\alpha\in Q_1^*$ send $M_i$ to $M_i$.
 \item \label{item:starseminilp}$*$-semi-nilpotent if the condition semi-nilpotent \eqref{item:seminilp} is satisfied with $Q_1$ and $Q_1^*$ exchanged.
 \item \label{item:starstrictlyseminilp} $*$-strictly semi-nilpotent if the condition strictly semi-nilpotent \eqref{item:strictlyseminilp} with $Q_1$ and $Q_1^*$ exchanged is satisfied.
 \end{enumerate}
Since the involution of $\Pi_Q$ swapping each arrow $a$ with $a^*$ swaps (3) with (1) and (4) with (2), the BPS cohomology for the $*$-versions is isomorphic to the non-$*$-versions, which we concentrate on exclusively from now on.
\subsubsection{Hall algebras}
Recall from \S\ref{subsection:cuspidals} that we have the Hall algebra of $Q$ over a finite field $\BoF_q$, $H_{Q,\BoF_q}$, which is a bialgebra.
 We define similarly
 \[
 H_{Q,\BoF_q}^{\nil}\coloneqq \bigoplus_{\substack{[M]\in\Rep_Q(\BoF_q))/\sim\\M \text{ is nilpotent}}}\BoC[M];\quad\quad H_{Q,\BoF_q}^{1\snil}\coloneqq \bigoplus_{\substack{[M]\in\Rep_Q(\BoF_q))/\sim\\M \text{ is $1$-nilpotent}}}\BoC[M].
\]
Since the categories $\Rep_Q^{\nil}(\BoF_q)$ and $\Rep_Q^{1\snil}(\BoF_q)$ of nilpotent and $1$-nilpotent representations of $Q$ are \emph{Serre subcategories} (i.e. stable under extension, subobjects and quotients), $H_{Q,\BoF}^{\nil}\subset H_{Q,\BoF_q}^{1\snil}\subset H_{Q,\BoF_q}$ are inclusions of biagebras.

\subsubsection{Kac polynomials}
\label{subsubsection:nilpKacpols}

In \cite{bozec2020number} Bozec, Schiffmann and Vasserot defined and studied Kac polynomials for the categories $\Rep_Q^{\nil}(\BoF_q)$ and $\Rep_Q^{1\snil}(\BoF_q)$ (in particular by relating them to point counts of Nakajima quiver varieties). Namely, one may define as in \S\ref{subsection:cuspidals} the counting functions $M_{Q,\dd}^{\nil}(q)$, $I_{Q,\dd}^{\nil}(q)$ and $A_{Q,\dd}^{\nil}(q)$, $M_{Q,\dd}^{1\snil}(q)$, $I_{Q,\dd}^{1\snil}(q)$ and $A_{Q,\dd}^{1\snil}(q)$ for any $\dd\in\BoN^{Q_0}$ by considering respectively nilpotent and $1$-nilpotent representations of $Q$, and all, or indecomposable, or absolutely indecomposable isomorphism classes of such representations.

In both the cases of nilpotent and of $1$-nilpotent representations, we have the sequence of equalities analogous to \eqref{equation:gradedcharacterHallalgebra} and by \cite{bozec2020number}, all these counting functions are polynomials in $q$. The analogue of the Kac positivity conjecture was proved in \cite{davison2020bps}:
\begin{proposition}[{\cite[\S4.2]{davison2020bps}}]
 The nilpotent versions of Kac polynomials $A_{Q,\dd}^{\nil}(q)$ and $A_{Q,\dd}^{1\snil}(q)$ have nonnegative coefficients.
\end{proposition}

\subsubsection{Cuspidal polynomials}
We define cuspidal functions in $H_{Q,\BoF_q}^{\nil}$ and $H_{Q,\BoF_q}^{1\snil}$ as
\[
 H_{Q,\BoF_q}^{\nil,\cusp}=H_{Q,\BoF_q}^{\nil}\cap H_{Q,\BoF_q}^{\cusp}, \quad H_{Q,\BoF_q}^{1\snil,\cusp}=H_{Q,\BoF_q}^{1\snil}\cap H_{Q,\BoF_q}^{\cusp}.
\]
We define the functions $C^{\nil}_{Q,\dd}(q)$ and $C^{1\snil}_{Q,\dd}(q)$, $\dd\in\BoN^{Q_0}$ as
\[
 C^{\nil}_{Q,\dd}(q)\coloneqq \dim_{\BoC} H_{Q,\BoF_q}^{\nil,\cusp}, \quad C^{1\snil}_{Q,\dd}(q)\coloneqq\dim_{\BoC}H_{Q,\BoF_q}^{1\snil,\cusp}.
\]
The inductive description of cuspidal functions $C_{Q,\dd}$ in \cite{bozec2019counting} works without modification for nilpotent and $1$-nilpotent representations and this together with the polynomiality result for the numbers of isoclasses of nilpotent and $1$-nilpotent representations provides the polynomiality of the nilpotent versions of the functions $C_{Q,\dd}^{\nil}$ and $C_{Q,\dd}^{1\snil}$. One of the key points is the validity of the Sevenhant--Van den Bergh theorem \cite[Theorem 1.1]{sevenhant2001relation} describing the Hall algebra of a quiver as a quantum group for quivers with loops and the categories of representations with or without nilpotency conditions. Sevenhant and Van den Bergh assume that the quiver under consideration is loop-free and do not consider the case of nilpotent representations, but their methods adapt in a straightforward way.

We define \emph{absolutely cuspidal polynomials} by the same combinatorial procedure as in \cite{bozec2019counting}. That is, for $\square\in\{\nil,1\snil\}$, we define $C_{Q,\dd}^{\square,\rmabs}(q)$ as
\[
 \begin{aligned}
 &C_{Q,\dd}^{\square,\rmabs}(q)=C_{Q,\dd}^{\square}(q) &\text{ if $\chi_Q(\dd,\dd)<0$}\\
 &\Exp_{z}\left(\sum_{l\in\BoZ_{>0}}C^{\square}_{Q,l\dd}(q)z^{l\dd}\right)=\Exp_{q,z}\left(\sum_{l\in\BoZ_{>0}}C^{\square,\rmabs}_{Q,l\dd}(q)z^{l\dd}\right)&\text{ if $\dd\in(\BoN^{Q_0})_{\prim}$, \;$\chi_Q(\dd,\dd)=0$}.
 \end{aligned}
\]
Using Proposition~\ref{proposition:isotropiccuspidalpolynomials} and the explicit description of isotropic cuspidal functions in \cite[Theorems 5.2 and 5.4]{hennecart2021isotropic}, the following is easily obtained.

\begin{proposition}
Let $\dd\in\BoN^{Q_0}$ be isotropic: $\chi_Q(\dd,\dd)=0$. Then
\[
 C_{Q,\dd}^{\nil,\rmabs}(q)=
 \left\{
 \begin{aligned}
 1&\text{ if the support of $\dd$ is the Jordan quiver or a cyclic quiver with cyclic orientation,}\\
 q&\text{ if the support of $\dd$ is an affine ADE quiver without cycles,}\\
 0&\text{ otherwise.}
 \end{aligned}
 \right.
\]
and
\[
 C_{Q,\dd}^{1\snil,\rmabs}(q)=
 \left\{
 \begin{aligned}
 1&\text{ if the support of $\dd$ is the Jordan quiver,}\\
 q&\text{ if the support of $\dd$ is an affine ADE quiver,}\\
 0&\text{ otherwise.}
 \end{aligned}
 \right.
 \]
\end{proposition}
Repeating the proofs of \cite{bozec2019counting}, one obtains the following result.

\begin{proposition}
\label{proposition:interpretationcusppols}
Let $\square\in\{\nil,1\snil\}$. If there exists a $\BoZ^{Q_0}$-graded generalised Kac--Moody Lie algebra $\mathfrak{g}_{Q}^{\square}$, associated to the monoid $\BoN^{Q_0}$ with bilinear form given by the symmetrised Euler form $(-,-)_Q\colon \BoN^{Q_0}\times\BoN^{Q_0}\rightarrow\BoZ$, with graded character of the positive half given by
\begin{align*}
 \ch^{\BoZ}(\Fn_{Q}^{\square,+})\coloneqq&\sum_{\dd\in \BoN^{Q_0}}\sum_{i\in \BoZ}\dim((\Fn_{Q,\dd}^{\square,+})^i)q^{i/2}z^{\dd}
=\sum_{\dd\in\BoN^{Q_0}}A_{Q,\dd}^{\square}(q)z^{\dd},
\end{align*}
then it has weight function
\[
\begin{matrix}
 C_Q^{\square,\rmabs}&\colon& \BoN^{Q_0}&\rightarrow&\BoN[t^{\pm 1}]\\
       &   & \dd    &\mapsto  &C_{Q,\dd}^{\square,\rmabs}(t).
\end{matrix}
\]
\end{proposition}

\subsubsection{The (strictly) semi-nilpotent BPS Lie algebras}
We let $\CM_{\Pi_Q}^{\mathcal{SN}}\subset\CM_{\Pi_Q}$ be the closed submonoid parametrising semisimple representations of $\Pi_Q$ such that the only arrows of $\overline{Q}$ that may act nontrivially are the arrows $\alpha\in Q_1^*$ (i.e. semisimple semi-nilpotent representations of $\Pi_Q$). We let $\CM_{\Pi_Q}^{\SSN}\subset \CM_{\Pi_Q}^{\mathcal{SN}}$ be the submonoid parametrising semisimple representations $M$ of $\Pi_Q$ such that the only arrows of $\overline{Q}$ that may act nontrivially are the \emph{loops} $\alpha\in Q_1^*$ (i.e. semisimple strictly semi-nilpotent representations).

For $\triangle\in\{\mathcal{SN},\SSN\}$, the stacks of $\triangle$ representations of $\Pi_Q$ are defined by the Cartesian squares
\[
 \begin{tikzcd}
	{\FM_{\Pi_Q}^{\triangle}} & {\FM_{\Pi_Q}} \\
	{\CM_{\Pi_Q}^{\triangle}} & {\CM_{\Pi_Q}.}
	\arrow[from=1-2, to=2-2]
	\arrow[from=1-1, to=2-1]
	\arrow["\imath_{\triangle}",from=2-1, to=2-2]
	\arrow[from=1-1, to=1-2]
	\arrow["\lrcorner"{anchor=center, pos=0.125}, draw=none, from=1-1, to=2-2]
\end{tikzcd}
\]

Using the pull-back of the sheaf-theoretic Hall algebra in \S\ref{subsection:theCoHA}, we have the relative $\triangle$ cohomological Hall algebras
\[
 \SA_{\Pi_Q}^{\triangle,\psi}\coloneqq \imath_{\Delta}^!\SA_{\Pi_Q}^{\psi},
\]
with the multiplication $\imath_{\triangle}^!\mult^{\psi}$. It is a complex of mixed Hodge modules on $\CM_{\Pi_Q}^{\triangle}$. The absolute $\triangle$ CoHA is defined to be $\HO^*\!\!\SA_{\Pi_Q}^{\triangle,\psi}$.

We similarly define the relative $\triangle$ BPS algebra and BPS Lie algebra as
\[
 \BPS_{\Pi_Q,\Alg}^{\triangle,\psi}\coloneqq \imath_{\triangle}^{!}\BPS_{\Pi_Q,\Alg}^{\psi}, \quad \BPS_{\Pi_Q,\Lie}^{\triangle}\coloneqq \imath_{\triangle}^{!}\BPS_{\Pi_Q,\Lie},
\]
and their absolute versions by taking the derived global sections
\[
 \rmBPS_{\Pi_Q,\Alg}^{\triangle,\psi}\coloneqq \HO^*\!\BPS_{\Pi_Q,\Alg}^{\triangle,\psi}, \quad \Fn_{\Pi_Q}^{\triangle,+}\coloneqq \HO^*\!\BPS_{\Pi_Q,\Lie}^{\triangle}.
\]

By \cite{davison2020bps}, we have
\begin{equation}
\label{equation:characterSSSN}
 \ch^{\BoZ}(\Fn_{\Pi_Q}^{\mathcal{SN},+})=\sum_{\dd\in\BoN^{Q_0}}A_{Q,\dd}^{\nil}(q)z^{\dd}\quad \quad \textrm{and}\quad\quad \ch^{\BoZ}(\Fn_{\Pi_Q}^{\mathcal{SSN},+})=\sum_{\dd\in\BoN^{Q_0}}A_{Q,\dd}^{1\snil}(q)z^{\dd}.
\end{equation}
\begin{remark}
Note that the power of $q$ in Kac polynomials appearing in \eqref{equation:characterSSSN} is $q$ and not $q^{-1}$ as it was in \eqref{equation:charBPS}. This is the same symmetry appearing in \cite[Theorem 1.4 and \S1.7.1]{bozec2020number}.
\end{remark}
\begin{theorem}
\label{theorem:positivitycusppols}
 For $\square\in\{\nil,1\snil\}$ and any $\dd\in\BoN^{Q_0}$, we have $C_{Q,\dd}^{\square,\rmabs}(q)\in\BoN[q]$. More precisely, for $\dd\in\Sigma_{\Pi_Q}$, we have
 \[
 C_{Q,\dd}^{\nil,\rmabs}(q)=\sum_{j\in\BoZ}\HO^{j}(\imath_{\dd,\mathcal{SN}}^!\IC(\CM_{\Pi_Q,\dd}))q^{j/2}
 \]
and
\[
 C_{Q,\dd}^{1\snil,\rmabs}(q)=\sum_{j\in\BoZ}\HO^{j}(\imath_{\dd,\SSN}^!\IC(\CM_{\Pi_Q,\dd}))q^{j/2}.
\]
If $\dd\in\Sigma_{\Pi_Q}$ is isotropic and $l\in\BoN_{>0}$, then
\[
 C_{Q,l\dd}^{\nil,\rmabs}(q)=C_{Q,\dd}^{\nil,\rmabs}(q), \quad C_{Q,l\dd}^{1\snil,\rmabs}(q)=C_{Q,\dd}^{1\snil,\rmabs}(q).
\]
\end{theorem}
\begin{proof}
 By Proposition~\ref{proposition:interpretationcusppols} and Equations \eqref{equation:characterSSSN}, it suffices to prove that $\Fn_{\Pi_Q}^{\mathcal{SN},+}$ and $\Fn_{\Pi_Q}^{\mathcal{SSN},+}$ are positive halves of generalised Kac--Moody Lie algebra associated with the monoid $\BoN^{Q_0}$ with the symmetrised Euler form and with cohomologically graded spaces of roots given by $\HO^*\!\!\imath_{\triangle,\dd}^!\SG_{\Pi_Q,\dd}$ for $\triangle\in\{\mathcal{SN},\SSN\}$ respectively. This is precisely what is obtained by applying the strictly monoidal functor $\HO^*\!\!\imath_{\triangle}^!$ to Theorem~\ref{theorem:BPSalgisenvBorcherds}.
\end{proof}

\subsection{The fully nilpotent CoHA and BPS algebras}
Let $Q=(Q_0,Q_1)$ be a quiver. Recall the embedding of monoids $\imath_{\CN}\colon\BoN^{Q_0}\rightarrow\CM_{\Pi_Q}$ sending $\dd\mapsto 0_{\dd}$, where $0_{\dd}$ is the $\dd$-dimensional representation of $\Pi_Q$ on which all arrows of $\overline{Q}$ act by $0$. The \emph{fully nilpotent cohomological Hall algebra} is $\ulrelCoHA_{\Pi_Q}^{\CN}\coloneqq \imath_{\CN}^!\ulrelCoHA_{\Pi_Q}$. This is a (cohomologically graded) mixed Hodge module on $\BoN^{Q_0}$ and it carries the CoHA product given by $\imath_{\CN}^!\mult$. It has sub-mixed Hodge modules given by $\rmBPS^{\CN}_{\Pi_Q}\coloneqq \imath_{\CN}^{!}\ulBPS_{\Pi_Q,\Alg}$ (the fully nilpotent BPS algebra) and $\rmBPS_{\Pi_Q,\Lie}^{\CN}\coloneqq \imath_{\CN}^!\ulBPS_{\Pi_Q,\Lie}$ (the fully nilpotent BPS Lie algebra). As usual, we let $\ulrelCoHA_{\Pi_Q}^{\CN,\psi}$ and $\rmBPS_{\Pi_Q,\Alg}^{\CN,\psi}$ be the corresponding $\psi$-twisted algebras. We have an inclusion of algebras $\rmBPS_{\Pi_Q,\Alg}^{\CN,\psi}\subset \ulrelCoHA_{\Pi_Q}^{\CN,\psi}$ and an inclusion of Lie algebras $\rmBPS_{\Pi_Q,\Lie}^{\CN}\subset \rmBPS_{\Pi_Q,\Alg}^{\CN,\psi}$. We can summarize in the following proposition the properties satisfied by these algebras.

\begin{proposition}
\label{proposition:BPSfullynilpotent}
We have the following:
\begin{enumerate}
 \item $\rmBPS_{\Pi_Q,\Alg}^{\CN,\psi}\cong\UEA(\rmBPS_{\Pi_Q,\Lie}^{\CN})$,
 \item $\ch^{\BoZ}(\rmBPS_{\Pi_Q,\Lie}^{\CN})=\sum_{\dd\in\BoN^{Q_0}}A_{Q,\dd}(q)z^{\dd}$,
 \item $\rmBPS_{\Pi_Q,\Lie}^{\CN}\cong \Fn^{\CN,+}_{\Pi_Q}$ is the GKM Lie algebra associated to the monoid $\BoN^{Q_0}$ with the symmetrised Euler form of $Q$ and the weight function
 \[
 \begin{matrix}
  \Phi_{\Pi_Q}^+&\rightarrow&\BoN[t^{\pm1/2}]&\\
  \dd&\mapsto&\sum_{j\in\BoZ}\dim\HO^j(\BD\SG_{\dd})t^{-j/2}&=C_{Q,\dd}^{\rmabs}(t).
 \end{matrix}
 \]
where $\SG_{\dd}$ is as in \S\ref{subsubsec:ppre}.
\end{enumerate}
\end{proposition}
\begin{proof}
 This follows from the application of the strictly monoidal functor $\imath_{\CN}^!$ to Theorem~\ref{theorem:BPSalgisenvBorcherds}, and the fact that the perverse sheaves $\SG_{\dd}$ are Verdier self-dual, Proposition~\ref{proposition:discreterestriction} and Theorem~\ref{theorem:ICNQV}. In a little more detail, we infer that $\HO(i^*\BD\SG_{\dd})\cong \HO(\SG_{\dd})$ from the Verdier self-duality of $\SG_{\dd}$ and the $\BoC^*$-equivariance under the scaling action that contracts $\CM_{\Pi_Q,\dd}$ to $0_{\dd}$ (using \cite[Prop.3.7.5]{KSsheaves}).  Then we use isomorphisms $\HO(i^!_{\CN}\SG_{\dd})\cong \HO(\BD i^*\BD\SG_{\dd})\cong \BD\HO( i^*\BD\SG_{\dd})\cong \BD\HO(\SG_{\dd})$.  The one remaining instance of the Verdier duality functor is what lies behind the sign in the exponent of $t$ in part (3).
\end{proof}

\section{Nakajima quiver varieties and lowest weight modules}
\label{section:lowestweightmodules}
\subsection{Semistable modules}
Let $Q=(Q_1,Q_0,s,t\colon Q_1\rightarrow Q_0)$ be a quiver. A (King) \textit{stability condition} for $Q$ is a tuple $\zeta\in\BoQ^{Q_0}$. For $\dd\in\BoZ^{Q_0}\setminus\{0\}$ we define the \emph{slope}
\[
\mu_{\zeta}(\dd)=\zeta\cdotsh \dd/\lvert \dd\lvert
\]
and for $\rho$ a nonzero $\BoC Q$-module, we define $\mu_{\zeta}(\rho)=\mu_{\zeta}(\dim_{Q_0}(\rho))$. Then $\rho$ is called $\zeta$-\emph{semistable} if for all nonzero proper submodules $\rho'\subset \rho$ we have $\mu_{\zeta}(\rho')\leq \mu_{\zeta}(\rho)$, and is called $\zeta$-\emph{stable} if the inequality is strict. A semistable module is called \textit{polystable} if it is a direct sum of stable modules.

Fix a quiver $Q$, a potential $W\in\BoC Q/[\BoC Q,\BoC Q]$, a stability condition $\zeta\in\BoQ^{Q_0}$, and a slope $\theta\in \BoQ$ such that for all dimension vectors $\dd,\dd'$ of $\zeta$-slope $\theta$ we have $\chi_Q(\dd,\dd')=\chi_Q(\dd',\dd)$. This condition is called \emph{genericity}, and will be automatic for all the quivers that we encounter, since they will be symmetric. 

We denote by $\FM^{\zeta\sstab}_{Q,\dd}$ the moduli stack of $\dd$-dimensional $\zeta$-semistable $\BoC Q$-modules, and by $\CM^{\zeta\sstab}_{Q,\dd}$ the moduli scheme, constructed via GIT in \cite{king1994moduli}. Points of $\CM^{\zeta\sstab}_{Q,\dd}$ are in bijection with isomorphism classes of $\dd$-dimensional polystable $\BoC Q$-modules. We denote by $\JH^{\zeta}\colon \FM^{\zeta\sstab}_{Q,\dd}\rightarrow \CM^{\zeta\sstab}_{Q,\dd}$ the morphism taking a semistable module to the associated polystable module. The potential $W$ defines a function $\Tr(W)$ on $\FM^{\zeta\sstab}_{Q,\dd}$ defined as in \S \ref{subsection:DMreminder}, factoring through $\CM^{\zeta\sstab}_{Q,\dd}$.

The BPS sheaf is defined in \cite{davison2020cohomological} by setting
\[
\BPS_{Q,W,\dd}^{\zeta}\coloneqq (\ptau^{\leq 1}\!\JH^{\zeta}_*\phip{\Tr(W)}\BoQ_{\FM^{\zeta\sstab}_{Q,\dd}}^{\vir})[1].
\]
It is a perverse sheaf (not just a complex of perverse sheaves). It also admits a canonical upgrade to a monodromic mixed Hodge module, which we will ignore in this section.

Now assume that $Q$ is symmetric, so the stability condition $\zeta'=(0,\ldots,0)$ is generic. Since every $\BoC Q$-module is then semistable, we can identify $\FM^{\zeta'\sstab}_{Q,\dd}=\FM_{Q,\dd}$ and $\CM^{\zeta'\sstab}_{Q,\dd}=\CM_{Q,\dd}$. We denote by 
\[
q_{\zeta,\dd}\colon \CM^{\zeta\sstab}_{Q,\dd}\rightarrow\CM_{Q,\dd}
\]
the affinization/forgetful map, taking a polystable module to its semisimplification. Where there is no ambiguity we omit the subscript $Q$. The following is proved in \cite[Lemma 4.7]{toda2017gopakumar}.
\begin{proposition}
\label{TodaProp}
Let $Q$ be a symmetric quiver with potential $W\in \BoC Q/[\BoC Q,\BoC Q]$, let $\zeta\in\BoQ^{Q_0}$ be a stability condition. Then there is a natural isomorphism 
\[
(q_{\zeta,\dd})_*\BPS_{Q,W,\dd}^{\zeta}\cong \BPS_{Q,W,\dd}.
\]
\end{proposition}

\subsection{From Nakajima quiver varieties to BPS cohomology} \label{NQVs_def} Let $\ff\in\BoN^{Q_0}$ be a dimension vector. We form the framed quiver $Q_{\ff}$ by setting $(Q_{\ff})_0=Q_0\cup\{\infty\}$ and setting $(Q_{\ff})_1=Q_1\cup I$ where $I=\{\alpha_{i,n}\;\lvert\;i\in Q_0,1\leq n\leq f_i\}$. We set $s(\alpha_{i,n})=\infty$ and $t(\alpha_{i,n})=i$. Given a dimension vector $\dd\in \BoN^{Q_0}$ and a number $m\in\BoN$ we denote by $(\dd,m)\in\BoN^{(Q_{\ff})_0}=\BoN^{(\overline{Q_{\ff}})_0}$ the dimension vector extending $\dd$ by assigning $m$ to the vertex $\infty$. We say that a $(\dd,1)$-dimensional $\overline{Q_{\ff}}$-module is \textit{stable} if it does not contain a $(\dd',1)$-dimensional submodule with $\dd'\neq \dd$. This is the same as $\zeta$-stability, for $\zeta=(0,\ldots,0,1)$. We denote by $\BoA_{\overline{Q_{\ff}},(\dd,1)}^{\stab}\subset \BoA_{\overline{Q_{\ff}},(\dd,1)}$ the stable locus, and define $\mu_{Q_{\ff},\dd}^{-1}(0)^{\stab}\subset \mu_{Q_{\ff},\dd}^{-1}(0)$ similarly, where $\mu_{Q_{\ff},\dd}\coloneqq \mu_{Q_{\ff},(\dd,1)}$ is the moment map \eqref{mmap}. The gauge group $\GL_{\dd}$ acts scheme-theoretically freely on $\mu_{Q_{\ff},\dd}^{-1}(0)^{\stab}$ and we define the \textit{Nakajima quiver variety}
\begin{equation*}
\label{equation:smoothNQV}
N(Q,\dd,\ff)\coloneqq \mu_{Q_{\ff},\dd}^{-1}(0)^{\stab}/\GL_{\dd}.
\end{equation*}
We define the singular Nakajima quiver variety by
\begin{equation*}
\label{equation:singNQV}
N'(Q,\dd,\ff)=\mu_{Q_{\ff},\dd}^{-1}(0)\cms\GL_{\dd}=\Spec(\Gamma(\mu_{Q_{\ff},\dd}^{-1}(0))^{\GL_{\dd}})\subset \CM_{\overline{Q_{\ff}},(1,\dd)}.
\end{equation*}
For $\triangle\in\{\mathcal{N},\mathcal{SN},\SSN\}$, we denote by $N^{\triangle}(Q,\dd,\ff)\subset N(Q,\dd,\ff)$ the subvariety corresponding to representations of $\Pi_{Q_{\ff}}$ satisfying $\triangle$.

We set $\SQ=(\SQ_0,\SQ_1)\coloneqq\widetilde{Q_{\ff}}$, the tripled quiver of the framed quiver defined as in \S \ref{subsection:QandR}, and give this quiver the canonical cubic cubic potential 
\[
\tilde{W}=\left(\sum_{i\in (Q_{\ff})_0}\omega_i\right)\left(\sum_{a\in (Q_{\ff})_1}[a,a^*]\right).
\]
We fix the stability condition $(0,\ldots,0,1)\in\BoQ^{\SQ_0}$. In words, a $(\dd,1)$-dimensional $\SQ$-module is semistable, equivalently stable, if it is generated by the one-dimensional vector space assigned to the vertex $\infty$. There is a closed embedding of varieties
\[
l\colon N(Q,\dd,\ff)\times\BoA^1\hookrightarrow \CM_{\SQ,(\dd,1)}^{\zeta\sstab}
\]
which extends a $\overline{Q_{\ff}}$-module to a $\SQ$-module by assigning the operator of multiplication by $z\in\BoA^1$ to all the loops $\omega_i\in\SQ_1\setminus (\overline{Q_{\ff}})_1$.
\begin{proposition}[\cite{davison2020bps} Proposition 6.3]
\label{NakBPSProp}
Let $\SQ$, $\tilde{W}$ and $\zeta$ be as above. Then
\[
\BPS^{\zeta}_{\SQ,\tilde{W},(\dd,1)}\cong l_*\BoQ_{N(Q,\dd,\ff)\times\BoA^1}^{\vir}.
\]
\end{proposition}
\begin{proof}
The proof in [loc.cit.] concerns the quiver $Q^{+}$ obtained after removing the loop $\omega_{\infty}$ from $\SQ$, but adapts without modification to our setup.
\end{proof}
Combining Propositions \ref{TodaProp} and \ref{NakBPSProp} yields the following:
\begin{proposition}
\label{proposition:NQVtoBPS}
Let $\dd\in\BoN^{Q_0}$. There is a natural isomorphism 
\[
\Fn^{\ttBPS,\triangle,+}_{\Pi_{Q_{\ff}},(\dd,1)}\cong \HO^*\!(N^{\triangle}(Q,\dd,\ff),\imath_{\triangle}^!\BoQ^{\vir}_{N(Q,\dd,\ff)}),
\]
where $\imath_{\triangle}\colon N^{\triangle}(Q,\dd,\ff)\rightarrow N(Q,\dd,\ff)$ is the inclusion.
\end{proposition}
\begin{proof}
Set $\zeta=(0,\ldots,0,1)\in\BoQ^{\SQ_0}$. Recall that $h\colon \Msp_{\SQ} \to \Msp_{\overline{Q_\vec{f}}}$ is the forgetful map. We compose the isomorphisms
\begin{align*}
 \BPS_{\Pi_{Q_{\ff}},\Lie,(\dd,1)}\cong &h_*\BPS_{\SQ,\tilde{W},(\dd,1)}[-1]& \textrm{Proposition~\ref{proposition:BPS3dcoincide}}\\
\cong &h_*(q_{\SQ,\zeta,(\dd,1)})_{*}\BPS^{\zeta}_{\SQ,\tilde{W},(\dd,1)}[-1] &\textrm{Proposition~\ref{TodaProp}}\\
\cong &h_*(q_{\SQ,\zeta,(\dd,1)})_{*}l_*\BoQ_{N(Q,\dd,\ff)\times\BoA^1}^{\vir}[-1]& \textrm{Proposition~\ref{NakBPSProp}}\\
\cong &(q_{\overline{Q_{\ff}},\zeta,(\dd,1)})_{*}\BoQ_{N(Q,\dd,\ff)}^{\vir}
\end{align*}
and then the result follows, after base change along $\imath_{\triangle}$, by applying the derived global sections functor.
\end{proof}

Set 
\begin{equation}
\label{equation:moduleNQV}
\mathbb{M}^{\triangle}_{\ff}(Q)=\bigoplus_{\dd\in \BoN^{Q_0}}\HO^*(N^{\triangle}(Q,\dd,\ff),\imath^!_{\triangle}\BoQ^{\vir}_{N(Q,\dd,\ff)}).
\end{equation}
Then $\mathbb{M}^{\triangle}_{\ff}(Q)$ is already known to be a module for the algebra $\HO^*(\iota_{\triangle}^!\relCoHA_{\Pi_Q}^{\psi})$ and thus for the subalgebra $\UEA(\Fn^{\ttBPS,\triangle,+}_{\Pi_Q})\cong \aBPS^{\triangle,\psi}_{\Pi_Q,\Alg}$, see \cite{yang2018cohomological, davison2020bps} for details of the construction. We denote the resulting module $\mathbb{M}^{\triangle,\psi}_{\ff}(Q)$.

\subsection{Lowest weight modules}
Recall from \S\ref{subsection:lowestweightmodule} the definition of lowest weight modules over generalised Kac--Moody Lie algebras. We will use the following easy algebraic lemma.
\begin{lemma}
\label{gkm_res_lemma}
Let $M'$ be a monoid, satisfying our standing assumptions that the morphism $+$ has finite fibres, and the morphism $M'\rightarrow \widehat{M'}$ is an embedding in a lattice. Let $(-,-)$ be a bilinear form on the monoid $M=M'\times\BoN$. Let
\[
\begin{matrix}
 P&\colon& \Phi_{M,(-,-)}^+&\rightarrow& \BoN[t^{\pm 1/2}]\\
 &   & m   &\mapsto  & P_m(t^{1/2})
\end{matrix}
\]
be a weight function, with $\Fg_{M,(-,-),P}$ the associated generalised Kac--Moody algebra, so that for each $m\in \Phi_{M,(-,-)}^+$ we may define a cohomologically graded vector space $V_m\subset \Fg_{M,(-,-),P}$ of generators of weight $m$, with characteristic polynomial $\ch^{\BoZ}(V_m)=P_m(t^{1/2})$. We assume that $P(m)=1$ for $m\in \Phi^{+,\real}_{M,(-,-)}$. Restricting $(-,-)$ to $M'=M\times\{0\}$, and restricting $P$ to $\Phi_{M',(-,-)}^+$ we define also $\Fg\coloneqq \Fg_{M',(-,-),P}$. For $\ff\in \widehat{M'}^{\vee}=\Hom(\widehat{M'},\BoZ)$ we denote by $L_{\ff}$ the associated simple lowest weight $\Fg$-module (\S\ref{subsection:lowestweightmodule}). Then there is a direct sum decomposition of $\Fg$-modules
\begin{equation}
\label{mrt_iso}
\bigoplus_{m'\in M'}\Fg_{M,(-,-),P,(m',1)}\cong \bigoplus_{m'\in M'}V_{(m',1)}\otimes L_{\lambda_{m'}}
\end{equation}
where we define the linear map
\begin{align*}
\lambda_{m'}\colon M'&\rightarrow \BoZ\\
n'&\mapsto ((m',1),(n',0)).
\end{align*}
\end{lemma}
\begin{proof}
We denote the left hand side of \eqref{mrt_iso} by $L$. Then it is easy to see that $L$ is a lowest weight module for the action of $\Fg$, and the spaces $V_{(m',1)}$ for $(m',1)\in\Phi^+_{M',(-,-)}$ are lowest weight spaces of weight $\lambda_{m'}$ which generate $L$ a a $\Fg$-module. On the other hand \cite[Corollary 2.5]{borcherds1988generalized} $\Fg_{M,(-,-),P}$ carries an invariant bilinear form $\langle-,-\rangle$ such that $(-,-)_0\coloneqq \langle -,\omega(-)\rangle$ is positive definite on $\bigoplus_{m'\in M'}\Fg_{M,(-,-),P,(m',1)}$, where $\omega$ is the Cartan involution from \S \ref{CI_ssec}. It follows that $\bigoplus_{m'\in M'}\Fg_{M,(-,-),P,(m',1)}$ is semisimple as a $\Fg$-representation, and the result follows.
\end{proof}

For $\ee\in N=\Hom(\BoZ^{Q_0},\BoZ)$ we denote by $L_{\ee}^{\triangle}$ the lowest weight module for $\Fg_{\Pi_Q}^{\ttBPS,\triangle}$ of lowest weight $\ee$. 

The following is now our main theorem on the representation theory of $\Fg_{\Pi_Q}^{\ttBPS}$, and more generally $\Fg_{\Pi_Q}^{\ttBPS,\triangle}$.
\begin{theorem}
\label{theorem:mainRTthm}
The $\Fn_{\Pi_Q}^{\ttBPS,\triangle,+}$-action on $\mathbb{M}_{\ff}^{\triangle,\psi}(Q)$ extends to a $\Fg_{\Pi_Q}^{\ttBPS,\triangle}$-action, for which $\mathbb{M}^{\triangle,\psi}_{\ff}(Q)$ decomposes as a direct sum of cohomologically shifted simple lowest weight representations $L^{\triangle}_{\ee}$.  Precisely, we have
\begin{equation}
\label{equation:lwr}
\mathbb{M}^{\triangle,\psi}_{\ff}(Q)\cong \bigoplus_{(\dd,1)\in\Phi_{\Pi_{Q_{\ff}}}^+}\HO^*\!(\imath_{\triangle}^!\IC(\CM_{\Pi_{Q_{\ff}},(\dd,1)}))\otimes L_{((\dd,1),(-,0))_{Q_{\ff}}}^{\triangle}.
\end{equation}
\end{theorem}
We can rewrite the weight appearing in the theorem
\[
((\dd,1),(-,0))_{Q_{\ff}}\colon \dd'\mapsto (\dd,\dd')_Q-\ff\cdot \dd'.
\]
\begin{proof}
By Proposition~\ref{proposition:NQVtoBPS} and the remarks following it there is an isomorphism of $\Fn_{\Pi_{Q}}^{\ttBPS,\triangle,+}=\bigoplus_{0\neq m'\in \BoN^{Q_0}}\Fn_{\Pi_{Q_{\ff}},(m',0)}^{\ttBPS,\triangle,+}$-modules
\[
\mathbb{M}^{\triangle,\psi}_{\ff}(Q)\cong \bigoplus_{m'\in \BoN^{Q_0}}\Fn_{\Pi_{Q_{\ff}},(m',1)}^{\ttBPS,\triangle,+},
\]
and so the action of $\Fn_{\Pi_{Q}}^{\ttBPS,\triangle,+}$ extends to a $\Fg_{\Pi_{Q}}^{\ttBPS,\triangle}$-action via the identification
\[
\Fg_{\Pi_{Q}}^{\ttBPS,\triangle}=\bigoplus_{\dd\in\BoZ^{Q_0}}\Fg_{\Pi_{Q_{\ff}},(\dd,0)}^{\ttBPS,\triangle}
\]
and the action of $\Fg_{\Pi_{Q_{\ff}}}^{\ttBPS,\triangle}$ on itself. The theorem then follows from Lemma~\ref{gkm_res_lemma} and the characterisation of simple roots of $\Fg_{\Pi_{Q_{\ff}}}^{\ttBPS,\triangle}$ given in Theorem~\ref{corollary:absBor}.
\end{proof}

\begin{corollary}
Let $\ff\in\BoN^{Q_0}$ have full support. The action of $\Fg_{\Pi_Q}^{\ttBPS,\triangle}$ on $\mathbb{M}^{\triangle}_{\ff}(Q)$ is faithful.
\end{corollary}
\begin{proof}
By Theorem \ref{theorem:mainRTthm} there is a summand $L_{((0,1),(-,0))_{Q_{\ff}}}^{\triangle}\subset \mathbb{M}^{\triangle,\psi}_{\ff}(Q)$, and in particular, a lowest weight vector $e_{\infty}\in \mathbb{M}^{\triangle,\psi}_{\ff}(Q)$ of degree $(0,\ldots,0,1)\in\BoN^{(Q_{\Pi_{\ff}})_0}$, which we identify as one of the Chevalley generators of $\Fg^{\ttBPS,\triangle}_{\Pi_{Q_{\ff}}}$, of degree $(0,\ldots,0,1)$, forming part of the $\mathfrak{sl}_2$-triple $e_{\infty},f_{\infty},h_{\infty}$. Now let $v\in \Fn_{\Pi_Q,\dd}^{\ttBPS,\triangle,+}$. Then $[h_{\infty},v]=(-\ff\cdot\dd) v\neq 0$ (by our assumption on $\ff$), so by the representation theory of $\mathfrak{sl}_2$ we have $[v,e_{\infty}]\neq 0$. We deduce that the $\Fn_{\Pi_Q}^{\ttBPS,\triangle,+}$-action and the $\Fh_{\Pi_Q}^{\ttBPS,\triangle}$-action are faithful.

With $(-,-)_0$ the positive definite form from the proof of Lemma \ref{gkm_res_lemma}, we have $(e_{\infty},[w(v),[v,e_{\infty}]])_0=([v,e_{\infty}],[v,e_{\infty}])_0\neq 0$ and so the $\Fn_{\Pi_Q}^{\ttBPS,\triangle,-}$-action is also faithful. The $\Fn_{\Pi_Q}^{\ttBPS,\triangle,+}$ action acts by raising operators with respect to the decomposition \eqref{equation:moduleNQV} into factors indexed by $\BoN^{Q_0}$, while the $\Fn_{\Pi_Q}^{\ttBPS,\triangle,-}$ action acts by lowering operators, and the $\Fh_{\Pi_Q}^{\ttBPS,\triangle}$ action preserves the decomposition.  So faithfulness of the $\Fg_{\Pi_Q}^{\ttBPS,\triangle}$ action follows from the faithfulness of the three summands in its triangular decomposition.
\end{proof}
\subsection{Actions on cohomology of quiver varieties, revisited}
\label{subsec:GN}
\subsubsection{Hilbert schemes of $\BoA^2$ and Heisenberg algebras}
First we recover the famous description of the cohomology of Hilbert schemes of $\BoA^2$ as an irreducible lowest weight module for an infinite-dimensional Heisenberg algebra, as in \cite{nakajima1999lectures,nakajima1997heisenberg}. 

Let $Q$ be the Jordan quiver, with one vertex and one loop, and let $\ff=1$, so that $\overline{Q_{\ff}}$ is the ADHM quiver (with the help of which the construction of instantons in \cite{atiyah1994construction} can be formulated):
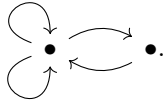
\[
\begin{tikzcd}
\bullet \arrow[r,bend left]\arrow[out=160,in=90,loop]\arrow[out=200,in=270,loop] & \bullet\arrow[l,bend left,swap].
\end{tikzcd}
\]
It is easy to check that in this case $(0,1)$ is the only element of $\Phi^+_{\Pi_{Q_{\ff}}}$ of the form $(n,1)$ (using the characterisation given in Proposition \ref{proposition:rootsneighbourhood} and the definition of the set of roots). In addition, there is a natural isomorphism \cite{nakajima1999lectures}
\[
N(Q,n,1)\cong \Hilb_n(\BoA^2).
\]
By definition $\Fn^{\ttBPS,+}_{\Pi_Q,n}\cong \HO^*\IC(\BoA^2)\cong \BoQ[2]$, and $\Fn^{\ttBPS,+}_{\Pi_Q}$ is the cohomologically graded generalised Kac--Moody Lie algebra determined by the weight function 
\begin{align*}
P\colon \Phi^+_{\Pi_Q}\cong\BoZ_{\geq 1}\rightarrow& \BoN[t^{\pm 1/2}]&\\n\mapsto& t^{-1}.&
\end{align*}
Precisely, $\Fg_{\Pi_Q}^{\aBPS}$ is the generalised Kac--Moody Lie algebra spanned by elements $e_n,f_n,h_n$ of respective cohomological degrees $-2,2,0$ and $\BoZ^{Q_0}$-degree $n,0,-n$ respectively, i.e. the infinite-dimensional Heisenberg algebra. Since all but one of the summands on the right hand side of \eqref{equation:lwr} vanish, Theorem~\ref{theorem:mainRTthm} yields
\[
\bigoplus_{n\geq 0}\HO^*\!(\Hilb_n(\BoA^2),\BoQ^{\vir})\cong L_{1}
\]
where $L_1=\BoQ[e_1,e_2,\ldots]$ is the irreducible lowest weight module for $\Fg_{\Pi_Q}^{\ttBPS}$ of weight $1$.
\subsubsection{Irreducible lowest weight $\Fg_{\Pi_Q}^{\ttBPS,\SSN}$-modules from Nakajima quiver varieties}
In this section we recover Nakajima's original construction \cite{nakajima1998quiver} of lowest weight representations of Kac--Moody Lie algebras by restricting to the zeroth cohomological degree in our main theorem (Theorem~\ref{theorem:mainRTthm}) on the representation theory of $\Fg^{\ttBPS,\triangle}_{\Pi_Q}$ for the choice $\triangle=\mathcal{SSN}$.

Let $Q$ be a quiver. Let $\ff\in \BoN^{Q_0}$ be a dimension vector. Consider the CoHA $\HO^0(\relCoHA_{\Pi_Q}^{\SSN,\psi})$ obtained by taking the zeroth cohomologically graded piece of the CoHA $\HO^*\!\relCoHA_{\Pi_Q}^{\SSN,\psi}$. Recall that by Theorem~\ref{theorem:degree0SSN} this is isomorphic to the enveloping algebra of half of the Borcherds--Bozec algebra $\Fg_Q$. By Theorem~\ref{corollary:absBor}, if $\dd\in \Sigma_{\Pi_{Q_{\ff}}}$ there is an identification between
\[
\HO^0(\iota_{\SSN}^!\IC(\CM_{\Pi_{Q_{\ff}},\dd}))
\]
and a basis of Chevalley generators of weight $\dd$ inside $\Fg_{Q_{\ff}}$: in particular, the above cohomology group vanishes if $\dd$ is supported at more than one vertex. Applying $\HO^0(-)$ to the isomorphism
\[
\mathbb{M}^{\SSN}_{\ff}(Q)\cong \bigoplus_{(\dd,1)\in\Phi_{Q_{\ff}}^+}\HO^*\!(\iota_{\SSN}^!\IC(\CM_{\Pi_{Q_{\ff}},(\dd,1)}))\otimes L_{((\dd,1),(-,0))_{Q_{\ff}}}
\]
we thus find that every summand apart from the one for $(\dd,1)=(0,\ldots,0,1)\in\Phi^+_{\Pi_{Q_{\ff}}}$ vanishes, so that we obtain the following theorem
\begin{theorem}
There is an isomorphism of $\UEA(\Fg_{Q})=\HO^0(\SA_{\Pi_Q}^{\SSN,\psi})$-modules
\[
\bigoplus_{\dd\in \BoN^{Q_0}}\HO^0(N^{\SSN}(Q,\dd,\ff),\iota_{\SSN}^!\BoQ^{\vir})\cong L_{-\ff\cdot}.
\]
\end{theorem}

Here, we denote by $L_{-\ff\cdot }$ the simple lowest weight module of $\Fg_Q$ of lowest weight $\dd\mapsto -\ff\cdot\dd$, the usual dot product as in \S \ref{subsection:QandR}.  In the special case in which $Q$ has no loops, this is Nakajima's original theorem \cite[Thm.10.2]{nakajima1998quiver}.  In the case in which there are loops, this is a CoHA analogue of results in \cite{bozec2016quivers} on highest weight crystals for the Borcherds--Bozec algebra.

\section{Equivariant parameters}
\label{section:Equivariantparameters}

\subsection{Deformed BPS algebra}
\label{subsection:deformed}
In this section, we explain how to adapt the main results of this paper when we turn on equivariant parameters. We let $\CA$ and $\JH\colon\FM_{\CA}\rightarrow\CM_{\CA}$ be as in \S\ref{section:modulistack2dcats}. We let $G$ be a connected algebraic group and we assume that $G$ acts on $\JH$ and that this action is compatible with the direct sum morphism: if $G$ acts diagonally on $\FM_{\CA}\times\FM_{\CA}$, the direct sum morphism $\oplus\colon\FM_{\CA}\times\FM_{\CA}\rightarrow \FM_{\CA}$ is $G$-equivariant. To avoid discussing group actions on stacks \cite{romagny2005group}, we assume that each connected component of $\FM_{\CA}$ is a global quotient stack $X/H$ of an algebraic variety $X$ by an algebraic group $H$, and $G$ acts on $X$ giving a $G\times H$-action on $X$ (i.e. the $G$ and $H$ actions commute with each other).

We indicate by the $G$-superscript the objects and maps induced when taking the stack-theoretic quotient by $G$. For example, we let $\FM_{\CA}^G\coloneqq \FM_{\CA}/G$ (for a given connected component $X/H$ of $\FM_{\CA}$, there is a connected component $X/(H\times G)$ of $\FM_{\CA}^G$), $\CM_{\CA}^G\coloneqq \CM_{\CA}/G$, $\JH^G\colon \FM_{\CA}^G\rightarrow\CM_{\CA}^G$, $\oplus^G\colon \CM_{\CA}^G\times_{\mathrm{B}G}\CM_{\CA}^G\rightarrow\CM_{\CA}^G$. Note that the source of the map $\oplus^G$ can be identified with the stack quotient $(\CM_{\CA}\times\CM_{\CA})/G$ by the diagonal action of $G$.

We have the (smooth) quotient map $\pi\colon \CM_{\CA}\rightarrow\CM_{\CA}^G$, and the commutative square
\begin{equation}
\label{equation:dsequivariant}
 \begin{tikzcd}
	{\CM_{\CA}\times\CM_{\CA}} & {\CM_{\CA}} \\
	{\CM_{\CA}^{G}\times_{\mathrm{B}G}\CM_{\CA}^G} & {\CM_{\CA}^G}
	\arrow["\pi\times\pi"',from=1-1, to=2-1]
	\arrow["\oplus^G",from=2-1, to=2-2]
	\arrow["\oplus",from=1-1, to=1-2]
	\arrow["\pi",from=1-2, to=2-2]
	\arrow["\lrcorner"{anchor=center, pos=0.125}, draw=none, from=1-1, to=2-2]
\end{tikzcd}
\end{equation}
is Cartesian.

The categories $\Perv(\CM_{\CA}^G)[\dim G]$, $\MHM(\CM_{\CA}^G)[\dim G]$, $\CD^+_{\rmc}(\CM_{\CA}^G)$ and $\CD^+(\MHM(\CM_{\CA}^G))$ all carry a symmetric monoidal tensor product defined by $\SF\boxdot^G\SG\coloneqq \oplus_*^G(\SF\boxtimes\SG)$.

By smooth base change in the diagram \eqref{equation:dsequivariant}, $\pi^{*}\colon \CD^+(\MHM(\CM_{\CA}^G))\rightarrow \CD^+(\MHM(\CM_{\CA}))$ is a symmetric monoidal functor, mapping objects of $\MHM(\CM_{\CA}^G)[\dim G]$ to objects of $\MHM(\CM_{\CA})$. We have a similar statement for the constructible derived categories and the categories of perverse sheaves.

We define $\underline{\BoQ}_{\FM^G_{\CA}}^{\vir}$ by setting its restriction to the connected component $\FM^G_{\CA,a}$, to be $\underline{\BoQ}_{\FM^G_{\CA,a}}\otimes \BoL^{(a,a)_{\SC}/2}$, as in the non-equivariant case \S \ref{subsection:theCoHA}. We gather in the following proposition the properties of the $G$-equivariant versions of the main objects of this paper. Parts (3) and (4) follow from (2), and the analogous results for $G=\{1\}$.

\begin{proposition}
 \begin{enumerate}
 \item The relative CoHA has the structure of an algebra object in $\CD^+(\MHM(\CM_{\CA}^G))$ defined as in \cite{davison2022BPS} (working over $\mathrm{B}G$). We denote by $\mult^G$ the multiplication map (which, in this general context, we assume to be associative, see Assumption~\ref{ass:associativity}).
  \item We have canonical identifications of algebra objects $\pi^{*}\ulrelCoHA_{\CA}^G=\ulrelCoHA_{\CA}$, $\pi^{*}\uleqrelBPSalg{\CA}{G}=\ulBPS_{\CA,\Alg}$.
 \item The mixed Hodge module complex $\ulrelCoHA_{\CA}^G\coloneqq (\JH_{\CA}^G)_*\underline{\BoQ}_{\FM^G_{\CA}}^{\vir}$ is pure of weight zero, concentrated in cohomological degrees $\geq -\dim G$.
 \item The degree $-\dim G$ cohomology $\ulBPS_{\CA,\Alg}^{G}\coloneqq \CH^{-\dim G}(\ulrelCoHA_{\CA}^G)[\dim G]$ has an induced algebra structure. It is an algebra object in $\MHM(\CM_{\CA}^G)[\dim G]$.
 \end{enumerate}
\end{proposition}

\begin{remark}
 The proof of associativity of the CoHA product at the sheaf level is proven in \cite[Appendices]{davison2022BPS} for categories of geometric (coherent sheaves on surfaces) and algebraic (representations of algebras of homological dimension $2$).
\end{remark}

\subsection{Deformed relative BPS Lie algebra}
\label{subsection:deformedrelative}
For $l\in\BoN$ and $m'\in\K(\CA)$, we let
\[
 u_m^G\colon \CM_{\CA,m'}^G\rightarrow\CM_{\CA,lm'}^G
\]
be the map obtained by composing the diagonal morphism $\CM_{\CA,m'}^{G}\rightarrow\CM_{\CA,m'}^{G}\times_{\mathrm{B}G}\hdots\times_{\mathrm{B}G}\CM_{\CA,m'}^{G}$ with the (iterated) direct sum morphism $\CM_{\CA,m'}^{G}\times_{\mathrm{B}G}\hdots\times_{\mathrm{B}G}\CM_{\CA,m'}^{G}\rightarrow\CM_{\CA,lm'}^{G}$.
We let
\[
 \ul{\SG}^G_{\CA,m}\coloneqq
 \left\{
 \begin{aligned}
 &\ul{\IC}(\CM_{\CA,m}^G)[\dim G]\quad&\text{ if $m\in\Sigma_{\K(\CA)}$}\\
 &(u_m^G)_*\ul{\IC}(\CM_{\CA,m'}^G)[\dim G]\quad&\text{ if $m=lm'\in\Phi^{\iso}_+$ with $m'\in\Sigma_{\K(\CA)}$ and $l\geq 2$}\\
 &0\quad&\text{ otherwise.}
 \end{aligned}
 \right.
\]

The construction of GKM Lie algebras and their enveloping algebras in the symmetric monoidal category of mixed Hodge modules on a monoid in the category of schemes as in \S\ref{subsection:GKMtensorcategories} adapts in a straightforward way to the situation where $\CM$ is replaced by a stack quotient of $\CM$ by an algebraic group $G$ and one works relatively over $\mathrm{B}G$. We let $\CM^G\coloneqq \CM/G$ and the direct sum map is a (finite) morphism $\oplus\colon \CM^G\times_{\mathrm{B}G}\CM^G\rightarrow\CM^G$. The free algebra is defined accordingly and one can make sense of Serre relations in the same way as for schemes, and we have the following:

\begin{theorem}
\label{theorem:relBPSequiv}
 We have a canonical isomorphism of algebra objects in $\MHM(\CM_{\CA}^G)$
 \[
 \UEA(\Fn^+_{\K(\CA),(-,-)_{\SC},\ul{\SG}_{\CA}^G})\cong \uleqrelBPSalg{\CA}{G}.
 \]
\end{theorem}
\begin{proof}
 The forgetful functor $\pi^{*}\colon \MHM(\CM_{\CA}^G)[\dim G]\rightarrow\MHM(\CM_{\CA})$ is monoidal and conservative. Therefore, Theorem \ref{theorem:relBPSequiv} follows from Theorem \ref{theorem:BPSalgisenvBorcherds}.
\end{proof}

We can therefore define (as in the case $G=\{1\}$) the relative BPS Lie algebra as
\[
 \ulBPS_{\CA,\Lie}^{G}\coloneqq \mathfrak{n}^+_{\K(\CA),(-,-)_{\SC},\ul{\SG}_{\CA}^G}.
\]
As in \S\ref{subsubsection:fullPBW}, we define a PBW map
\begin{equation}
\label{equation:PBWmap}
 \Psi_{\CA}^{\psi,G}\colon\Sym_{\boxdot^G}\left(\ulBPS_{\CA,\Lie}^{G}\otimes\HO^*_{\BoC^*}\right)\rightarrow \ulrelCoHA_{\CA}^{G}.
\end{equation}

\begin{theorem}[PBW isomorphism]
\label{theorem:PBWequiv}
 The map $\Psi_{\CA}^{\psi,G}$ is an isomorphism of mixed Hodge module complexes.
\end{theorem}
\begin{proof}
 The forgetful functor $\pi^*\colon \CD^+(\MHM(\CM_{\CA}^G))\rightarrow \CD^+(\MHM(\CM_{\CA}))$ is conservative. Therefore, Theorem \ref{theorem:PBWequiv} follows from Theorem \ref{theorem:relPBW}.
\end{proof}

\subsection{Deformed absolute BPS Lie algebra}
One of the key facts used to deduce results for the absolute CoHA from results for the relative CoHA (for example the GKM description and PBW isomorphism) is the compatibility of the derived global sections functor
\[
 \CD^+(\MHM(\CM_{\CA}))\rightarrow \CD^+(\MHM(\K(\CA)))
\]
with the respective monoidal tensor products. In the equivariant setting, this compatibility no longer holds in general. The derived global section functor 
\[
 \CD^+(\MHM(\CM_{\CA}^G))\rightarrow \CD^+(\MHM(\K(\CA)/G))
\]
(where $G$ acts trivially on $\K(\CA)$ and $\K(\CA)/G$ is a quotient stack) is not monoidal for the tensor structures induced by the direct sum (as in \S\ref{subsection:deformed}), that is for a given complex variety $X$ and $G$-equivariant mixed Hodge modules $\ul{\SF},\ul{\SG}$ on $X$, the natural map
\begin{equation}
\label{equation:kunnethmap}
 \HO^*(X/G,\ul{\SF})\otimes_{\HO^*(\mathrm{B}G,\BoQ)} \HO^*(X/G,\ul{\SG})\rightarrow \HO^*(X/G\times_{\mathrm{B}G} X/G,\ul{\SF}\boxtimes\ul{\SG})
\end{equation}
may not be an isomorphism. Nevertheless, it is an isomorphism under certain hypotheses on $\ul{\SF},\ul{\SG}$ and $G$.

\begin{lemma}
\label{degen_pure}
 Let $X$ be an algebraic variety acted on by the algebraic group $G$. We assume that $\HO^*(\B G,\BoQ)$ is pure. We let $r\colon X/G\rightarrow \mathrm{B}G$ be the natural map and $\pi\colon X\rightarrow X/G$ be the quotient map. We let $\ul{\SF},\ul{\SG}$ be mixed Hodge modules on $X/G$ such that $r_*\ul{\SF}$ and $r_*\ul{\SG}$ are pure mixed Hodge modules on $\mathrm{B}G$. Then, the K\"unneth map \eqref{equation:kunnethmap} is an isomorphism. More precisely, we have
 \[
 \HO^*(X/G,\ul{\SF})=\HO^*(X,\pi^*\ul{\SF})\otimes\HO^*(\mathrm{B}G,\BoQ)
 \]
 and similarly for $\ul{\SG}$.
\end{lemma}
\begin{proof}
The Leray--Serre spectral sequence for the fibration $r$ computing $\HO^*(X/G,\ul{\SF})$ degenerates at the second page $E^2_{p,q}=\HO^p(\B G,\CH^q(r_*\ul{\SF}))=\HO^{p+\dim(G)}(\B G,\BoQ)\otimes\HO^{q-\dim(G)}(X,\pi^*\ul{\SF})$ for purity reasons.
\end{proof}

This proof shows more generally that the K\"unneth map is an isomorphism when both $\HO^*(X/G,\SF)$ and $\HO^*(X/G,\SG)$ are free over $\HO^*(\mathrm{B}G,\BoQ)$. In the sequel, we assume that the mixed Hodge modules $\ul{\SG}_{\CA,m}^G$, $m\in\Phi^+_{\K(\CA)}$ defined in \S\ref{subsection:deformedrelative} are such that the K\"unneth map is an isomorphism for any pair $\ul{\SG}_{\CA,m}^{G},\ul{\SG}_{\CA,n}^G$. We say that the mixed Hodge modules $\ul{\SG}^G_{\CA,m}$, $m\in\Phi^+_{\K(\CA)}$ satisfy the \emph{K\"unneth property}. Then it follows formally that there is an  isomorphism of Lie algebras
\begin{equation}
\label{EoScompare}
\HO^*(\Fn^+_{\CM_{\CA}^G,(-,-)_{\SC},\ul{\SG}_{\CA}^G})\cong \Fn^+_{\K(\CA),(-,-)_{\SC},\ul{\SG}_{\CA}}\otimes \HO^*(\B G,\BoQ).
\end{equation}

We immediately deduce the following.

\begin{theorem}
\label{theorem:absoluteCoHA}
Assume that the mixed Hodge modules $\ul{\SG}^G_{\CA,m}$, $m\in\Phi^+_{\K(\CA)}$ satisfy the K\"unneth property. The deformed BPS algebra $\rmBPS_{\CA,\Alg}^{G}$ is isomorphic to the enveloping algebra of the generalised Kac--Moody algebra defined over $\HO^*(\mathrm{B}G,\BoQ)$, associated to the monoid with bilinear form $(\K(\CA),(-,-)_{\SC})$ and the weight function described in Theorem \ref{corollary:absBor} (with $\iota_{\triangle}$ the identity morphism here -- nilpotent versions are considered in Proposition \ref{proposition:equivariantnilpotent}). In particular, there is an isomorphism of algebras
 \[
 \rmBPS_{\CA,\Alg}^{G}\cong \rmBPS_{\CA,\Alg}\otimes \HO^*(\B G,\BoQ).
 \]
 
\end{theorem}
\begin{proof}
 The fact that the mixed Hodge modules $\ul{\SG}^G_{\CA,m}$ satisfy the K\"unneth property implies that the $\HO^*(\B G,\BoQ)$-algebra obtained by taking derived global sections of $\BPS_{\CA,\Alg}^{G}$ is exactly the $\HO^*(\B G,\BoQ)$-linear enveloping algebra of the GKM Lie algebra described in the statement of the theorem. Theorem \ref{theorem:absoluteCoHA} is then a consequence of Theorem \ref{theorem:relBPSequiv}.
\end{proof}

Given a choice of positive determinant bundle on $\FM_{\CA}^G$ we define a map
\begin{equation}
 \label{equation:absPBWequiv}
 \Sym_{\HO^*(\B G,\BoQ)}\left(\ul{\rmBPS}_{\CA,\Lie}^{G}\otimes\HO^*_{\BoC^*}\right)\rightarrow \HO^*(\ulrelCoHA_{\CA}^{\psi,G})
\end{equation}
as in \S \ref{subsubsection:fullPBW} by using the algebra structure on $\HO^*(\ulrelCoHA_{\CA}^{\psi,G})$, the $\HO^*_{\BoC^*}$-action on the CoHA and the iterated CoHA multiplication.

\begin{theorem}[Deformed PBW theorem]
\label{theorem:absoluteequivPBW}
 The map \eqref{equation:absPBWequiv} is an isomorphism of mixed Hodge module complexes.
\end{theorem}
\begin{proof}
 The K\"unneth property implies that the map \eqref{equation:absPBWequiv} coincides with the map $\HO^*(\Psi_{\CA}^{\psi,G})$. Therefore, Theorem \ref{theorem:absoluteequivPBW} follows from Theorem \ref{theorem:PBWequiv}.
\end{proof}

\subsection{Examples}
\label{subsection_example_tori}
When $\CA$ is the category of representations of the preprojective algebra $\Pi_Q$ of a quiver, there is a natural choice for the group $G$. Namely, we let
\[
 G_{\arr}\coloneqq \prod_{i\in Q_0}\mathrm{Sp}_{2g_i}\times\prod_{i\neq j\in Q_0}\GL_{q_{i,j}},
\]
be the group of arrow symmetries of $\overline{Q}$ preserving the moment map, where $g_i$ is the number of loops at the vertex $i$ in $Q_1$ and $q_{i,j}$ the number of arrows between $i$ and $j$ in $\overline{Q}$ \cite[\S2.1.3]{maulik2019quantum}. The group $G_{\arr}$ acts naturally on the representation spaces $\BoA_{\overline{Q},\dd}$ defined in \S \ref{subsection:QandR}. Indeed, we can write
\[
 \BoA_{\overline{Q},\dd}=\prod_{i\in Q_0}\End(\BoC^{d_i})\otimes \BoC^{2g_i}\times\prod_{i\neq j\in Q_0}\Tan^*(\Hom(\BoC^{d_i},\BoC^{d_j})\otimes \BoC^{q_{i,j}}),
\]
where $q_{i,j}$ is the number of arrows $i\rightarrow j$ in $Q$, $\mathrm{Sp}_{2g_i}$ acts on $\BoC^{2g_i}$ and $\GL_{q_{i,j}}$ acts on $\BoC^{q_{i,j}}$.

We can take for $G$ any (connected) subgroup of $G_{\arr}\times\BoC^*$. The additional $\BoC^*$ action rescales cotangent fibers (i.e. if $(x,x^*)\in\BoA_{\overline{Q},\dd}$ and $u\in\BoC^*$, $u\cdot (x,x^*)=(x,ux^*)$. For example, $G$ can be a maximal torus of $G_{\arr}$.

As shown in the following lemma, equivariant formality holds for the mixed Hodge modules involved in the description of the BPS algebra of $\Pi_Q$.

\begin{lemma}
\label{lemma:formalityquivervarieties}
 Let $Q$ be a quiver, $\Pi_Q$ its preprojective algebra and $\dd\in\BoN^{Q_0}$ a dimension vector. Then, $\mathrm{IH}^n(\CM_{\Pi_Q,\dd})$ vanishes for $n$ odd. Therefore, for any connected linear algebraic group $G$ acting on $\CM_{\Pi_Q,\dd}$, the intersection cohomology complex is naturally equivariant and $\ICA(\CM_{\Pi_Q,\dd}/G)=\mathrm{IH}^*_G(\CM_{\Pi_Q,\dd})\cong\ICA(\CM_{\Pi_Q,\dd})\otimes_{\BoQ}\HO^*(\mathrm{B}G,\BoQ)$.
\end{lemma}
\begin{proof}
The last statement is known as ``equivariant formality" and follows from the purity of the intersection cohomology. Purity, as well as vanishing of odd intersection cohomology, follow from Theorem \ref{theorem:BPSalgisenvBorcherds} by taking derived global sections, and the analogous purity and vanishing statements for $\HO^{\BoMo}(\FM_{\Pi_Q},\BoQ^{\vir})$: this is \cite[Theorem A]{davison2016integrality}.
\end{proof}

\subsection{3d deformed BPS Lie algebra}
One can define the deformed 3d BPS Lie algebra using the tripled quiver with potential and dimensional reduction. It is defined as in \S\ref{subsection:DMreminder} after upgrading objects to mixed Hodge modules as in \cite{davison2020cohomological}. The group $G_{\arr}\times\BoC^*$ acts on the representation space $\BoA_{\tilde{Q},\dd}\cong\BoA_{\overline{Q},\dd}\times\prod_{i\in Q_0}\End(\BoC^{d_i})$ of the tripled quiver as follows. The action on the first factor is as defined in \S \ref{subsection_example_tori}. The subgroup $G_{\arr}$ acts trivially on the second factor and the additional copy of $\BoC^*$ acts by rescaling the factor $\prod_{i\in Q_0}\End(\BoC^{d_i})$ by $u^{-1}$. For this action, the canonical potential $\tilde{W}$ gives a $G_{\arr}\times\BoC^*$-invariant function on $\BoA_{\tilde{Q},\dd}$. The vanishing cycle cohomology gives the cohomological Hall algebra $\ulrelCoHA_{\tilde{Q},\tilde{W}}^{\psi,G}$, and the perverse filtration together with dimensional reduction provides us with the 3d BPS Lie algebra $\ulBPS_{\Pi_Q,\Lie}^{\td,\psi,G}$. It is an object of the shifted category of mixed Hodge modules $\MHM(\CM_{\Pi_Q}^G)[\dim G]$. All results presented in \S\ref{section:comparison} hold for the $G$-equivariant deformations of the relative and absolute versions of the BPS algebra and Lie algebra:

\begin{proposition}
\label{Gequiv2d3d}
 We have a canonical isomorphism of Lie algebra objects in mixed Hodge modules
 \[
 \ulBPS_{\Pi_Q,\Lie}^{\td,\psi,G}\cong\ulBPS_{\Pi_Q,\Lie}^{G}.
 \]
\end{proposition}
\begin{proof}
 Again, by conservativity of the forgetful functor, it suffices to check this statement when $G$ is the trivial group and in this case, this is Theorem \ref{thm:comp}.
\end{proof}

As in the non-equivariant setting, the 3d BPS Lie algebra $\Fg_{\tilde{Q},\tilde{W}}^{\psi,G}$ is obtained by passing to derived global sections of $\BPS_{\Pi_Q,\Lie}^{\td,G}$. The following is our main result on deformations of BPS Lie algebras for tripled quivers with potential: it says that all deformations obtained by switching on equivariant parameters associated to arrows of the quiver are \emph{trivial}.
\begin{corollary}
\label{triv_ext_BPS}
There is an isomorphism of Lie algebras 
\[
\Fg_{\tilde{Q},\tilde{W}}^{\psi,G}\cong \Fg_{\tilde{Q},\tilde{W}}^{\psi}\otimes\HO^*(\B G,\BoQ).
\]
\end{corollary}
\begin{proof} 
We have $\Fg_{\tilde{Q},\tilde{W}}^{\psi} \cong \Fn_{\Pi_Q}^{\psi}$ by Theorem~\ref{thm:comp} and $\Fg_{\tilde{Q},\tilde{W}}^{\psi,G} \cong \Fn_{\Pi_Q}^{\psi,G}$ by Proposition~\ref{Gequiv2d3d}. The isomorphism then follows from \eqref{EoScompare}.

\end{proof}

\subsection{Equivariant cohomology of Nakajima quiver varieties}
\label{subsection:equivariantcohNQV}
Our description of the whole cohomology of Nakajima quiver varieties (\S\ref{section:lowestweightmodules}) adapts to the equivariant setting. As in \cite[\S2.1.2]{maulik2019quantum}, the group $G_{\arr}\times\BoC^*$ acts on quiver varieties and since they have pure cohomology, equivariant formality holds for them for any action of a connected linear algebraic group $G\subset G_{\arr}\times\BoC^*$ (as in Lemma \ref{lemma:formalityquivervarieties}). The description of the $G$-equivariant cohomology of quiver varieties as a direct sum of lowest weight modules (Theorem \ref{theorem:mainRTthm}) adapts to the $G$-equivariant cohomology in a straightforward way: this eventually amounts to extending the scalars from $\BoQ$ to $\HO^*(\mathrm{B}G,\BoQ)$.

\subsection{Nilpotent versions}
\label{subsection:nilequivversions}
If $\CM_{\CA}^{\triangle}\subset\CM_{\CA}$ is a saturated submonoid that is stable under the $G$-action on $\CM_{\CA}$, then one can define the corresponding relative and absolute CoHAs $\ulrelCoHA_{\CA}^{\triangle,\psi,G}$ and $\HO^*(\ulrelCoHA_{\CA}^{\triangle,\psi,G})$, and the $\triangle$ BPS algebra and $\triangle$ Lie algebra $\ul{\mathrm{BPS}}_{\CA,\Alg}^{\triangle,\psi,G}$ and $\ul{\mathrm{BPS}}_{\CA,\Lie}^{\triangle,G}$ via $\iota_{\triangle}^!$, as before. In particular, this applies to the case of the preprojective algebras of quivers and $\triangle\in\{\CN,\SemiN,\SSN\}$. The results of \S\ref{section:lowestweightmodules} also translate to the equivariant BPS algebras with various nilpotency conditions. Crucial in the step from relative to absolute CoHAs is again the equivariant formality of the sheaf cohomologies involved. For the nilpotent versions, this equivariant formality follows again from the odd cohomology vanishing.
\begin{proposition}
 The stacks $\mathfrak{M}_{\Pi_Q}^{\triangle}$ for $\triangle\in\{\CN,\SemiN,\SSN\}$ are pure and have no odd cohomology. $\triangle$ Nakajima quiver varieties are pure and have no odd cohomology for $\triangle\in\{\CN,\SemiN,\SSN\}$.
\end{proposition}
\begin{proof}
The first statement is \cite[Theorem A]{schiffmann2020cohomological}, for $\triangle \in\{\SemiN,\SSN\}$. For $\triangle=\CN$ it holds for the BPS Lie algebra since it is Verdier dual to the BPS Lie algebra $\Fg_{\Pi_Q}$, and then cohomological integrality implies it for the cohomology of the whole stack.

The second statement comes from the fact that all nilpotent versions of Nakajima quiver varieties are very $\ell$-pure and polynomial count (when working over a finite field $\BoF_q$ and $\ell$-adic cohomology, for $\ell$ and $q$ coprime), \cite[Theorem 1]{hausel2010kac}. Alternatively, they follow from the first statement and cohomological wall crossing, i.e. \cite[Theorem C]{davison2016integrality}.
\end{proof}

\begin{corollary}
 For $\triangle\in\{\CN,\SemiN,\SSN\}$, the complexes of mixed Hodge modules $\iota_{\triangle}^!\ul{\IC}(\CM_{\Pi_Q,\dd})$ have pure cohomology, and vanishing odd cohomology.
\end{corollary}
In summary, we can generalise the results of this section for our chosen three Serre subcategories:
\begin{proposition}
\label{proposition:equivariantnilpotent}
For $\triangle\in\{\CN,\SemiN,\SSN\}$ there are isomorphisms of Lie algebras 
\[
\Fg_{\tilde{Q},\tilde{W}}^{\triangle,\psi,G}\cong \Fg_{\tilde{Q},\tilde{W}}^{\triangle,\psi}\otimes\HO^*(\B G,\BoQ)\cong \Fn^{\triangle,+}_{\Pi_Q}\otimes \HO^*(\B G,\BoQ),
\]
an isomorphism of algebras
\[
\ul{\mathrm{BPS}}_{\CA,\Alg}^{\triangle,\psi,G}\cong \ul{\mathrm{BPS}}_{\CA,\Alg}^{\triangle,\psi}\otimes\HO^*(\B G,\BoQ).
\]
as well as a PBW isomorphism (not respecting the algebra structures on the domain and the target)
\[
 \Sym_{\HO^*(\B G,\BoQ)}\left(\ul{\rmBPS}_{\CA,\Lie}^{\triangle,G}\otimes\HO^*_{\BoC^*}\right)\rightarrow \HO^*(\ulrelCoHA_{\CA}^{\triangle,\psi,G}).
 \]
\end{proposition}
Here, the algebra $\Fg_{\tilde{Q},\tilde{W}}^{\triangle,\psi,G}$ is defined as the derived global sections of $(\imath_{\triangle}^{G})^!\BPS_{\tilde{Q},\tilde{W}}^{\psi,G}$ where $\imath^G_{\triangle}\colon\CM_{\tilde{Q}}^{\triangle}/G\rightarrow \CM_{\tilde{Q}}/G$ is the inclusion of the saturated submonoid of $\CM_{\tilde{Q}}$ of representations of $\tilde{Q}$ such that the underlying representation of $\overline{Q}$ has the property $\triangle$.

\appendix
\section{Table of notations and conventions}
\subsection{Generalised Kac--Moody algebras and their representations}
\begin{enumerate}
 \item $\Fg_{\Gamma}$ is the Kac--Moody Lie algebra associated to a graph $\Gamma$ (without loops) \cite{kac1990infinite}.
 \item $\Fg_Q=\Fn_Q^+\oplus\mathfrak{h}\oplus\Fn_Q^+$, $\mathfrak{h}=\BoC^{Q_0}$, the Borcherds--Bozec Lie algebra of a quiver $Q$, with its triangular decomposition (page \pageref{subsubsection:borcherdsbozec}).
\end{enumerate}

\subsection{CoHAs and BPS (Lie) algebras}
\begin{enumerate}
 \item $\ulrelCoHA_{\CA}\in\CD^+(\MHM(\CM_{\CA}))$ the relative cohomological Hall algebra of $\CA$. This is by definition the complex of mixed Hodge modules $\JH_*\BD\BoQ_{\FM_{\CA}}^{\vir}$ together with the CoHA multiplication; $\SA_{\CA}=\rat(\ulrelCoHA_{\CA})\in\CD^+(\Perv(\CM_{\CA}))\cong\CD^+_{\rmc}(\CM_{\CA})$ (page \pageref{subsection:theCoHA}). $\ulrelCoHA_{\CA}^{\psi}$, $\SA_{\CA}^{\psi}$ the relative CoHAs with the $\psi$-twisted multiplication (page \pageref{subsubsection:twistmultiplication}).
 \item $\ulrelCoHA_{Q,W}\in\MMHM(\CM_{Q})$ the relative CoHA of the quiver with potential $(Q,W)$, defined as the complex of monodromic mixed Hodge modules $\JH_*\phip{\Tr(W)}\BoQ_{\FM_{Q}}^{\vir}$; $\SA_{Q,W}\in\Perv(\CM_{Q})$ the relative CoHA of the quiver with potential $(Q,W)$, defined as the complex of constructible sheaves $\JH_*\phip{\Tr(W)}\BoQ_{\FM_{Q}}^{\vir}$. Alternatively, $\SA_{Q,W}=\rat(\ulrelCoHA_{Q,W})$. $\ulrelCoHA_{Q,W}^{\psi}$, $\SA_{Q,W}^{\psi}$ the $\psi$-twisted versions (page \pageref{subsection:DMreminder}).
 \item $\ulBPS_{\CA,\Alg}\in\MHM(\CM_{\CA})$ the relative BPS algebra of the category $\CA$, defined as the degree zero cohomology of $\ulrelCoHA_{\CA}$, $\BPS_{\CA,\Alg}\in\Perv(\CM_{\CA})$ the relative BPS algebra, defined as the degree zero perverse cohomology of $\relCoHA_{\CA}$ or equivalently as $\rat(\ulBPS_{\CA,\Alg})$ (page \pageref{subsection:BPSalg}).
 \item $\ulBPS_{\CA,\Alg}^{\psi}$, $\BPS_{\CA,\Alg}^{\psi}$, the relative BPS algebras with the $\psi$-twisted multiplications (page \pageref{subsubsection:twistmultiplication}).
 \item $\underline{\rmBPS}_{\CA}=\HO^*(\ulBPS_{\CA,\Alg})\in\MHM(\K(\CA))$ the absolute BPS algebra, a $\K(\CA)$-graded complex of mixed Hodge structures. $\rmBPS_{\CA,\Alg}=\HO^*(\BPS_{\CA,\Alg})$, the absolute BPS algebra, a $\K(\CA)$-graded complex of vector spaces (i.e. a $\K(\CA)\times\BoZ$-graded vector space) (page \pageref{subsection:BPSalg}).
 \item $\underline{\rmBPS}_{\CA,\Alg}^{\psi}$, $\rmBPS_{\CA,\Alg}^{\psi}$ the algebras with the $\psi$-twisted multiplications (page \pageref{subsubsection:twistmultiplication}).
 \item $\Fn_{\CM_{\CA},(-,-)_{\SC},\underline{\SG}_{\CA}}^+=\underline{\BPS}_{\CA,\Lie}\in\MHM(\CM_{\CA})$ the relative BPS Lie algebra, defined as a positive part of a generalised Kac--Moody Lie algebra in the category of mixed Hodge modules, by generators and relations. $\Fn_{\CM_{\CA},(-,-)_{\SC},\SG_{\CA}}^+=\BPS_{\CA,\Lie}\in\Perv(\CM_{\CA})$ the relative BPS Lie algebra, defined as a positive part of a generalised Kac--Moody Lie algebra in the perverse sheaves, by generators and relations. Equivalently, $\BPS_{\CA,\Lie}=\rat(\underline{\BPS}_{\CA,\Alg})$ (page \pageref{def:BPSheafalg}).
 \item $\underline{\rmBPS}_{\CA,\Lie}=\HO^*(\ulBPS_{\CA,\Lie})$ the absolute BPS Lie algebra, a complex of mixed Hodge modules on $\K(\CA)$. By strict monoidality of $\HO^*$, it is isomorphic to $\Fn_{\K(\CA),(-,-)_{\SC},\HO^*(\underline{\SG}_{\CA})}^+$. We also denote it $\underline{\Fn}_{\CA}^{\ttBPS,+}$. $\rmBPS_{\CA,\Lie}=\HO^*(\BPS_{\CA,\Lie})$ the absolute BPS Lie algebra, a complex of vector spaces $\K(\CA)$. By strict monoidality of $\HO^*$, it is isomorphic to $\Fn_{\K(\CA),(-,-)_{\SC},\HO^*(\SG_{\CA})}^+$. Alternatively, it is obtained from $\underline{\rmBPS}_{\CA,\Lie}$ by forgetting the mixed Hodge structure. We also denote it $\Fn_{\CA}^{\ttBPS,+}$ (page \pageref{def:BPSheafalg}).
 \item $\Fg_{\CA}^{\ttBPS}=\Fn_{\CA}^{\ttBPS,-}\oplus\mathfrak{h}_{\CA}\oplus\Fn_{\CA}^{\ttBPS,+}$, $\mathfrak{h}_{\CA}=\widehat{\K(\CA)}\otimes_{\BoZ}\BoC$, the full BPS Lie algebra of $\CA$ (page \pageref{def:BPSheafalg}).
 \item $\ulBPS_{Q,W}\in\MMHM(\CM_{Q})$ the BPS sheaf of the quiver with potential $(Q,W)$, defined as $\tau^{\leq 0}(\ulrelCoHA_{Q,W}[1])$. $\BPS_{Q,W}\in\Perv(\CM_Q)$ the BPS sheaf of the quiver with potential $(Q,W)$, defined as $\ptau^{\leq 0}(\SA_{Q,W}[1])$ (page \pageref{subsection:DMreminder}).
 \item $\underline{\Fg}_{Q,W}=\HO^*(\ulBPS_{Q,W})$, the absolute BPS Lie algebra. It is a complex of $\BoN^{Q_0}$-graded mixed Hodge structures. Its Lie bracket is induced by the injective map $\underline{\Fg}_{Q,W}\rightarrow\HO^*\ulrelCoHA_{Q,W}^{\psi}$. $\Fg_{Q,W}=\HO^*(\ulBPS_{Q,W})$, the absolute BPS Lie algebra. It is a complex of $\BoN^{Q_0}$-graded vector spaces. Its Lie bracket is induced by the injective map $\Fg_{Q,W}\rightarrow\HO^*\SA_{Q,W}^{\psi}$ (page \pageref{theorem:3dBPS_def}).
 \item $\underline{\BPS}_{\Pi_Q,\Lie}^{\td}\in\MHM(\CM_{\Pi_Q})$ the BPS sheaf (without its Lie algebra structure) defined by dimensional reduction applied to the BPS mixed Hodge module of the tripled quiver with its canonical potential $(\tilde{Q},\tilde{W})$. $\BPS_{\Pi_Q,\Lie}^{\td}\in\Perv(\CM_{\Pi_Q})$ the BPS sheaf defined by dimensional reduction applied to the BPS perverse sheaf of the tripled quiver with its canonical potential $(\tilde{Q},\tilde{W})$, or equivalently, $\BPS_{\Pi_Q,\Lie}^{\td}=\rat(\underline{\BPS}_{\Pi_Q,\Lie}^{\td})$ (page \pageref{theorem:lessperverse}).
 \item $\underline{\BPS}_{\Pi_Q,\Lie}^{\td,\psi}\in\MHM(\CM_{\Pi_Q})$ the relative BPS Lie algebra. This is the mixed Hodge module $\underline{\BPS}_{\Pi_Q,\Lie}^{\td}$ on which the Lie algebra structure is induced by the monomorphism $\underline{\BPS}_{\Pi_Q,\Lie}^{\td}\rightarrow\ulBPS_{\Pi_Q,\Alg}^{\psi}$. $\BPS_{\Pi_Q,\Lie}^{\td,\psi}\in\Perv(\CM_{\Pi_Q})$ the relative BPS Lie algebra. This is the perverse sheaf $\underline{\BPS}_{\Pi_Q,\Lie}^{\td}$ on which the Lie algebra structure is induced by the monomorphism $\BPS_{\Pi_Q,\Lie}^{\td}\rightarrow\BPS_{\Pi_Q,\Alg}^{\psi}$ (page \pageref{theorem:lessperverse}).
\end{enumerate}
\subsection{Nilpotent versions}
Let $\imath_{\triangle}\colon\CM_{\CA}^{\triangle}\rightarrow\CM_{\CA}$ be the inclusion of a saturated locally closed submonoid. For example, $\ulrelCoHA_{\CA}^{\triangle}=\imath_{\triangle}^!\ulrelCoHA_{\CA}$, $\SA_{\CA}^{\triangle}=\imath_{\triangle}^!\SA_{\CA}$, $\ulBPS_{\CA,\Alg}^{\triangle}=\imath_{\triangle}^!\ulBPS_{\CA,\Alg}$, $\BPS_{\CA,\Alg}^{\triangle}=\imath_{\triangle}^!\BPS_{\CA,\Alg}$, and we have the $\psi$-twisted versions of these four algebra objects, $\ulrelCoHA_{\CA}^{\triangle,\psi}$, $\SA_{\CA}^{\triangle,\psi}$, $\ulBPS_{\CA,\Alg}^{\triangle,\psi}$ and $\BPS_{\CA,\Alg}^{\triangle,\psi}$ (page \pageref{subsection:theCoHA}).
\subsection{Equivariant versions}
We add a superscript $G$ to indicate the $G$-equivariant versions of the objects considered (page \pageref{section:Equivariantparameters}).

\subsection{Nakajima quiver varieties and their (BPS) cohomology}
\begin{enumerate}
\item $N(Q,\dd,\ff)$ the smooth Nakajima quiver variety constructed from the doubled quiver $\overline{Q}$, with dimension vector $\dd$ and framing dimension vector $\ff$ (page \pageref{equation:smoothNQV}).
 \item $N'(Q,\dd,\ff)$ the singular Nakajima quiver variety, constructed by GIT quotient, from the doubled quiver $\overline{Q}$ with dimension vector $\dd$ and framing dimension vector $\ff$ (page \pageref{equation:singNQV}).
 $N^{\triangle}(Q,\dd,\ff)\subset N(Q,\dd,\ff)$ is the subscheme of stable representations of $\Pi_{Q_{\ff}}$ satisfying the condition $\triangle$. Similarly, we have $N'^{\triangle}(Q,\dd,\ff)\subset N'(Q,\dd,\ff)$ (page \pageref{NQVs_def}).
 \item $\mathbb{M}_{\ff}^{\triangle}(Q)=\bigoplus_{\dd\in\BoN^{Q_0}}\HO^*(N^{\triangle}(Q,\dd,\ff),\imath_{\triangle}^!\BoQ^{\vir}_{N(Q,\dd,\ff)})$ the direct sum of the Borel--Moore homologies of all Nakajima quiver varieties with fixed framing dimension vector (page \pageref{equation:moduleNQV}).
 \item $\mathbb{M}_{\ff}^{\triangle,\psi}(Q)$, the vector space $\mathbb{M}_{\ff}^{\triangle}(Q)$ with its structure of $\Fg_{\Pi_Q}^{\ttBPS,\triangle}$-module or $\Fn_{\Pi_Q}^{\ttBPS,\triangle,+}$-module (page \pageref{equation:moduleNQV}).
\end{enumerate}

\printbibliography
\end{document}